\documentclass[letterpaper]{article}
\usepackage[margin=1.5in]{geometry}
\usepackage{ifthen}

\usepackage{amsmath,amsfonts,mathrsfs,amsthm,thmtools}
\usepackage{amssymb}

\usepackage[unicode,bookmarks]{hyperref} 
\usepackage{cleveref}
\usepackage[usenames,dvipsnames]{xcolor}
\hypersetup{colorlinks=true,citecolor=NavyBlue,linkcolor=BrickRed,urlcolor=Black}

\usepackage{expl3}

\ExplSyntaxOn
\prop_new:N \g_cite_map_prop
\tl_new:N \l_citekey_result_tl

\cs_new:Npn \mapcitekey #1#2 {
  \clist_map_inline:nn {#2}
       {  \prop_gput:Nnn  \g_cite_map_prop  {##1} {#1}   }
}

\cs_new:Npn \getcitekey #1 {
   \prop_get:NoN \g_cite_map_prop{#1}  \l_citekey_result_tl
   \quark_if_no_value:NF \l_citekey_result_tl
       {  \tl_set_eq:NN #1  \l_citekey_result_tl  }
}

\cs_new:Npn \showcitekeymaps {\prop_show:N  \g_cite_map_prop }
\ExplSyntaxOff

\usepackage{etoolbox}
\makeatletter
\patchcmd{\@citex}{\if@filesw}{\getcitekey\@citeb \if@filesw}%
    {\typeout{*** SUCCESS ***}}{\typeout{*** FAIL ***}}
\patchcmd{\nocite}{\if@filesw}{\getcitekey\@citeb \if@filesw}%
    {\typeout{*** SUCCESS ***}}{\typeout{*** FAIL ***}}
\makeatother

\usepackage{enumitem}

\newenvironment{aenumerate}{%
	\begin{enumerate}[label=(\alph{*}), ref=(\alph{*})]
}{%
	\end{enumerate}%
}

\usepackage{chngcntr}

\usepackage{tikz}
\usepackage{tikz-cd}
\tikzset{commutative diagrams/arrow style=Latin Modern}

\newcommand{\Dmod}{\mathscr{D}}

\newcommand{\shT}{\mathscr{T}}

\newcommand{\derR}{\mathbf{R}}

\newcommand{\decal}[1]{\lbrack #1 \rbrack}

\newcommand{\shH}{\mathcal{H}}

\newcommand{\abs}[1]{\lvert #1 \rvert}

\newcommand{\eps}{\varepsilon}

\newcommand{\tensor}{\otimes}

\newcommand{\dbar}{\bar{\partial}}

\newcommand{\dz}{\mathit{dz}}
\newcommand{\dzb}{d\bar{z}}

\newcommand{\shHom}{\mathcal{H}\hspace{-1pt}\mathit{om}}

\newcommand{\NN}{\mathbb{N}}
\newcommand{\ZZ}{\mathbb{Z}}
\newcommand{\QQ}{\mathbb{Q}}

\newcommand{\CC}{\mathbb{C}}

\newcommand{\PP}{\mathbb{P}}

\DeclareMathOperator{\coker}{coker}
\DeclareMathOperator{\ad}{ad}

\DeclareMathOperator{\sgn}{sgn}

\DeclareMathOperator{\id}{id}

\DeclareMathOperator{\Supp}{Supp}
\DeclareMathOperator{\codim}{codim}

\DeclareMathOperator{\Cone}{Cone}
\DeclareMathOperator{\Sym}{Sym}
\DeclareMathOperator{\gr}{gr}
\DeclareMathOperator{\DR}{DR}

\DeclareMathOperator{\Hom}{Hom}

\DeclareMathOperator{\SL}{SL}

\newcommand{\sltwo}{\mathfrak{sl}_2(\CC)}

\newcommand{\shf}[1]{\mathscr{#1}}

\newcommand{\shV}{\shf{V}}

\newcommand{\argbl}{-}

\def\overbar#1#2#3{{%
	\setbox0=\hbox{\displaystyle{#1}}%
	\dimen0=\wd0
	\advance\dimen0 by -#2 
	\vbox {\nointerlineskip \moveright #3 \vbox{\hrule height 0.3pt width \dimen0}%
		\nointerlineskip \vskip 1.5pt \box0}%
}}

\newcommand{\pil}{\pi_{\ast}}

\newcommand{\iu}{i^{\ast}}

\newcommand{\tl}{t_{\ast}}

\newcommand{\piu}{\pi^{\ast}}

\newcommand{\shE}{\shf{E}}

\newcommand{\shO}{\shf{O}}

\counterwithin{equation}{paragraph}
\counterwithin{subsection}{section}
\counterwithin{figure}{section}

\newtheorem{thm}{Theorem}[section]
\newtheorem*{thm*}{Theorem}

\newtheorem*{lem*}{Lemma}
\newtheorem{prop}[thm]{Proposition}
\newtheorem*{prop*}{Proposition}

\newtheorem*{cor*}{Corollary}

\declaretheoremstyle[numbered=yes,headformat=\NAME\NOTE,numberwithin=paragraph,%
	bodyfont=\normalfont\itshape,postheadspace={ },%
	spaceabove=\topsep,spacebelow=\topsep]{par-thm}
\declaretheoremstyle[numbered=yes,headformat=\NAME\NOTE,numberwithin=paragraph,%
bodyfont=\normalfont,postheadspace={ },%
	spaceabove=\topsep,spacebelow=\topsep]{par-def}
\declaretheoremstyle[numbered=yes,headformat=\NAME\NOTE,numberwithin=paragraph,%
	headfont=\normalfont\itshape,bodyfont=\normalfont,postheadspace={ },
	spaceabove=\topsep,spacebelow=\topsep]{par-exa}
\declaretheorem[name=Theorem,style=par-thm,preheadhook={},%
	postheadhook={\leavevmode}]{pthm} 
\declaretheorem[name=Lemma,style=par-thm,preheadhook={},%
	postheadhook={\leavevmode}]{plem} 
\declaretheorem[name=Corollary,style=par-thm,preheadhook={},%
	postheadhook={\leavevmode}]{pcor} 
\declaretheorem[name=Proposition,style=par-thm,preheadhook={},%
	postheadhook={\leavevmode}]{pprop} 
\declaretheorem[name=Conjecture,style=par-thm,preheadhook={},%
	postheadhook={\leavevmode}]{pconj} 
 
\declaretheorem[name=Example,style=par-exa,preheadhook={},%
	postheadhook={\leavevmode}]{pexa} 

\theoremstyle{definition}

\newtheorem*{dfn*}{Definition}

\newtheorem*{conj*}{Conjecture}

\theoremstyle{remark}

\newtheorem*{exa*}{Example}
\newtheorem*{problem*}{Problem}

\theoremstyle{plain}

\counterwithout{paragraph}{subsubsection}

\newcommand{\newpar}{\paragraph{\hspace{-1em}.}}
\makeatletter 
\renewcommand\paragraph{\@startsection{paragraph}{4}{\z@}%
	{1.25ex \@plus.2ex \@minus.2ex}{-0.5em}%
	{\normalfont\normalsize\bfseries}} 
\makeatother

\crefformat{paragraph}{#2\S#1#3}
\crefformat{section}{#2chapter~#1#3}
\crefformat{subsection}{#2section~#1#3}
\crefformat{pthm}{#2\S#1\ Theorem#3}
\crefformat{pconj}{#2\S#1\ Conjecture#3}
\crefformat{plem}{#2\S#1\ Lemma#3}
\crefformat{pcor}{#2\S#1\ Corollary#3}
\crefformat{pprop}{#2\S#1\ Proposition#3}
\crefformat{pdfn}{#2\S#1\ Definition#3}
\crefformat{pexa}{#2\S#1\ Example#3}

\Crefformat{paragraph}{#2\S#1#3}
\Crefformat{section}{#2Chapter~#1#3}
\Crefformat{subsection}{#2Section~#1#3}
\Crefformat{pthm}{#2\S#1\ Theorem#3}
\Crefformat{pconj}{#2\S#1\ Conjecture#3}
\Crefformat{plem}{#2\S#1\ Lemma#3}
\Crefformat{pcor}{#2\S#1\ Corollary#3}
\Crefformat{pprop}{#2\S#1\ Proposition#3}
\Crefformat{pdfn}{#2\S#1\ Definition#3}
\Crefformat{pexa}{#2\S#1\ Example#3}

\makeatletter
\let\old@caption\caption
\renewcommand*{\caption}[1]{%
	\setcounter{figure}{\value{equation}}%
	\stepcounter{equation}%
	\old@caption{#1}\relax%
}
\makeatother

\newcounter{intro}

\newtheorem{intro-conjecture}[intro]{Conjecture}
\newtheorem{intro-corollary}[intro]{Corollary}
\newtheorem{intro-theorem}[intro]{Theorem}

\newcommand{\parref}[1]{\hyperref[#1]{\S\ref*{#1}}}
\newcommand{\chapref}[1]{\hyperref[#1]{Chapter~\ref*{#1}}}

\makeatletter
\newcommand*\if@single[3]{%
  \setbox0\hbox{${\mathaccent"0362{#1}}^H$}%
  \setbox2\hbox{${\mathaccent"0362{\kern0pt#1}}^H$}%
  \ifdim\ht0=\ht2 #3\else #2\fi
  }
\newcommand*\rel@kern[1]{\kern#1\dimexpr\macc@kerna}
\newcommand*\widebar[1]{\@ifnextchar^{{\wide@bar{#1}{0}}}{\wide@bar{#1}{1}}}
\newcommand*\wide@bar[2]{\if@single{#1}{\wide@bar@{#1}{#2}{1}}{\wide@bar@{#1}{#2}{2}}}
\newcommand*\wide@bar@[3]{%
  \begingroup
  \def\mathaccent##1##2{%
    \if#32 \let\macc@nucleus\first@char \fi
    \setbox\z@\hbox{$\macc@style{\macc@nucleus}_{}$}%
    \setbox\tw@\hbox{$\macc@style{\macc@nucleus}{}_{}$}%
    \dimen@\wd\tw@
    \advance\dimen@-\wd\z@
    \divide\dimen@ 3
    \@tempdima\wd\tw@
    \advance\@tempdima-\scriptspace
    \divide\@tempdima 10
    \advance\dimen@-\@tempdima
    \ifdim\dimen@>\z@ \dimen@0pt\fi
    \rel@kern{0.6}\kern-\dimen@
    \if#31
      \overline{\rel@kern{-0.6}\kern\dimen@\macc@nucleus\rel@kern{0.4}\kern\dimen@}%
      \advance\dimen@0.4\dimexpr\macc@kerna
      \let\final@kern#2%
      \ifdim\dimen@<\z@ \let\final@kern1\fi
      \if\final@kern1 \kern-\dimen@\fi
    \else
      \overline{\rel@kern{-0.6}\kern\dimen@#1}%
    \fi
  }%
  \macc@depth\@ne
  \let\math@bgroup\@empty \let\math@egroup\macc@set@skewchar
  \mathsurround\z@ \frozen@everymath{\mathgroup\macc@group\relax}%
  \macc@set@skewchar\relax
  \let\mathaccentV\macc@nested@a
  \if#31
    \macc@nested@a\relax111{#1}%
  \else
    \def\gobble@till@marker##1\endmarker{}%
    \futurelet\first@char\gobble@till@marker#1\endmarker
    \ifcat\noexpand\first@char A\else
      \def\first@char{}%
    \fi
    \macc@nested@a\relax111{\first@char}%
  \fi
  \endgroup
}
\makeatother

\mapcitekey{Maulik+Shen+Yin:FourierMukaiTransforms}{MSY}
\mapcitekey{Shen+Yin:FourierMukaiTransforms}{SY}
\mapcitekey{Eisenbud+Floystad+Schreyer:SheafCohomology}{EFS}
\mapcitekey{Matsushita:HigherDirectImages}{Mat-img}
\mapcitekey{Matsushita:Equidimensionality}{Mat-equi}
\mapcitekey{Bernstein+Gelfand+Gelfand:BGG}{BGG}
\mapcitekey{McCleary:UsersGuide}{McCleary}
\mapcitekey{Boardman:SpectralSequences}{Boardman}
\mapcitekey{Mochizuki:L2-complexes}{Mochizuki}
\mapcitekey{Saito:HodgeModules}{HM}
\mapcitekey{Saito:MixedHodgeModules}{MHM}
\mapcitekey{Deligne:Decompositions}{Deligne}
\mapcitekey{Saito:Introduction}{Saito-intro}
\mapcitekey{Saito:HodgeFiltration}{Saito-HF}
\mapcitekey{Kollar:DualizingI}{Kollar}
\mapcitekey{Kollar:DualizingII}{KollarII}
\mapcitekey{Campana:LocalProjectivity}{Campana}
\mapcitekey{Huybrechts+Mauri:LagrangianFibrations}{LF}
\mapcitekey{Ngo:PerverseSheaves}{Ngo}
\mapcitekey{Golovin:GlobalDimension}{Golovin}
\mapcitekey{Shen+Yin:Topology}{SY-top}
\mapcitekey{Hwang:BaseManifolds}{Hwang}
\mapcitekey{Ngo:LemmeFondamental}{Ngo-lemme}
\mapcitekey{Looijenga+Lunts:LieAlgebra}{LL}
\mapcitekey{Verbitsky:CohomologyApplications}{Verbitsky}
\mapcitekey{Cataldo+Hausel+Migliorini:HitchinSystems}{CHM}
\mapcitekey{Maulik+Shen:P=W}{MS}
\mapcitekey{Hausel+Mellit+Minets+Schiffmann:P=W}{HMMS}
\mapcitekey{Cataldo+Rapagnetta+Sacca:OGrady10}{CRS}

\newcommand{\half}{\frac{1}{2}}
\newcommand{\Pmod}{\mathcal{P}}

\newcommand{\dt}{\mathit{dt}}
\newcommand{\dtj}{\dt_j}
\newcommand{\dtJ}{\dt_J}
\newcommand{\Mt}{\tilde{M}}
\newcommand{\Hsl}{\mathsf{H}}
\newcommand{\Xsl}{\mathsf{X}}
\newcommand{\Ysl}{\mathsf{Y}}
\newcommand{\wsl}{\mathsf{w}}
\newcommand{\esl}{\mathsf{e}}
\newcommand{\fsl}{\mathsf{f}}
\newcommand{\hsl}{\mathsf{h}}

\renewcommand{\shV}{\mathcal{V}}
\newcommand{\cont}{\operatorname{\lrcorner}}
\newcommand{\algS}{\mathcal{S}_B}
\newcommand{\algE}{\Omega_B}
\newcommand{\BGGL}{\mathbf{L}_B}
\newcommand{\BGGR}{\mathbf{R}_B}
\newcommand{\Dbcoh}{D_{\mathit{coh}}^\mathit{b}}
\newcommand{\wed}{\mathop{\mathchoice{\textstyle\bigwedge}{\bigwedge}%
	{\bigwedge}{\bigwedge}}\nolimits}
\newcommand{\Mh}{\widehat{M}}
\newcommand{\Gh}{\widehat{G}}
\newcommand{\slthree}{\operatorname{\mathfrak{sl}}_3(\CC)}
\newcommand{\slfour}{\operatorname{\mathfrak{sl}}_4(\CC)}
\newcommand{\sosix}{\operatorname{\mathfrak{so}}_6(\CC)}
\newcommand{\shA}{\mathcal{A}}
\newcommand{\DB}{\mathbf{D}_B}
\newcommand{\wsigma}{\wsl_{\sigma}}
\newcommand{\Ysigma}{\Ysl_{\sigma}}
\newcommand{\womega}{\wsl_{\omega}}
\newcommand{\womegat}{\wsl_{\omega_2}}
\newcommand{\Cpi}{\mathcal{C}_{\pi}}
\newcommand{\dpi}{d_{\pi}}
\DeclareMathOperator{\Ad}{Ad}
\newcommand{\iomega}{\Theta}

\begin{document}

\title{Hodge theory and Lagrangian fibrations on holomorphic symplectic manifolds}

\author{%
 Christian Schnell \\
 Department of Mathematics, Stony Brook University \\
 \texttt{christian.schnell@stonybrook.edu}
}

\date{\today}
\maketitle

\section{Introduction}

\newpar 
The purpose of this paper is to establish several new results about the Hodge theory of
Lagrangian fibrations on (not necessarily compact) holomorphic symplectic manifolds.
In particular, we prove two beautiful recent conjectures by Maulik, Shen and Yin;
and we show, without using hyperk\"ahler metrics, that every Lagrangian
fibration gives rise to an action by the Lie algebra
$\slthree$ (in the noncompact case) or $\slfour$ (in the compact case).

\newpar
The most interesting Lagrangian fibration is arguably the Hitchin fibration on the
moduli space of stable Higgs bundles on a smooth projective curve of genus $g \geq
2$. It was famously used by Ng{\^o} in his proof of the fundamental lemma \cite{Ngo-lemme},
and is also the central object in the $P=W$ conjecture by de Cataldo,
Hausel, and Migliorini \cite{CHM}, recently proved by Maulik and Shen \cite{MS}
and Hausel, Mellit, Minets, and Schiffmann \cite{HMMS}. But Lagrangian fibrations are
also very useful for studying compact hyperk\"ahler manifolds, which are compact
holomorphic symplectic manifolds with a hyperk\"ahler metric. For example, de
Cataldo, Rapagnetta, and Sacc\`a \cite{CRS} used a pair of Lagrangian fibrations to
compute the Hodge numbers of O'Grady's $10$-dimensional sporadic compact
hyperk\"ahler manifold. 

\newpar
Let $M$ be a holomorphic symplectic manifold of dimension $2n$ that is K\"ahler but
not necessarily compact, and let $\pi \colon M \to B$ be a Lagrangian fibration over
a complex manifold $B$ of dimension $n$. The general fiber of $\pi$ is an $n$-dimensional
abelian variety, but very little is known about the singular fibers. The method we
use in this paper is to apply the decomposition theorem; this produces certain
perverse sheaves on the base of the Lagrangian fibration. In some cases, such as
Ng\^o's support theorem \cite{Ngo}, these perverse sheaves are controlled by what
happens on the smooth locus, but in general, their behavior is a mystery. This
mystery is the subject of the conjectures by Maulik, Shen, and Yin \cite{SY,MSY}.

\newpar
Our main result is that there is a very close relationship between two seemingly unrelated
objects: the $k$-th perverse sheaf $P_k$ in the decomposition theorem for $\pi$, and
the derived direct image $\derR \pil \Omega_M^{n+k}$ of the sheaf of holomorphic
$(n+k)$-forms on $M$ (see \Cref{thm:isomorphism}). This is formulated -- and proved
-- with the help of Saito's theory of Hodge modules, and the BGG correspondence
(between graded modules over the symmetric and exterior algebras).
On both sides, we need to take the associated graded with respect to a certain
filtration: in the case of $P_k$, this is the Hodge filtration of $P_k$, viewed as a
Hodge module; in the case of $\derR \pil \Omega_M^{n+k}$, it is the perverse
filtration coming from the decomposition theorem.

\newpar
Along the way, we prove a relative Hard Lefschetz theorem for the action of the
holomorphic symplectic form (in \Cref{thm:HL-sigma}), as well as the symmetry conjecture
$G_{i,k} \cong G_{k,i}$ of Shen and Yin for the complexes $G_{i,k} = \gr_{-k}^F
\DR(\Pmod_i) \decal{-i}$ (in \Cref{conj:symmetry}). The structure that one gets on
the direct sum of all the complexes $G_{i,k}$ looks somewhat like the ``Hodge
diamond'' of a compact hyperk\"ahler manifold, except that it has a hexagonal shape (see
\Cref{par:hexagon}) and comes with an action by the Lie algebra $\slthree$ (see
\Cref{par:slthree}). One interesting aspect is that all of these structures are only
visible in the derived category. Perhaps the most useful feature of the present work
is that no restrictions on the singular fibers are needed: all the results
below apply for example to the entire Hitchin fibration on the moduli space of stable
Higgs bundles (provided that the rank and the degree are coprime).

\newpar
One application of our main result is a different proof for the ``numerical perverse =
Hodge'' symmetry for irreducible compact hyperk\"ahler manifolds \cite{SY-top} that
does not rely on the existence of a hyperk\"ahler metric (see \Cref{thm:Shen-Yin}).
Another application is that the Lie algebra $\slfour \cong \sosix$ acts on the
cohomology of a compact holomorphic symplectic manifold with a Lagrangian fibration
(see \Cref{chap:slfour});
this generalizes a result by Looijenga-Lunts \cite[\S4]{LL} and Verbitsky
\cite{Verbitsky}, who proved this for irreducible compact hyperk\"ahler manifolds.
Once again, our proof does not rely on the existence of a hyperk\"ahler metric.

\subsection{Lagrangian fibrations}

\newpar
We now give a more detailed summary of the paper. Let $M$ be a holomorphic symplectic
manifold of dimension $2n$; we assume that $M$ is K\"ahler, but we allow $M$ to be
noncompact. We denote by $\sigma \in H^0(M, \Omega_M^2)$ the holomorphic symplectic
form; $\sigma$ is nondegenerate, which means that it induces an isomorphism between the
holomorphic tangent sheaf $\shT_B$ and the sheaf of holomorphic $1$-forms
$\Omega_B^1$. We need to add the assumption that $d \sigma = 0$; this would of course
be automatic in the compact case (by Hodge theory). We also fix a K\"ahler form
$\omega \in A^{1,1}(M)$; recall that $\omega$ is real and positive, and satisfies $d
\omega = 0$. In the special case where $M$ is a compact hyperk\"ahler manifold, this
might be the K\"ahler form of a hyperk\"ahler metric; but in the noncompact case, any
K\"ahler metric seems to be as good as any other.

\newpar
Further, let $\pi \colon M \to B$ be a Lagrangian fibration on $M$. This means that
$B$ is a complex manifold of dimension $n$, and that $\pi$ is a proper surjective
holomorphic mapping whose smooth fibers are Lagrangian, in the sense that $\sigma$
restricts to zero on every smooth fiber of $\pi$. For a very nice introduction to the
general theory of Lagrangian fibrations, see the recent paper by Huybrechts and Mauri
\cite{LF}. It is known that the smooth fibers are abelian varieties of dimension $n$.
Moreover, according to a theorem by Matsushita \cite[Thm.~1]{Mat-equi}, all fibers of
$\pi$ have dimension $n$, and $\sigma$ pulls back to zero on a resolution of
singularities (of the reduction) of every fiber. Let me emphasize again that $B$ is
assumed to be a complex manifold; in the special case where $M$ is an irreducible compact
hyperk\"ahler manifold, this implies that $B$ is isomorphic to $\PP^n$ \cite{Hwang}.  

\newpar \label{par:analogy}
The starting point for the work by Shen and Yin \cite{SY} is the following curious
analogy between two completely different sets of objects. On the one hand, we have the
sheaves of holomorphic forms $\Omega_M^{n+k}$ for $k = -n, \dotsc, n$. The rank of 
$\Omega_M^{n+k} \cong \wed^{n+k} \Omega_M^1$ is of course $\binom{2n}{n+k}$. In
addition, wedge product with the symplectic form induces an isomorphism
\[
	\sigma^k \colon \Omega_M^{n-k} \to \Omega_M^{n+k}.
\]
On the other
hand, we have a collection of perverse sheaves $P_i$ on the base $B$ of the Lagragian
fibration, where $P_i$ is defined as the $i$-th perverse cohomology sheaf of the
complex $\derR \pil \QQ_M \decal{2n}$. The fact that all fibers of $\pi$ have
dimension $n$ implies that $P_i$ is nonzero only for $i = -n, \dotsc, n$. The
restriction of $P_i$ to the smooth locus of $\pi$ is just the local system on the
$(n+i)$-th cohomology of the fibers, which are $n$-dimensional abelian varieties.
As it happens, $H^{n+i}(A, \CC) \cong \wed^{n+i} H^1(A, \CC)$, which means that the
generic rank of $P_i$ is also $\binom{2n}{n+i}$. The most striking part of the
analogy is that wedge product with the K\"ahler form induces isomorphisms
\[
	\omega^i \colon P_{-i} \to P_i.
\]
This result, called the relative Hard Lefschetz theorem, used to be 
known only for projective morphisms, but recent work by Mochizuki \cite{Mochizuki}
shows that it also holds for proper holomorphic mappings from K\"ahler
manifolds.\footnote{Locally on the base, Lagrangian fibrations are actually
	projective \cite{Campana}.}
The same is true for the decomposition theorem, which guarantees the existence of a
decomposition
\[
	\derR \pil \QQ_M \decal{2n} \cong \bigoplus_{i=-n}^n P_i \decal{-i}
\]
in the derived category. Since we have chosen a K\"ahler form, there is a preferred
choice of decomposition, constructed by Deligne; this is described in
\Cref{par:Deligne}.

\subsection{The symmetry conjecture of Shen and Yin}

\newpar
Are the two objects $\Omega_M^{n+k}$ and $P_k$ actually related in some way?
In \cite{SY}, Shen and Yin proposed a conjectural symmetry that would relate at least
certain complexes of coherent sheaves derived from the two objects. Their conjecture
is formulated using Saito's theory of Hodge modules \cite{HM,MHM}. Recall that $\QQ_M
\decal{2n}$ is actually a Hodge module of weight $2n$. We can adjust the weight by a
Tate twist, making $\QQ_M(n) \decal{2n}$ a Hodge module of weight $0$. Saito's
version of the decomposition theorem
\[
	\derR \pil \QQ_M(n) \decal{2n} \cong \bigoplus_{i=-n}^n P_i \decal{-i}
\]
then shows that each $P_i$ is a Hodge module of weight $i$ on $B$. We denote
by $\Pmod_i$ the underlying (regular holonomic) right $\Dmod_B$-module, and by
$F_{\bullet} \Pmod_i$ its Hodge filtration, which is an increasing filtration by
coherent $\shO_B$-modules. The perverse sheaf $P_i$ and the $\Dmod_B$-module
$\Pmod_i$ are related by the Riemann-Hilbert correspondence:
\[
	P_i \tensor_{\QQ} \CC \cong \DR(\Pmod_i)
\]
Here $\DR(\Pmod_i)$ is the de Rham complex; since $\Pmod_i$ is a right
$\Dmod_B$-module, the de Rham complex (which is usually called the ``Spencer
complex'' in the $\Dmod$-module literature) is the complex
\[
	\DR(\Pmod_i) = \Bigl\lbrack 
	\Pmod_i \tensor \wed^n \shT_B \to \dotsb \to 
		\Pmod_i \tensor \shT_B \to \Pmod_i
	\Bigr\rbrack.
\]
It lives in cohomological degrees $-n, \dotsc, 0$, and the differential is
induced by the multiplication map $\Pmod_i \tensor \shT_B \to \Pmod_i$. The de Rham
complex is filtered by the subcomplexes
\[
	F_k \DR(\Pmod_i) = \Bigl\lbrack 
	F_{k-n} \Pmod_i \tensor \wed^n \shT_B \to \dotsb \to 
		F_{k-1} \Pmod_i \tensor \shT_B \to F_k \Pmod_i 
	\Bigr\rbrack,
\]
and the graded pieces of this filtration give us several complexes of coherent
$\shO_B$-modules
\[
	\gr_k^F \DR(\Pmod_i) = \Bigl\lbrack 
	\gr_{k-n}^F \Pmod_i \tensor \wed^n \shT_B \to \dotsb \to 
		\gr_{k-1}^F \Pmod_i \tensor \shT_B \to \gr_k^F \Pmod_i 
	\Bigr\rbrack.
\]
One consequence of Saito's theory is that one has an isomorphism (in the derived
category)
\begin{equation} \label{eq:Saito-Omega}
	\derR \pil \Omega_M^{n+k} \decal{n-k} \cong \bigoplus_{i=-n}^n \gr_{-k}^F
	\DR(\Pmod_i) \decal{-i};
\end{equation}
here the decomposition is induced by the one in the decomposition theorem.

\newpar
In \cite{SY}, Shen and Yin introduced what they call the ``perverse-Hodge complexes''
\[
	G_{i,k} = \gr_{-k}^F \DR(\Pmod_i) \decal{-i},
\]
although with a different choice of indexing.\footnote{In the notation of
\cite[0.2]{SY}, one has $\mathcal{G}_{i,k} = G_{i-n,k-n}$.} 
These are complexes of coherent $\shO_B$-modules on the base manifold $B$ of the
Lagrangian fibration. The two indices $i$ and $k$ have the following meaning:
\begin{enumerate}
	\item The first index $i$ records the cohomological degree, in the sense that
	$G_{i,k}$ is associated (over the smooth locus of $\pi$) with the $(n+i)$-th
	cohomology groups of the fibers.
	\item The second index $k$ records the holomorphic degree, in the sense that $G_{i,k}$
	is associated with the sheaf $\Omega_M^{n+k}$ of holomorphic forms of degre
	$(n+k)$.
\end{enumerate}

\newpar
Over the smooth locus of $\pi$, the complexes $G_{i,k}$ can be described fairly
explicitly using the Hodge bundles $\shV^{p,q}$ in the variation of Hodge structure:
in fact, the restriction of $G_{i,k}$ to the smooth locus is the complex
\[
	\Bigl\lbrack
		\Omega_B^k \tensor \shV^{n,i} \to \Omega_B^{k+1} \tensor \shV^{n-1,i+1} 
		\to \dotsb \to \Omega_B^n \tensor \shV^{k, n+i-k} 
	\Bigr\rbrack.
\]
Note again that the cohomological degree of each term is $n+i$, whereas the
holomorphic degree is $n+k$. But the behavior of these complexes on the singular
locus of $\pi$ is quite
mysterious. Shen and Yin proved that the two complexes $G_{i,k}$ and $G_{k,i}$ are 
isomorphic over the smooth locus of $\pi$, and motivated by this, they made the
following bold conjecture.

\begin{pconj} \label{conj:symmetry}
	In the derived category of coherent $\shO_B$-modules, one has $G_{i,k} \cong G_{k,i}$.
\end{pconj}

Note that the complexes $G_{i,k}$ are in general \emph{not} determined by their
restriction to the smooth locus of $\pi$; there are also examples where $G_{i,k}$
and $G_{k,i}$ are isomorphic in the derived category, but not isomorphic as
complexes \cite[\S2.4]{SY}. In any case, if the conjecture by Shen and Yin is true,
then one can combine it with Saito's isomorphism \eqref{eq:Saito-Omega} to get
\[
	\derR \pil \Omega_M^{n+k} \decal{n} \cong \bigoplus_{i=-n}^n G_{i,k} \decal{k}
	\cong \bigoplus_{i=-n}^n G_{k,i} \decal{k}
	= \bigoplus_{i=-n}^n \gr_{-i}^F \DR(\Pmod_k),
\]
which relates $\Omega_M^{n+k}$ and $P_k$ at least in a sort of indirect way. 

\newpar
The original motivation for \Cref{conj:symmetry} is the ``numerical perverse =
Hodge'' symmetry for compact hyperk\"ahler manifolds in \cite[Thm.~0.2]{SY-top}. Its
proof by Shen and Yin makes heavy use of the hyperk\"ahler metric, and the conjecture
arose in at attempt to find a more ``local'' explanation for
this symmetry, and to extend it to Lagrangian fibrations on noncompact holomorphic
symplectic manifolds. In terms of the complexes $G_{i,k}$, the main result in
\cite{SY-top} is the statement that
\[
	H^j(B, G_{i,k}) \cong H^j(B, G_{k,i})
\]
for all $i,j,k \in \ZZ$, provided that $M$ is an irreducible compact hyperk\"ahler
manifold. In my opinion, this was the most convincing piece of evidence for the conjecture.

\newpar
Our first result explains the symmetry between $G_{i,k}$ and $G_{k,i}$ as coming from
a new relative Hard Lefschetz theorem for the symplectic form (which, unlike the 
relative Hard Lefschetz theorem for the K\"ahler form, only holds in the derived
category). The decomposition theorem gives us a decomposition
\[
	\pi_{+} (\omega_M, F_{\bullet} \omega_M) \cong
	\bigoplus_{i=-n}^n (\Pmod_i, F_{\bullet} \Pmod_i) \decal{-i}
\]
for the direct image of the $\Dmod$-module $\omega_M$ (again with the filtration for
which $\gr_{-n}^F \omega_M = \omega_M$), in the derived category of filtered
$\Dmod_B$-modules. Since the K\"ahler form $\omega$ is closed and of type $(1,1)$, it
gives rise to a morphism 
\[
	\omega \colon \bigoplus_{i=-n}^n (\Pmod_i, F_{\bullet} \Pmod_i) \decal{-i}
	\to \bigoplus_{i=-n}^n (\Pmod_i, F_{\bullet-1} \Pmod_i) \decal{2-i}.
\]
Let us denote by $\omega_j \colon (\Pmod_i, F_{\bullet} \Pmod_i) \to (\Pmod_{i+j},
F_{\bullet-1} \Pmod_{i+j}) \decal{2-j}$ the individual components of $\omega$ with
respect to this decomposition. The topmost component $\omega_2$ accounts for the
action of the K\"ahler form on the cohomology of the fibers of $\pi$. In these terms,
the relative Hard Lefschetz theorem (for proper holomorphic mappings from K\"ahler
manifolds) is saying that
\[
	\omega_2^i \colon (\Pmod_{-i}, F_{\bullet} \Pmod_{-i})
	\to (\Pmod_i, F_{\bullet-i} \Pmod_i)
\]
is an isomorphism for every $i \geq 1$. From $\omega_2$, we also get a morphism of
complexes
\[
	\omega_2 \colon G_{i,k} \to G_{i+2,k+1} \decal{2},
\]
and as a consequence of the relative Hard Lefschetz theorem, the induced morphism
\[
	\omega_2^i \colon G_{-i,k} \to G_{i,i+k} \decal{2i}
\]
is an isomorphism for every $i \geq 1$ and every $k \in \ZZ$. 

\newpar
The symplectic form $\sigma$ is closed and of type $(2,0)$, and so it gives rise to a
morphism
\[
	\sigma \colon \bigoplus_{i=-n}^n (\Pmod_i, F_{\bullet} \Pmod_i) \decal{-i}
	\to \bigoplus_{i=-n}^n (\Pmod_i, F_{\bullet-2} \Pmod_i) \decal{2-i},
\]
again in the derived category of filtered $\Dmod_B$-modules. As before, we denote the
components of $\sigma$ with respect to this decomposition by $\sigma_j \colon
(\Pmod_i, F_{\bullet} \Pmod_i) \to (\Pmod_{i+j}, F_{\bullet-2} \Pmod_{i+j})
\decal{2-j}$. This time, we get $\sigma_2 = 0$ because $\pi$ is a Lagrangian fibration and
$\sigma$ acts trivially on the cohomology of the fibers (see
\Cref{lem:sigma_2}). The first nonzero component of $\sigma$ is therefore
\[
	\sigma_1 \colon (\Pmod_i, F_{\bullet} \Pmod_i) \to (\Pmod_{i+1}, F_{\bullet-2}
		\Pmod_{i+1}) \decal{1}.
\]
From $\sigma_1$, we get another morphism (which now only exists in the derived category)
\[
	\sigma_1 \colon G_{i,k} \to G_{i+1,k+2} \decal{2},
\]
The following result might be called the ``symplectic relative Hard Lefschetz theorem''.

\begin{pthm} \label{thm:HL-sigma}
	The induced morphism
	\[
		\sigma_1^k \colon G_{i,-k} \to G_{i+k,k} \decal{2k}
	\]
	is an isomorphism for every $k \geq 1$ and every $i \in \ZZ$.
\end{pthm}

\newpar \label{par:hexagon}
The following picture may be helpful in understanding this result. The
complexes $G_{i,k}$ are exact unless $-n \leq i,k \leq n$ and $-n \leq i-k
\leq n$ (see \Cref{lem:amplitude}). We can therefore arrange them on a hexagonal grid 
by putting $G_{i,k}$ at the point with coordinates $i \rho + k$, where $\rho =
\frac{1}{2}(-1 + \sqrt{-3})$ is a cube root of unity:

\begin{center}
\begin{tikzpicture}[scale=1] 
	\foreach \i in {0,...,3} {
		\pgfmathsetmacro{\b}{-3+\i}
		\foreach \k in {\b,...,3} { 
			\filldraw[black] (\k - 0.5 * \i, 0.866 * \i) circle (0.1cm);
		}
	}
	\foreach \i in {-3,...,-1} {
		\pgfmathsetmacro{\t}{3+\i}
		\foreach \k in {-3,...,\t} { 
			\filldraw[black] (\k - 0.5 * \i, 0.866 * \i) circle (0.1cm);
		}
	}
	\draw (-3 + 0.5 * 3, -0.866 * 3) node[anchor=north] {$(-n,-n)$};
	\draw (3 - 0.5 * 3, 0.866 * 3) node[anchor=south] {$(n,n)$};
	\draw (0.5 * 3, -0.866 * 3) node[anchor=north] {$(-n,0)$};
	\draw (-0.5 * 3, 0.866 * 3) node[anchor=south] {$(n,0)$};
	\draw (-3, 0) node[anchor=south] {$(0,-n)$};
	\draw (3, 0) node[anchor=south] {$(0,n)$};
	\draw (0, 0.866 * 2) node[anchor=south] {$(i,k)$};
	\draw (3 * 0.5, 0.866) node[anchor=south] {$(k,i)$};
	\draw[dashed] (-0.5 * 3, 0.866 * 3) -- (0.5 * 3, -0.866 * 3);
	\draw[dashed] (0.5 * 3, 0.866 * 3) -- (-0.5 * 3, -0.866 * 3);
	\draw[dashed] (-3,0) -- (3,0);
	\draw[white,line width=0.6mm,->,>=latex,cap=round] 
		(-2 + 0.5 * 2, -0.866 * 2) -- (-2 + 0.5 * 2, 0.05);
	\draw[white,line width=0.6mm,->,>=latex,cap=round] 
		(-2 + 0.5 * 2, -0.866 * 2) -- (0.54, -0.844);
	\draw[red,thick,->,>=latex,cap=round] (-2 + 0.5 * 2, -0.866 * 2) -- (-2 + 0.5 * 2, 0)
			node[midway,fill=white] {$\omega_2$};
	\draw[red,thick,->,>=latex,cap=round] (-2 + 0.5 * 2, -0.866 * 2) -- (0.5 * 1, -0.866)
			node[midway,fill=white] {$\sigma_1$};
\end{tikzpicture}
\end{center}

The symmetry coming from the relative Hard Lefschetz theorem for the K\"ahler form
$\omega$, which exchanges the two points $(-i,k)$ and $(i,i+k)$, is reflection in one
of the diagonals of this hexagon; the symmetry coming from \Cref{thm:HL-sigma},
which exchanges the two points $(i,-k)$ and $(i+k,k)$, is reflection in another one.%
\footnote{If we forget about the shifts that appear when applying $\omega_2$ and
$\sigma_1$} 
These two reflections together generate the symmetric group $S_3$, and the reflection
in the remaining diagonal that we get in this way exchanges the two points $(i,k)$ and
$(k,i)$. This gives a conceptual proof for \Cref{conj:symmetry}. 

\newpar \label{par:slthree}
Note that $S_3$ is the Weyl group of the Lie algebra $\slthree$; this suggests that
the direct sum of all the complexes $G_{i,k}$ should form a representation
of $\slthree$. The precise statement is slightly more cumbersome because both
$\omega_2$ and $\sigma_1$ involve a shift.

\begin{pthm} \label{thm:slthree}
	The two operators $\omega_2 \colon G_{i,k} \to G_{i+2,k+1} \decal{2}$ and
	$\sigma_1 \colon G_{i,k} \to G_{i+1,k+2} \decal{2}$ determine a representation of
	the Lie algebra $\slthree$ on the object 
	\[
		\bigoplus_{i,k = -n}^n G_{i,k} \Bigl[ \lfloor \tfrac{2}{3}(i+k) \rfloor \Bigr],
	\]
	in the derived category.
\end{pthm}

\newpar
The $\sltwo$-representations determined by $\omega_2$ and $\sigma_1$ give us two
more operators
\[
	\Ysl_{\omega_2} \colon G_{i,k} \to G_{i-2,k-1} \decal{-2} 
	\quad \text{and} \quad 
	\Ysl_{\sigma_1} \colon G_{i,k} \to G_{i-1,k-2} \decal{-2}. 
\]
We prove \Cref{thm:slthree} by verifying that the four operators $\omega_2$,
$\sigma_1$, $\Ysl_{\omega_2}$, and $\Ysl_{\sigma_1}$ satisfy the Serre relations for
the Lie algebra $\slthree$. (This is similar to the $\sltwo$-action on the cohomology
of a compact K\"ahler manifold, which is also described in terms of generators and
relations.) Since $\omega$ and $\sigma$ commute as $2$-forms on
$M$, it is easy to see that $[\omega_2, \sigma_1] = 0$. The nonobvious part of the
Serre relations is that $[\Ysl_{\omega_2}, \Ysl_{\sigma_1}] = 0$. This sort of
identity may be familiar from the work of Looijenga and Lunts \cite[Sec.~4]{LL},
specifically from the computation of the total Lie algebra acting on the cohomology
of an irreducible compact hyperk\"ahler manifold. That said, the proof is completely
different in our case, because we do not have the hyperk\"ahler metric (or harmonic
forms) to work with.

\newpar
The nontrivial Serre relation also explains very nicely the analogy between the
relative Hard Lefschetz theorems for $\omega_2$ and $\sigma_1$ that we had observed
in \Cref{par:analogy}.

\begin{pcor}
	The isomorphism $G_{i,k} \cong G_{k,i}$ can be chosen in such a way that it
	interchanges the action of $\omega_2 \colon G_{i,k} \to G_{i+2,k+1} \decal{2}$ and
	$\sigma_1 \colon G_{i,k} \to G_{i+1,k+2} \decal{2}$.
\end{pcor}

Concretely, this is saying that the reflection along the third diagonal in the
picture in \Cref{par:hexagon} also interchanges the two arrows marked $\omega_2$
and $\sigma_1$. This is a general fact about representations of the Lie
algebra $\slthree$.

\subsection{Relating holomorphic forms and perverse sheaves}

\newpar
We will deduce \Cref{thm:HL-sigma} from a much more precise relationship between
$\Omega_M^{n+i}$ and $P_i$. The main result, stated below, is a sharpening of another
conjecture by Maulik, Shen and Yin \cite{MSY}, who had proposed relating the two
objects with the help of the Fourier-Mukai transform (whose existence for arbitrary
Lagrangian fibrations is unfortunately still a conjecture). It is formulated using
the BGG correspondence
\cite{BGG,EFS}, which relates graded modules over the symmetric algebra $\algS =
\Sym(\shT_B)$ and graded modules over the algebra $\algE = \bigoplus_j \Omega_B^j$ of
holomorphic forms on $B$.\footnote{One can think of the BGG correspondence as being a
linearization of the Fourier-Mukai transform.} 
More precisely, the BGG correspondence gives an equivalence
between derived categories
\[
	\BGGR \colon \Dbcoh G(\algS) \to \Dbcoh G(\algE).
\]
From the filtered $\Dmod_B$-module $(\Pmod_i, F_{\bullet} \Pmod_i)$, we obtain
the graded $\algS$-module $\gr_{\bullet}^F \Pmod_i$, and under the BGG
correspondence, this goes to the complex of graded $\algE$-modules
\[
	\BGGR \bigl( \gr_{\bullet}^F \Pmod_i \bigr)
	= \bigoplus_{k=-n}^n \gr_{-k}^F \DR(\Pmod_i) \decal{k}
	= \bigoplus_{k=-n}^n G_{i,k} \decal{i+k}.
\]
Here the grading in the direct sum is by $k$, and the $\algE$-module structure is
induced by the natural morphism of complexes
\[
	\Omega_B^j \tensor \gr_{-k}^F \DR(\Pmod_i) \to \gr_{-k-j}^F \DR(\Pmod_i) \decal{j}
\]
that contracts forms against vector fields. The content of the BGG correspondence is
that one can recover the graded $\algS$-module $\gr_{\bullet}^F \Pmod_i$ from this object.

\newpar
It turns out that one can also construct a complex of graded $\algE$-modules from
$\Omega_M^{n+k}$. If we rewrite \eqref{eq:Saito-Omega} in terms of the complexes
$G_{i,k}$, it becomes
\[
	\derR \pil \Omega_M^{n+k} \decal{n} \cong \bigoplus_{i=-n}^n G_{i,k} \decal{k}.
\]
The filtration by increasing $i$ might be called the ``perverse filtration'', because
it is induced by the usual perverse filtration in the decomposition theorem. Now the
derived pushforward $\derR \pil \shO_M$ clearly acts on this complex; it preserves
the perverse filtration, but not the grading in the above isomorphism. According
to a theorem of Matsushita \cite{Mat-img}, $\derR \pil \shO_M$ is formal
and closely related to $\algE$, in the sense that
\[
	\derR \pil \shO_M \cong \bigoplus_{j=0}^n R^j \pil \shO_M \decal{-j}
	\cong \bigoplus_{j=0}^n \Omega_B^j \decal{-j}.
\]
From $\derR \pil \Omega_M^{n+k} \decal{n}$, we obtain a complex of graded
$\algE$-modules by the following procedure. First, we take the associated graded with
respect to the perverse filtration; this gives
\[
	\derR \pil \Omega_M^{n+k} \decal{n} \cong \bigoplus_{i=-n}^n G_{i,k} \decal{k}
\]
the structure of a graded module over $\bigoplus_j \Omega_B^j \decal{-j}$. Then we
turn it into a complex of graded modules over $\algE$ by adding suitable shifts;
in this way, we arrive at 
\[
	\bigoplus_{i=-n}^n G_{i,k} \decal{i+k}.
\]
Here the summand $G_{i,k}$ has degree $i$ with respect to the grading, and the
$\algE$-module structure is induced by the collection of morphisms
\[
	R^j \pil \shO_M \tensor \gr_{-k}^F \DR(\Pmod_i) \to \gr_{-k}^F \DR(\Pmod_{i+j}) 
\]
together with Matsushita's theorem (see \Cref{thm:Matsushita}).

\newpar
The next result, which is really the main result of the paper, is that the two
complexes of graded $\algE$-modules derived from $P_i$ and $\Omega_M^{n+i}$ are
isomorphic in the derived category.

\begin{pthm} \label{thm:isomorphism}
	The two complexes
	\[
		\bigoplus_{k=-n}^n G_{i,k} \decal{i+k} \quad \text{and} \quad
		\bigoplus_{k=-n}^n G_{k,i} \decal{i+k}
	\]
	with their respective gradings and $\algE$-module structures, are isomorphic in
	$\Dbcoh G(\algE)$.
\end{pthm}

According to this theorem, $P_i$ and $\Omega_M^{n+i}$ are related in the following
way. Starting from the filtered $\Dmod_B$-module $(\Pmod_i, F_{\bullet}
\Pmod_i)$, we first take the associated graded with respect to the Hodge filtration,
and then we apply the BGG correspondence to produce a complex of graded
$\algE$-modules. This complex is isomorphic to the complex that we get from $\derR
\pil \Omega_M^{n+i} \decal{n}$ by taking the associated graded with respect to the
perverse filtration, and then using Matsushita's theorem to convert the action by
$\derR \pil \shO_M$ into an action by $\algE$. 

\newpar
Note the very striking symmetry in this procedure: on one side, we start from a
perverse sheaf and take the associated graded with respect to the Hodge filtration
(which involves holomorphic forms); on the other side, we start from the sheaf of
holomorphic forms and take the associated graded with respect to the perverse
filtration. In fact, the perverse sheaf $P_i$ is one of the graded pieces of the
complex $\derR \pil \QQ_M(n) \decal{2n}$ with respect to the perverse filtration, and
the sheaf $\Omega_M^{n+i}$ is one of the graded pieces of the de Rham
complex $\DR(\omega_M)$ with respect to the Hodge filtration. The main result is
therefore saying that
\[
	\gr^F \circ \gr^P \cong \gr^P \circ \gr^F;
\]
in words, taking the associated graded with respect to the Hodge filtration
commutes with taking the associated graded with respect to the perverse filtration.

\subsection{An outline of the proof}

\newpar
Let us a give a very brief outline of the proof, with references to the relevant
sections of the paper. From the fact that $\sigma^k \colon \Omega_M^{n-k} \to
\Omega_M^{n+k}$ is an isomorphism for every $k \geq 1$, we first deduce the following
theorem (which is \Cref{thm:equivalence} below).

\begin{pthm} 
	The following three statements are equivalent:
	\begin{aenumerate}
	\item The complexes $G_{i,k}$ and $G_{k,i}$ are isomorphic in the derived category
		(for all $i,k$).
	\item For every $k = -n, \dotsc, n$, there is an isomorphism (in the derived category)
		\[
			\bigoplus_{i=-n}^n G_{i,k} \decal{i+k} \cong \bigoplus_{i=-n}^n G_{k,i}
			\decal{i+k}.
		\]
	\item For every $k \geq 1$, the morphism $\sigma_1^k \colon G_{i,-k} \to G_{i,k}
		\decal{2k}$ is an isomorphism.
	\end{aenumerate}
\end{pthm}

Along the way, we give a short proof for Matsushita's theorem (in
\Cref{chap:Matsushita}), because it serves as a nice introduction to the general
method. This reduces the proof of the symplectic relative Hard Lefschetz theorem to
establishing the isomorphism in (b). 

\newpar
The heart of the matter is the proof of \Cref{thm:isomorphism}. This is a result
about the relationship between $\gr_{\bullet}^F \Pmod_i$ and $\derR \pil \Omega_M^{n+i}
\decal{n}$ for each $i = -n, \dotsc n$ separately -- but it turns out that by combining
all these objects into one, we get some additional structure that we can exploit. 
With that idea in mind, we consider the direct sum
\begin{equation} \label{eq:intro-object}
	\bigoplus_{i=-n}^n \gr_{\bullet}^F \Pmod_i \decal{-i},
\end{equation}
which lives in the derived category of graded modules over $\algS = \Sym(\shT_B$).
Under the BGG correspondence (see \Cref{chap:BGG}), this goes to the complex
\[
	G = \bigoplus_{i,k=-n}^n G_{i,k} \decal{k} = 
	\BGGR \left( \bigoplus_{i=-n}^n \gr_{\bullet}^F \Pmod_i \decal{-i} \right).
\]
Our starting point is a concrete description of $G$ in terms of smooth
differential forms. Saito's version of the decomposition theorem gives us an
isomorphism between \eqref{eq:intro-object} and a certain complex of graded
$\algS$-modules $\gr_{\bullet}^F \Cpi$ (see \Cref{par:Laumon}). If we apply the BGG
correspondence to this isomorphism, we obtain an isomorphism (in the derived
category)
\[
	G \cong (M,d),
\]
where $(M,d)$ is the complex of graded $\algE$-modules with
\[
	M_k^i = \pil \shA_M^{n+k,n+i} \quad \text{and} \quad d = (-1)^k \dbar.
\]
The module structure on $(M,d)$ is the obvious one: a local section $\beta$ of
$\Omega_B^j$ acts as wedge product with the pulllback $\piu \beta$. 

\newpar
The next idea is to transform this module structure into a different one with the
help of the symplectic form $\sigma$ and the K\"ahler form $\omega$. This is based on
the following simple construction (which is also behind Matsushita's theorem). Suppose
that $\beta \in H^0(U, \Omega_B^1)$ is a holomorphic $1$-form, defined on an open
subset $U \subseteq B$. Using the symplectic form, we can transform the holomorphic
$1$-form $\piu \beta$ into a holomorphic vector field
\[
	v(\beta) \in H^0 \bigl( \pi^{-1}(U), \shT_M \bigr), \quad
	\piu \beta = v(\beta) \cont \sigma,
\]
where $\cont$ means contraction with a vector field. Using the K\"ahler form, we can
further transform the vector field $v(\beta)$ into a $\dbar$-closed $(0,1)$-form
\[
	f(\beta) \in A^{0,1} \bigl( \pi^{-1}(U) \bigr), \quad
	f(\beta) = -v(\beta) \cont \omega.
\]
All three of these objects act on the complex $(M,d)$, either by wedge product or by
contraction. We use the reflection operator (or Weil element) coming from the
symplectic form $\sigma$ to show that $G$ is isomorphic to the auxiliary object
\[
	G_v = \bigoplus_{i,k=-n}^n G_{i,-k} \decal{-k},
\]
in a way that exchanges the action by $\beta \in \Omega_B^1$ and the action of the
vector field $v(\beta)$. We then use the Weil operator coming from the relative Hard
Lefschetz theorem for the K\"ahler form $\omega$ to show that $G_v$ is in turn
isomorphic to
\[
	G_f = \bigoplus_{i,k=-n}^n G_{i,k} \decal{i+k},
\]
in a way that exchanges the action by the vector field $v(\beta)$ and the action of
the $(0,1$)-form $f(\beta)$. \Cref{thm:isomorphism} follows from the isomorphism $G
\cong G_f$ by a careful analysis of the BGG correspondence and some general facts
about Hodge modules.

\subsection{Some applications}

\newpar
One application is a new proof for the ``numerical perverse = Hodge'' symmetry
\cite{SY-top} for compact holomorphic symplectic manifolds. Our proof does
\emph{not} use the existence of a hyperk\"ahler metric on $M$; in return, we need to
assume that $B$ is a complex manifold.

\begin{pthm}[Shen-Yin] \label{thm:Shen-Yin}
	Let $M$ be a holomorphic symplectic manifold that is compact and K\"ahler. If $\pi
	\colon M \to B$ is a Lagrangian fibration whose base $B$ is a complex manifold,
	then
	\[
		H^j(B, G_{i,k}) \cong H^j(B, G_{k,i}) \quad \text{for all $i,j,k \in
		\ZZ$.}
	\]
\end{pthm}

\newpar
Looijenga-Lunts \cite[\S4]{LL} and Verbitsky \cite{Verbitsky} showed that cohomology
of a compact hyperk\"ahler manifold with a Lagrangian fibration carries an action by
the Lie algebra $\sosix \cong \slfour$. We prove a generalization of this result to
Lagrangian fibrations on compact holomorphic symplectic manifolds.

\begin{pthm}
	Let $M$ be a holomorphic symplectic compact K\"ahler manifold, and let 
	$\pi \colon M \to B$ be a Lagrangian fibration over a compact K\"ahler manifold
	$B$. In this situation, the cohomology of $M$ is a representation of the Lie
	algebra $\slfour$.
\end{pthm}

More precisely, the weight spaces of this representation are
\[
	H^{i,j,k} = H^{i+j}(B, G_{j,k}) = H^j \bigl( B, \gr_{-k}^F \DR(\Pmod_i) \bigr),
\]
and the representation on
\[
	\bigoplus_{j,k = -n}^n H^{n+k,n+j}(M) \cong \bigoplus_{i,j,k = -n}^n H^{i,j,k}
\]
is built from the two operators $\omega_2$ and $\sigma_1$ and the action by a
K\"ahler form on $B$. Besides the symplectic relative Hard Lefschetz theorem, our
proof relies on some identities among differential forms on $M$; this is entirely
different from Verbitsky's proof, which needs a hyperk\"ahler metric and the theory
of harmonic forms.

\newpar
Another (quite elementary) byproduct is the following bound on the supports in the
decomposition theorem for Lagrangian fibrations (see \Cref{prop:strict-support});
this only relies on the the fact the Lagrangian fibrations are equidimensional.

\begin{pprop}
	In the decomposition by strict support of the Hodge module $P_i$, the support of
	every summand has dimension $\geq \abs{i}$.
\end{pprop}

For example, consider the decomposition by strict support of the Hodge module
$P_{-n+1}$, whose restriction to the smooth locus of $\pi$ is the variation of Hodge
structure on $H^1(M_b, \QQ)$, where $M_b = \pi^{-1}(b)$. The summand with strict
support $B$ is the intersection complex of the variation of Hodge structure. The
proposition is telling us that all other summands in the decomposition by
strict support must be supported on \emph{divisors} in $B$.

\section{Hodge modules and the decomposition theorem}

\newpar
In this chapter, we review a few relevant results about Hodge modules and introduce
the main objects of study. For a short overview of Hodge modules, one can look at
Saito's ``Introduction to Mixed Hodge Modules'' \cite{Saito-intro} or my more recent
survey paper \cite{sanya}; for more details, there are Saito's original papers
\cite{HM,MHM}, as well as the ``Mixed Hodge Module Project'' by Sabbah and myself
\cite{MHM-project}. 

\newpar
Recall that a Hodge module on a complex manifold $B$ has three components: a
perverse sheaf $P$ with coefficients in $\QQ$; a right $\Dmod_B$-module $\Pmod$; and
a good filtration $F_{\bullet} \Pmod$ by coherent $\shO_B$-modules. The three components
are related by the Riemann-Hilbert correspondence: the precise requirement is that
\[
	\DR(\Pmod) \cong P \tensor_{\QQ} \CC
\]
are isomorphic as perverse sheaves with coefficients in $\CC$. The de Rham complex
(or Spencer complex) of the right $\Dmod$-module $\Pmod$ is the complex
\[
	\DR(\Pmod) = \Bigl\lbrack 
	\Pmod \tensor \wed^n \shT_B \to \dotsb \to 
		\Pmod \tensor \shT_B \to \Pmod
	\Bigr\rbrack
\]
which lives in cohomological degrees $-n, \dotsc, 0$, where $n = \dim B$. The
differential in the de Rham complex is given by the (local) formula
\[
	\delta \colon \Pmod \tensor \wed^k \shT_B \to \Pmod \tensor \wed^{k-1} \shT_B,
	\quad
	\delta(s \tensor \partial_J) = \sum_{j=1}^n \sgn(J,j) \cdot s \partial_j \tensor
	\partial_{J \setminus \{j\}}.
\]
Here the notation is as follows. Let $t_1, \dotsc, t_n$ be local holomorphic coordinates on
$B$, and denote by $\partial_j = \partial/\partial t_j$ the resulting holomorphic
vector fields. For any subset $J \subseteq \{1, \dotsc, n\}$, we list the elements
in increasing order as $j_1 < \dotsb < j_{\ell}$, and then define
\[
	\partial_J = \partial_{j_1} \wedge \dotsb \wedge \partial_{j_{\ell}}
\]
with the convention that this expression equals $1$ when $J$ is empty. We also define 
\[
	\sgn(J, j) = \begin{cases}
		(-1)^{k-1} &\text{if $j = j_k$,} \\
		0 & \text{if $j \not\in J$.}
	\end{cases}
\]
Note that we are always using Deligne's Koszul sign rule, according to which
swapping two elements of degrees $p$ and $q$ leads to a sign $(-1)^{pq}$; this is the
reason for the factor $\sgn(J,j)$. 

\newpar
The de Rham complex $\DR(\Pmod)$ is filtered by the subcomplexes
\[
	F_k \DR(\Pmod) = \Bigl\lbrack 
	F_{k-n} \Pmod \tensor \wed^n \shT_B \to \dotsb \to 
		F_{k-1} \Pmod \tensor \shT_B \to F_k \Pmod
	\Bigr\rbrack.
\]
The graded pieces of this filtration give us a collection of complexes of coherent
$\shO_B$-modules
\[
	\gr_k^F \DR(\Pmod) = \Bigl\lbrack 
	\gr_{k-n}^F \Pmod \tensor \wed^n \shT_B \to \dotsb \to 
		\gr_{k-1}^F \Pmod \tensor \shT_B \to \gr_k^F \Pmod 
	\Bigr\rbrack.
\]
Since the rational structure on $P$ is mostly irrelevant for our purposes, we
generally work with the underlying filtered $\Dmod$-module $(\Pmod, F_{\bullet}
\Pmod)$. 

\newpar
As in the introduction, let $M$ be a holomorphic symplectic complex manifold of
dimension $2n$, and let $\pi \colon M \to B$ be a Lagrangian fibration. Then
$\QQ_M \decal{2n}$ is a Hodge module of weight $2n$ on $M$, and so the Tate twist
$\QQ_M(n) \decal{2n}$ has weight $0$. The underlying filtered $\Dmod$-module is
$\omega_M$, with the filtration for which $\gr_{-n}^F \omega_M = \omega_M$. Since we
are interested in the cohomology of the fibers, we now apply the decomposition
theorem \cite[Thm.~5.3.1]{HM}; this holds for proper holomorphic mappings from
K\"ahler manifolds by recent work of Mochizuki \cite{Mochizuki}. According to the
decomposition theorem, the direct image decomposes as
\[
	\derR \pil \QQ_M(n) \decal{2n} \cong \bigoplus_{i=-n}^n P_i \decal{-i},
\]
where each $P_i$ is a polarizable Hodge module of weight $i$ on the complex
manifold $B$. Note that $P_i$ can only be nonzero for $i = -n, \dotsc, n$; this is
because all fibers of $\pi$ have dimension $n$ by \cite[Thm.~1]{Mat-equi}.
On the open subset of $B$ over which
$\pi$ is submersive, the fibers of $\pi$ are $n$-dimensional abelian varieties, and
$P_i$ is just the variation of Hodge structure on the $(n+i)$-th cohomology of the
fibers. If we write $(\Pmod_i, F_{\bullet} \Pmod_i)$ for the filtered $\Dmod$-module
underlying $P_i$, we get an induced decomposition
\[
	\pi_{+} \bigl( \omega_M, F_{\bullet} \omega_M \bigr) \cong
	\bigoplus_{i=-n}^n (\Pmod_i, F_{\bullet} \Pmod_i) \decal{-i}.
\]
in the derived category of filtered right $\Dmod_B$-modules.

\newpar \label{par:Laumon}
We are going to need a more concrete description for the direct image of $(\omega_M,
F_{\bullet} \omega_M)$ in terms of smooth forms. This is easily derived from Saito's
formalism of induced $\Dmod$-modules \cite[\S2.1]{HM}. Let $\shA_M^{p,q}$ be
the sheaf of smooth
$(p,q)$-forms on $M$; this is a fine sheaf, and therefore acyclic for the functor
$\pil$. Let $\Cpi$ be the complex of filtered right $\Dmod_B$-modules
whose $i$-th term is the filtered right $\Dmod_B$-module
\[
	\Cpi^i = \bigoplus_{p+q=i} \pil \shA_M^{n+p,n+q} \tensor_{\shO_M} 
		(\Dmod_B, F_{\bullet+p} \Dmod_B),
\]
and whose differential is given (in local coordinates) by the formula
\[
	\dpi \colon \Cpi^i \to \Cpi^{i+1}, \quad
	\dpi \bigl( \alpha \tensor D \bigr) = d \alpha \tensor D + \sum_{j=1}^n \piu(\dtj) \wedge
	\alpha \tensor \partial_j D,
\]
where $\partial_j = \partial/\partial t_j$. The indexing is done is such a way that
$(\Pmod_i, F_{\bullet} \Pmod_i)$ is exactly the $i$-th cohomology module of the
complex $\Cpi$; the decomposition theorem tells us that 
\begin{equation} \label{eq:Cpi-decomposition}
	\Cpi \cong \bigoplus_{i=-n}^{n} (\Pmod_i, F_{\bullet} \Pmod_i) \decal{-i}
\end{equation}
in the derived category of filtered right $\Dmod_B$-modules. (More precisely, this is
a consequence of the strictness property for the direct image of the underlying
filtered $\Dmod$-modules.) Passing to the associated graded objects, we get an
isomorphism 
\[
	\gr_{\bullet}^F \Cpi \cong \bigoplus_{i=-n}^n \gr_{\bullet}^F \Pmod_i \decal{-i}
\]
in the derived category of graded modules over $\gr_{\bullet}^F \Dmod_B \cong
\Sym^{\bullet}(\shT_B$). The left-hand side is the complex with terms
\[
	\gr_{\bullet}^F \Cpi^i = \bigoplus_{p+q=i} \pil \shA_M^{n+p,n+q} \tensor_{\shO_M} 
	\Sym^{\bullet+p}(\shT_B),
\]
and with differential (again in local coordinates)
\[
	\dpi \bigl( \alpha \tensor P \bigr) = \dbar \alpha \tensor P + 
	\sum_{j=1}^n \piu(\dtj) \wedge \alpha \tensor \partial_j P.
\]
Compare this with Laumon's description \cite[Constr.~2.3.3]{Laumon:DF} of the
associated graded of the direct image in the derived category of filtered $\Dmod$-modules.

\newpar \label{par:Deligne}
Let $\omega \in A^{1,1}(M)$ be a K\"ahler form on $M$. Once we have made this choice,
there is a preferred decomposition in the decomposition theorem, constructed by
Deligne \cite{Deligne}; with considerable understatement, Deligne calls it ``less bad'' than
the others. This works as follows. Since $d \omega = 0$, the K\"ahler form induces a
morphism of complexes 
\[
	\omega \colon \Cpi \to \Cpi(-1) \decal{2}
\]
that increases the cohomological degree by $2$ and decreases the degree with respect
to the filtration by $1$. For any choice of decomposition in the decomposition theorem,
the isomorphism in \eqref{eq:Cpi-decomposition} lets us break up
\[
	\omega \colon \bigoplus_{i=-n}^{n} (\Pmod_i, F_{\bullet} \Pmod_i) \decal{-i}
		\to \bigoplus_{i=-n}^{n} (\Pmod_i, F_{\bullet-1} \Pmod_i) \decal{2-i}
\]
into a finite sum $\omega = \omega_2 + \omega_1 + \omega_0 + \dotsb$, where each component
$\omega_j$ is a morphism
\[
	\omega_j \colon (\Pmod_i, F_{\bullet} \Pmod_i) \to
	(\Pmod_{i+j}, F_{\bullet-1} \Pmod_{i+j}) \decal{2-j}
\]
in the derived category of filtered right $\Dmod_B$-modules. According to the
relative Hard Lefschetz theorem, the morphism
\[
	\omega_2^i \colon (\Pmod_{-i}, F_{\bullet} \Pmod_i)
	\to (\Pmod_i, F_{\bullet-i} \Pmod_i)
\]
is an isomorphism for every $i \geq 1$. This means that we get a representation of
the Lie algebra $\sltwo$ on the direct sum
\[
	\bigoplus_{i=-n}^n \Pmod_i.
\]
If $\Hsl, \Xsl, \Ysl \in \sltwo$ denote the three standard generators, then $\Xsl$
acts as $\omega_2$ and $\Hsl$ acts as multiplication by the integer $i$ on the
summand $\Pmod_i$. Deligne proves that there is a unique choice of decomposition for
which the components $\omega_j$ with $j \leq 1$ are \emph{primitive}, meaning that they
commute with the operator $\Ysl$ in the $\sltwo$-representation. Since the weight of
a primitive element (with respect to $\ad \Ysl$) must be $\leq 0$, it follows that
$\omega_1 = 0$. In general, Deligne's decomposition tends to eliminate unwanted
components in the decomposition of various operators; we will exploit this effect
later on.

\newpar
Let us now turn our attention to the symplectic form $\sigma \in H^0(M, \Omega_M^2)$.
Since we are assuming that $d \sigma = 0$, the symplectic form also induces a
morphism of complexes
\[
	\sigma \colon \Cpi \to \Cpi(-2) \decal{2}
\]
that increases the cohomological degree by $2$ and decreases the degree with respect
to the filtration by $2$. Using \eqref{eq:Cpi-decomposition}, we again get a
decomposition of
\[
	\sigma \colon \bigoplus_{i=-n}^{n} (\Pmod_i, F_{\bullet} \Pmod_i) \decal{-i}
		\to \bigoplus_{i=-n}^{n} (\Pmod_i, F_{\bullet-2} \Pmod_i) \decal{2-i}
\]
into a finite sum $\sigma = \sigma_2 + \sigma_1 + \sigma_0 + \dotsb$, where each component
$\sigma_j$ is now a morphism
\[
	\sigma_j \colon (\Pmod_i, F_{\bullet} \Pmod_i) \to
	(\Pmod_{i+j}, F_{\bullet-2} \Pmod_{i+j}) \decal{2-j}.
\]

\newpar
The Lagrangian condition implies the vanishing of the topmost component $\sigma_2$.

\begin{plem} \label{lem:sigma_2}
	We have $\sigma_2 = 0$.
\end{plem}

\begin{proof}
	Since there is no shift, $\sigma_2 \colon \Pmod_i \to \Pmod_{i+2}$ is a morphism
	of right $\Dmod_B$-modules. Both $\Dmod$-modules underlie polarizable Hodge
	modules on $B$, and therefore admit a decomposition by strict support
	\cite[\S5.1]{HM}. It is then enough to show that $\sigma_2$ vanishes on every
	summand in this decomposition; the reason is that morphisms of $\Dmod$-modules
	respect the decomposition by strict support. If we take one of the summands of the
	Hodge module $P_i$, say with
	strict support $Z \subseteq B$, then on a dense open subset of $Z$, it comes from
	a variation of Hodge structure of weight $i-\dim Z$ \cite[Lem.~5.1.10]{HM}. The
	strict support condition
	then means that we only have to check that the restriction of $\sigma_2$ to a
	general point $b \in Z$ is zero. Let $i_b \colon \{b\} \to B$ be the inclusion. By
	proper base change for Hodge modules, we have
	\[
		\iu_b \derR \pil \QQ_M(n)[2n] \cong \derR \pil \QQ_{M_b}(n)[2n],
	\]
	where $M_b = \pi^{-1}(b)$; consequently, $H^{-\dim Z} \iu_b P_i$, which is a Hodge
	structure of weight $i-\dim Z$, is isomorphic to
	a direct summand in $H^{2n+i-\dim Z}(M_b, \QQ)(n)$. If we let $\Mt_b \to M_b$ be a
	resolution of singularities, then for weight reasons, the composition
	\[
		H^{-\dim Z} \iu_B P_i \to H^{2n+i-\dim Z}(M_b, \QQ)(n)
		\to H^{2n+i-\dim Z}(\Mt_b, \QQ)(n)
	\]
	is injective. This reduces the problem to showing that the pullback of $\sigma$ to
	$\Mt_b$ is trivial; but this follows from the fact that $\pi \colon M \to B$ is
	Lagrangian, according to a theorem by Matsushita \cite[Thm.~1]{Mat-equi}.
\end{proof}

\newpar
Since we are using Deligne's decomposition, we can say a bit more about the
other components of $\sigma = \sigma_1 + \sigma_0 + \dotsb$. This is not really going
to play a role in what follows, but the same kind of proof will appear later on.

\begin{plem} \label{lem:sigma_k}
	We have $[\omega_2, \sigma_1] = 0$, and the components $\sigma_j$ with $j \leq 0$
	are primitive (with respect to the representation of $\sltwo$ determined by
	$\omega_2)$.
\end{plem}

\begin{proof}
	Since the two forms $\omega$ and $\sigma$ commute, we get $[\omega, \sigma] = 0$.
	Decomposing this relation by degree, we find that $[\omega_2, \sigma_1] = 0$;
	likewise, $[\omega_2, \sigma_0] = 0$, and because $\sigma_0$ has weight $0$ (in
	the $\sltwo$-representation), this means that $\sigma_0$ is primitive. We can
	therefore assume by induction that $\sigma_0, \dotsc, \sigma_{-k+1}$ are primitive
	for some $k \geq 1$.  Let us prove that $\sigma_{-k}$ is also primitive.
	From the relation $[\omega, \sigma] = 0$, we get
	\[
		[\omega_2, \sigma_{-k}] + [\omega_0, \sigma_{-k+2}] + \dotsb + [\omega_{-k+3},
		\sigma_1] = 0.
	\]
	We know that $(\ad \omega_2)^{j+1} \omega_{-j} = 0$ for all $j \geq 0$, due to the
	fact that $\omega_{-j}$ is primitive; similarly, $(\ad \omega_2)^{j+1} \sigma_{-j}
	= 0$ for $j = 0, \dotsc, k-1$. Now $\ad \omega_2$ is a derivation, and so
	\[
		-(\ad \omega_2)^{k+1} \sigma_{-k} = (\ad \omega_2)^k [\omega_0, \sigma_{-k+2}] 
		+ \dotsb + (\ad \omega_2)^k [\omega_{-k+3}, \sigma_1] = 0.
	\]
	Since $\sigma_{-k}$ has weight $-k$, this proves that it is primitive.
\end{proof}

\newpar
We are going to need two other facts about Hodge modules. The first is the
compatibility of the de Rham complex with direct images \cite[\S2.3.7]{HM}. It says that
\[
	\derR \pil \gr_{-k}^F \DR(\omega_M) \cong 
	\bigoplus_{i=-n}^n \gr_{-k}^F \DR(\Pmod_i) \decal{-i},
\]
where we give $\omega_M$ the filtration for which $\gr_{-n}^F \omega_M = \omega_M$,
in accordance with the Tate twist in $\QQ_M(n) \decal{2n}$. Since we have $\gr_{-k}^F
\DR(\omega_M) = \Omega_M^{n+k} \decal{n-k}$, it follows that
\begin{equation} \label{eq:pil-Omega}
	\derR \pil \Omega_M^{n+k} \decal{n-k} \cong
	\bigoplus_{i=-n}^n \gr_{-k}^F \DR(\Pmod_i) \decal{-i}.
\end{equation}

\newpar
The second fact about Hodge modules concerns duality. Let $\DB$ denote the duality
functor on Hodge modules. The polarization on $\QQ_M(n) \decal{2n}$ induces an
isomorphism $P_{-i} \cong \DB(P_i)$ between Hodge modules of weight $i$. On the level
of filtered $\Dmod$-modules, this gives us an isomorphism of right $\Dmod_B$-modules
\[
	\Pmod_{-i} \cong \omega_B \tensor \derR \shHom_{\Dmod_B} \bigl( \Pmod_i, \Dmod_B
	\bigr) \decal{n},
\]
compatible with the filtrations on both sides. Passing to the associated graded
modules over $\gr_{\bullet}^F \Dmod_B \cong \Sym(\shT_B)$, we get
\[
	\gr_{\bullet}^F \Pmod_{-i} \cong \omega_B \tensor \derR \shHom_{\Sym(\shT_B)}
	\Bigl( \gr_{\bullet}^F \Pmod_i, \Sym(\shT_B) \Bigr) \decal{n},
\]
where sections of $\Sym^j(\shT_B)$ act with an extra factor of $(-1)^j$ on the
right-hand side (due to the sign in the conversion from left to right $\Dmod$-modules).
The important fact, which is hidden inside the definition of the duality functor for
Hodge modules, is that $\gr_{\bullet}^F \Pmod_i$ is an
$n$-dimensional Cohen-Macaulay module over $\Sym^{\bullet}(\shT_B)$ \cite[Lem~5.1.13]{HM}. In
geometric terms, $\gr_{\bullet}^F \Pmod_i$ gives a coherent sheaf on the cotangent
bundle $T^{\ast} B$, whose support is the ($n$-dimensional) characteristic variety of
the $\Dmod$-module $\Pmod_i$, and the statement is that this sheaf is Cohen-Macaulay.
This is one of the special properties of Hodge modules, and the proof of
\Cref{thm:isomorphism} would not work without this fact.

\section{Basic properties of the complexes $G_{i,k}$}

\newpar
The conjecture by Shen and Yin is about the complexes of coherent $\shO_B$-modules
\[
	G_{i,k} = \gr_{-k}^F \DR(\Pmod_i) \decal{-i}.
\]
In this chapter, we are going to look at the basic properties of these complexes. The
results that we prove here only rely on the fact that $M$ and $B$ are complex manifolds
and that all fibers of $\pi \colon M \to B$ have the same dimension.\footnote{I thank
Junliang Shen for pointing this out to me.}

\newpar
The first observation is that the $G_{i,k}$ are related to the direct image of
$\Omega_M^{n+k}$ under the Lagrangian fibration $\pi \colon M \to B$. Indeed,
if we rewrite \eqref{eq:pil-Omega} in these terms, we get
\begin{equation} \label{eq:forms}
	\derR \pil \Omega_M^{n+k} \decal{n-k} \cong \bigoplus_{i=-n}^n G_{i,k}.
\end{equation}
Since $P_i$ can only be nonzero for $-n \leq i \leq n$, and since $\Omega_M^{n+k}$
can only be nonzero for $-n \leq k \leq n$, it follows that $G_{i,k} = 0$ unless $-n
\leq i,k \leq n$. 

\newpar
Next, let us see what duality can tell us about the complexes $G_{i,k}$. 
From the fact that $\DB(P_i) \cong P_{-i}$ are isomorphic as Hodge modules, we get an
isomorphism
\[
	\derR \shHom_{\shO_B} \Bigl( \gr_k^F \DR(\Pmod_i), \omega_B \decal{n} \Bigr)
	\cong \gr_{-k}^F \DR(\Pmod_{-i}),
\]
and therefore an isomorphism between $G_{-i,-k}$ and the Grothendieck dual of $G_{i,k}$:
\begin{equation} \label{eq:duality}
	\derR \shHom_{\shO_B} \bigl( G_{i,k}, \omega_B \decal{n} \bigr) \cong G_{-i,-k};
\end{equation}

\newpar
The isomorphism in \eqref{eq:forms} is also good for computing the amplitude 
of the complexes $G_{i,k}$, which is in agreement with \Cref{conj:symmetry}.
We know that all fibers of the Lagrangian fibration $\pi \colon
M \to B$ have dimension $n$, and so the left-hand side of \eqref{eq:forms} is concentrated in
degrees $\{-n+k, \dotsc, k\}$. The same thing is therefore true for the individual
summands $G_{i,k}$. On the other hand, we have $F_{-n-1} \Pmod_i = 0$, and so all
nonzero terms in 
\begin{equation} \label{eq:Gik-complex}
	G_{i,k} = \Bigl\lbrack 
	\gr_{-k-n}^F \Pmod_i \tensor \wed^n \shT_B \to \dotsb \to 
		\gr_{-k-1}^F \Pmod_i \tensor \shT_B \to \gr_{-k}^F \Pmod_i 
		\Bigr\rbrack \decal{-i}
\end{equation}
live in cohomological degrees $\bigl\{ -n+\max(i,i+k), \dotsc, i \bigr\}$. Taken
together with \eqref{eq:duality}, these simple observations prove the following lemma.

\begin{plem} \label{lem:amplitude}
	The complex $G_{i,k}$ is concentrated in degrees 
	\[
		\bigl\{ -n + \max(i,k,i+k), \dotsc, \min(i,k,i+k) \bigr\}.
	\]
	In particular, it is exact unless $-n \leq i-k \leq n$.
\end{plem}

The bound on the amplitude is symmetric in $i$ and $k$, as predicted by
\Cref{conj:symmetry}. Note that the complex $G_{i,k}$ is exact unless $\abs{i}
\leq n$ and $\abs{k} \leq n$ and $\abs{i-k} \leq n$.  This is the reason for the
hexagonal shape of the drawing in \parref{par:hexagon}.

\newpar
One nice consequence of the lemma is a sharp bound on the generation level of the
Hodge filtration on $P_i$; this is very hard to come by in general.
The complex $\gr_{-k}^F \DR(\Pmod_i) = G_{i,k} \decal{i}$ has no
cohomology in degree $0$ provided that $k-i < 0$, and this means that
\[
	\gr_{p-1}^F \Pmod_i \tensor \shT_B \to \gr_{p}^F \Pmod_i
\]
is surjective for $p \geq -i+1$. The Hodge filtration on the $\Dmod$-module
$\Pmod_i$ is therefore generated in degree $-i$; in symbols, $F_{-i+j} \Pmod_i =
F_{-i} \Pmod_i \cdot F_j \Dmod_B$ for $j \geq 0$. In Saito's terminology
\cite{Saito-HF}, this is saying that the generation level of $P_i$ is $\leq -i$. (The
bound is of course achieved over the smooth locus of $\pi$, since the smooth fibers
are $n$-dimensional abelian varieties.)

\newpar
Another consequence is that we can compute the projective amplitude of $G_{i,k}$. 

\begin{plem} \label{lem:amplitude-projective}
	On every Stein open subset of $B$, the complex $G_{i,k}$ is isomorphic (in the
	derived category) to a complex of locally free $\shO_B$-modules concentrated in
	degrees 
	\[
		\bigl\{ -n + \max(i,k,i+k), \dotsc, \min(i,k,i+k) \bigr\}.
	\]
\end{plem}

\begin{proof}
	After restricting to the open subset in question, we may assume that $B$ is a
	Stein manifold. In particular, every coherent $\shO_B$-module has a bounded
	resolution by locally free $\shO_B$-modules. In the derived category, the complex
	$G_{i,k}$ is therefore isomorphic to a bounded
	complex $\shE^{\bullet}$ of locally free $\shO_B$-modules, where $\shE^j = 0$ for $j >
	\min(i,k,i+k)$. The dual complex $\shHom_{\shO_B}(\shE^{-\bullet}, \omega_B)$
	computes $\derR \shHom_{\shO_B} \bigl( G_{i,k}, \omega_B \bigr) \cong G_{-i,-k}
	\decal{-n}$, and according to \Cref{lem:amplitude}, the complex $G_{-i,-k}
	\decal{-n}$ is concentrated in degrees
	\[
		\bigl\{ -\min(i,k,i+k), \dotsc, n - \max(i,k,i+k) \bigr\}.
	\]
	After truncating the complex $\shHom_{\shO_B}(\shE^{-\bullet}, \omega_B)$ in
	degrees $\leq n - \max(i,k,i+k)$, it becomes a complex of locally free
	$\shO_B$-modules in degrees 
	\[
		\bigl\{ -\min(i,k,i+k), \dotsc, n - \max(i,k,i+k) \bigr\}.
	\]
	We now get the result for the original complex $G_{i,k}$ by dualizing again.	
\end{proof}

Taken together with \eqref{eq:forms}, this tells us that $\derR \pil \Omega_M^{n+k}$
is isomorphic, on every Stein open subset of $B$, to a complex of locally free
$\shO_B$-modules in degrees $\{0, \dotsc, n\}$. Again, this is obviously true over
the smooth locus of $\pi$; the surprising thing is that it continues to be true on
the locus where the fibers are singular.

\newpar
Recall that any polarizable Hodge module admits, even locally on $B$, a decomposition by
strict support \cite[\S5.1]{HM}. The bound on the amplitude of the complexes
$G_{i,k}$ puts the following unexpected restriction on the structure of the Hodge
modules $P_i$.

\begin{pprop} \label{prop:strict-support}
	In the decomposition by strict support of the Hodge module $P_i$, the support of
	every summand has dimension $\geq \abs{i}$.
\end{pprop}

\begin{proof}
	Since $P_{-i} \cong P_i(i)$ by the relative Hard Lefschetz theorem, we may assume
	without loss of generality that $i \geq 0$. Let $P$ be one of the summands in the
	decomposition of $P_i$ by strict support. For any $k \geq i$, consider the
	complex $\gr_{-k}^F \DR(\Pmod)$. It is supported on $\Supp P$ and lives in degrees
	$\leq 0$, and so the dual complex 
	\[
		\derR \shHom_{\shO_B} \bigl( \gr_{-k}^F \DR(\Pmod), \omega_B \bigr)
	\]
	is concentrated in degrees $\geq r$, where $r = \codim_B(\Supp P)$. It is also
	a direct summand in 
	\[
		\derR \shHom_{\shO_B} \bigl( \gr_{-k}^F \DR(\Pmod_i), \omega_B \bigr)
		= \derR \shHom_{\shO_B} \bigl( G_{i,k}, \omega_B \decal{n} \bigr) \decal{-n-i}
		\cong G_{-i,-k} \decal{-n-i},
	\]
	and by \Cref{lem:amplitude}, it is therefore concentrated in degrees $\leq n-k$.
	Now if we had $r > n-i$, then the complex $\gr_{-k}^F \DR(\Pmod)$ would be exact for
	every $k \geq i$, and so $\gr_{-k}^F \Pmod = 0$ for $k \geq i$, and therefore
	$F_{-i} \Pmod = 0$. But this would say that $\Pmod = 0$, because the Hodge
	filtration on $\Pmod_i$ (and therefore on $\Pmod$) is generated in degree $-i$.
	Since this is impossible, we get $r \leq n-i$.
\end{proof}

\newpar
We end this chapter by recording the effect of the K\"ahler form $\omega$ and the
symplectic form $\sigma$ on the complexes $G_{i,k}$. We already said in the
introduction that the relative Hard Lefschetz theorem can be interpreted as a symmetry
among the complexes $G_{i,k}$. From $\omega_2 \colon (\Pmod_i, F_{\bullet} \Pmod_i)
\to (\Pmod_{i+2}, F_{\bullet-1} \Pmod_{i+2})$, we obtain a morphism of complexes
\[
	\omega_2 \colon G_{i,k} \to G_{i+2,k+1} \decal{2},
\]
and the relative Hard Lefschetz theorem implies that
\begin{equation} \label{eq:HL-omega}
	\omega_2^i \colon G_{-i,k} \to G_{i,i+k} \decal{2i}
\end{equation}
is an isomorphism for every $i \geq 1$. Similarly, $\sigma_1 \colon
(\Pmod_i, F_{\bullet} \Pmod_i) \to (\Pmod_{i+1}, F_{\bullet-2} \Pmod_{i+1})
\decal{1}$ gives us a morphism (in the derived category)
\[
	\sigma_1 \colon G_{i,k} \to G_{i+1,k+2} \decal{2},
\]
and the ``symplectic relative Hard Lefschetz theorem'' (in \Cref{thm:HL-sigma})
claims that
\begin{equation} \label{eq:HL}
	\sigma_1^k \colon G_{i,-k} \to G_{i+k,k} \decal{2k}
\end{equation}
is also an isomorphism (in the derived category). We will prove this in
\Cref{chap:symmetry}.

\section{Matsushita's theorem}
\label{chap:Matsushita}

\newpar
The fact that $\sigma^k \colon \Omega_M^{n-k} \to \Omega_M^{n+k}$ is an isomorphism
for every $k \geq 1$ gives us at least some information about the morphisms $\sigma_1
\colon G_{i,k} \to G_{i+1,k+2} \decal{2}$. On its own, this is not strong enough to
prove \Cref{thm:HL-sigma}, but it does lead to a rather short proof for the following
theorem by Matsushita \cite[Thm.~1.3]{Mat-img}, mentioned in the introduction.

\begin{pthm}[Matsushita] \label{thm:Matsushita}
	Let $\pi \colon M \to B$ be a Lagrangian fibration on a holomorphic symplectic
	manifold of dimension $\dim M = 2n$. If $M$ is K\"ahler, one has
	\[
		\derR \pil \shO_M \cong \bigoplus_{i=0}^n \Omega_B^i \decal{-i}.
	\]
\end{pthm}

In this chapter, we explain the proof of Matsushita's result, to demonstrate in an
interesting special case how to use the symplectic form $\sigma$.

\newpar
Since $M$ is holomorphic symplectic, $\sigma^n$ gives an isomorphism
between $\shO_M$ and the canonical bundle $\Omega_M^{2n}$. From \eqref{eq:pil-Omega}
with $k=n$, we therefore get
\[
	\derR \pil \shO_M \cong  
	\derR \pil \Omega_M^{2n} \cong
	\bigoplus_{i=-n}^n \gr_{-n}^F \DR(\Pmod_i) \decal{-i}.
\]
In the special case where $M$ and $B$ are projective complex manifolds, this kind of
result was first proved by Koll\'ar \cite{KollarII}. In order to prove
\Cref{thm:Matsushita}, it is then enough to show that
$\gr_{-n}^F \DR(\Pmod_i) \cong \Omega_B^i$. (Note that $\gr_{-n}^F \DR(\Pmod_i) =
\gr_{-n}^F \Pmod_i$ is actually a sheaf, due to the fact that $F_{-n-1} \Pmod_i = 0$.) 

\newpar
We will deduce this from the isomorphism $\sigma^k \colon \Omega_M^{n-k} \to
\Omega_M^{n+k}$. Recall from \eqref{eq:forms} that
\[
	\derR \pil \Omega_M^{n+k} \decal{n-k} \cong \bigoplus_{i=-n}^{n} G_{i,k}.
\]
If we take cohomology in the lowest possible degree, this gives
\[
	\pil \Omega_M^{n+k} \cong \bigoplus_{i=0}^{n-k} \shH^{k-n} G_{-i,k}
\]
for every $k \geq 0$, due to the bound on the amplitude of the complex $G_{i,k}$ in
\Cref{lem:amplitude}. In a similar manner, we find (still for $k \geq 0$) that 
\[
	\pil \Omega_M^{n-k} \cong 
	\bigoplus_{i=k}^n \shH^{-k-n} G_{-i,-k}.
\]
From the fact that $\sigma^k \colon \Omega_M^{n-k} \to \Omega_M^{n+k}$ is an
isomorphism, we now deduce that
\begin{equation} \label{eq:sigma-k}
	\sigma^k \colon 
	\bigoplus_{i=k}^n \shH^{-k-n} G_{-i,-k}
	\to \bigoplus_{i=0}^{n-k} \shH^{k-n} G_{-i,k}
\end{equation}
must be an isomorphism for every $k \geq 1$. 

\newpar
Let us look at how this isomorphism acts on the summand with $i=n$.
Recall from \Cref{lem:sigma_2} that $\sigma = \sigma_1 + \sigma_0 + \dotsb$,
where the individual components are morphisms 
\[
	\sigma_j \colon G_{i,k} \to G_{i+j,k+2} \decal{2-j}
\]
in the derived category. By expanding $\sigma^k = (\sigma_1 + \sigma_0 + \dotsb)^k$,
we find that the component of highest degree is exactly the morphism
\begin{equation} \label{eq:sigma_1^k}
	\sigma_1^k \colon \shH^{-k-n} G_{-n,-k} \to \shH^{k-n} G_{-n+k,k}.
\end{equation}
All the other components of $\sigma^k$ go into the remaining summands (with $i <
n-k$) of the second sum. This puts us into the situation of the following abstract
lemma.

\begin{plem} \label{lem:sums-abelian}
	Let $A,B,C,D$ be objects in an abelian category. Suppose that we have morphisms
	$f \colon A \to C$, $g \colon B \to C$, and $h \colon B \to D$, such that
	\[
		\begin{tikzcd}[ampersand replacement=\&]
			A \oplus B \rar{\left( \begin{smallmatrix} f & g \\ 0 & h \end{smallmatrix}
			\right)} \& C \oplus D
		\end{tikzcd}
	\]
	is an isomorphism. Then $f$ is injective, $C \cong A \oplus \coker(f)$ and $B
	\cong \coker(f) \oplus D$.
\end{plem}

\begin{proof}
	It is easy to see that $f$ is injective and that the composition
	\[
		B \to C \oplus D \to \coker(f) \oplus D
	\]
	is an isomorphism. Now the composition
	\[
		\begin{tikzcd}[column sep=small]
			\coker(f) \rar & B \rar{g} & C \rar & \coker(f)
		\end{tikzcd}
	\]
	is the identity, and this gives us the desired splitting $C \cong A \oplus
	\coker(f)$.
\end{proof}

\newpar
We apply this lemma to the isomorphism in \eqref{eq:sigma-k}, letting $A$ be the
$n$-th summand in the first sum and $C$ the $(n-k)$-th summand in the second
sum. The result is that the morphism in \eqref{eq:sigma_1^k} is injective, and that
we have a direct sum decomposition
\begin{equation} \label{eq:Matsushita-coker}
	\shH^{k-n} G_{-n+k,k} \cong 
	\shH^{-k-n} G_{-n,-k} \oplus \coker(\sigma_1^k).
\end{equation}
Now $P_n$ is just the Hodge module $\QQ_B \decal{n}$, and so the underlying right
$\Dmod$-module is $\Pmod_n \cong \omega_B$, with the filtration for which $\gr_{-n}^F
\Pmod_n \cong \omega_B$. Together with the relative Hard Lefschetz isomorphism in
\eqref{eq:HL-omega}, this gives
\[
	G_{-n,-k} \cong G_{n,n-k} \decal{2n} = \gr_{k-n}^F \DR(\Pmod_n) \decal{n}
	\cong \Omega_B^{n-k} \decal{n+k},
\]
and therefore $\shH^{-k-n} G_{-n,-k} \cong \Omega_B^{n-k}$. For the same reason, we
have
\[
	G_{-n+k,k} \cong G_{n-k,n} \decal{2n-2k} 
\]
and so $\shH^{k-n} G_{-n+k,k} \cong \gr_{-n}^F \Pmod_{n-k} \cong R^{n-k} \pil
\omega_B$. This is a torsion-free coherent sheaf \cite[Thm.~2.1]{Kollar} of generic
rank $\binom{n}{n-k}$.
The reason is that over the smooth locus of the Lagrangian fibration, $P_i$ comes from the
variation of Hodge structure (of weight $i-n$) on the cohomology groups $H^{n+i}(M_b,
\QQ)(n)$, and $\gr_{-n}^F \Pmod_i$ is the tensor product of $\omega_B$ with the Hodge
bundle whose fibers are the subspaces $H^{n,i}(M_b)$ in the Hodge decomposition.
Since $M_b$ is a compact complex torus (and in fact an abelian variety) of dimension
$n$, this subspace has dimension $\binom{n}{i}$.

\newpar
Now the first and the second term in \eqref{eq:Matsushita-coker} have the same
generic rank, and because the left-hand side is torsion-free, it follows that
$\coker(\sigma_1^k) = 0$, which means that $\gr_{-n}^F \Pmod_{n-k} \cong
\Omega_B^{n-k}$, as required for the proof of \Cref{thm:Matsushita}. So we see that
Matsushita's theorem holds mostly for formal reasons: the only geometric facts that we
used are that $\pi$ is equidimensional and that the generic fiber of $\pi$ is a
compact complex torus of dimension $n$.  (Along the way,
we also proved the case $k=n$ of \Cref{thm:HL-sigma}.)

\newpar
By a similar method, one can still do the next case of \Cref{thm:HL-sigma}, namely that
\[
	\sigma_1^{n-1} \colon G_{i,-n+1} \to G_{i+n-1,n-1} \decal{2n-2}
\]
is an isomorphism for every $i \in \ZZ$. As before, it turns out that the two
complexes can differ only by the addition of a single locally free sheaf on $B$; and from the
geometric fact that the general fiber of $\pi$ is a compact complex torus of
dimension $n$, one can deduce that the rank of this locally free sheaf must be zero.
But in all other cases, just knowing what happens over the smooth locus of $\pi$ is
not enough.

\section{Proof of the symmetry conjecture}
\label{chap:symmetry}

\newpar
In this chapter, we investigate the relationship between \Cref{conj:symmetry} (which
predicts that $G_{i,k} \cong G_{k,i}$) and the ``symplectic relative Hard Lefschetz
theorem'' in \Cref{thm:HL-sigma}. Perhaps surprisingly, it turns out that the two
statements are basically equivalent. The proof below relies a more careful analysis
of the isomorphism between $\derR \pil \Omega_M^{n-k}$ and $\derR \pil
\Omega_M^{n+k}$. We are going to need the following small lemma, which
generalizes \Cref{lem:sums-abelian} to decompositions in the derived category.

\begin{plem} \label{lem:sums-derived}
	Let $A,B,C,D$ be objects in the derived category of an abelian category. Suppose
	that we have morphisms $f \colon A \to C$, $g \colon B \to C$, and $h \colon B \to
	D$, such that
	\[
		\begin{tikzcd}[ampersand replacement=\&]
			A \oplus B \rar{\left( \begin{smallmatrix} f & g \\ 0 & h \end{smallmatrix}
			\right)} \& C \oplus D
		\end{tikzcd}
	\]
	is an isomorphism. Then $C \cong A \oplus \Cone(f)$ and $B \cong \Cone(f) \oplus D$.
\end{plem}

\begin{proof}
	Recall that the mapping cone $\Cone(f)$ sits in a distinguished triangle
	\[
		\begin{tikzcd}[column sep=small] 
			A \rar{f} & C \rar & \Cone(f) \rar & A \decal{1}.
		\end{tikzcd}
	\]
	The assumptions imply that $f \colon A \to C$ is injective on cohomology, and so
	\[
		0 \to H^i(A) \to H^i(C) \to H^i \bigl( \Cone(f) \bigr) \to 0
	\]
	is short exact for every $i \in \ZZ$. Now consider the composition
	\[
		B \to C \oplus D \to \Cone(f) \oplus D.
	\]
	It is easy to see that this induces isomorphisms on cohomology, using the short
	exact sequence and the fact that $H^i(A) \oplus H^i(B) \cong H^i(C) \oplus
	H^i(D)$. Therefore $B \cong \Cone(f) \oplus D$, because we are in the derived
	category. The composition
	\[
		\begin{tikzcd}[column sep=small]
			\Cone(f) \rar & B \rar{g} & C \rar & \Cone(f)
		\end{tikzcd}
	\]
	is the identity, and this again gives us the desired splitting $C \cong A \oplus
	\Cone(f)$.
\end{proof}

\newpar
The following result relates \Cref{conj:symmetry} and \Cref{thm:HL-sigma}. Since
\[
	\derR \pil \Omega_M^{n+k} \decal{n} \cong \bigoplus_{i=-n}^n G_{i,k} \decal{k},
\]
this will also be useful later on, when we prove \Cref{thm:isomorphism}.

\begin{pthm} \label{thm:equivalence}
	The following three statements are equivalent:
	\begin{aenumerate}
	\item \label{thm:equivalence-a}
		The complexes $G_{i,k}$ and $G_{k,i}$ are isomorphic in the derived
		category (for all $i,k$).
	\item \label{thm:equivalence-b}
		For every $k = -n, \dotsc, n$, there is an isomorphism (in the derived category)
		\[
			\bigoplus_{i=-n}^n G_{i,k} \decal{i+k} \cong \bigoplus_{i=-n}^n G_{k,i}
			\decal{i+k}.
		\]
	\item \label{thm:equivalence-c}
		For every $k \geq 1$, the morphism $\sigma_1^k \colon G_{i,-k} \to G_{i+k,k}
		\decal{2k}$ is an isomorphism.
	\end{aenumerate}
\end{pthm}

Note that we are not making any assumptions about the isomorphisms in
\ref*{thm:equivalence-a} or \ref*{thm:equivalence-b},
only that they exist (in the derived category). Obviously, the statement in
\ref*{thm:equivalence-a} implies the statement in \ref*{thm:equivalence-b}.  As
explained in \Cref{par:hexagon}, the statement in
\ref*{thm:equivalence-c}, together with the relative Hard Lefschetz isomorphism
$G_{-i,k} \cong G_{i,i+k} \decal{2i}$, implies the statement in
\ref*{thm:equivalence-a}; in more detail,
\[
	G_{i,k} \cong G_{-i,k-i} \decal{-2i} \cong G_{-k,i-k} \decal{-2k} \cong G_{k,i}
\]
where the first and third isomorphism come from the relative Hard Lefschetz theorem,
and the second isomorphism from \ref*{thm:equivalence-c}. Now we only need to show
that the statement in \ref*{thm:equivalence-b} is strong enough to prove the one in
\ref*{thm:equivalence-c}. 

\newpar
We are going to prove, by descending induction on $0 \leq i \leq n$, that 
\[
	\sigma_1^k \colon G_{-i,-k} \to G_{-i+k,k} \decal{2k}
	\quad \text{and} \quad
	\sigma_1^k \colon G_{i-k,-k} \to G_{i,k} \decal{2k}
\]
are isomorphisms for every $1 \leq k \leq n$. (The two statements are equivalent to
each other by duality.) Our starting point is the isomorphism between the two complexes
\[
	\derR \pil \Omega_M^{n-k} \decal{n+k} \cong \bigoplus_{i=-n}^{n-k} G_{i,-k} 
	\quad \text{and} \quad
	\derR \pil \Omega_M^{n+k} \decal{n+k} \cong \bigoplus_{i=k-n}^{n} G_{i,k} \decal{2k}
\]
induced by $\sigma^k \colon \Omega_M^{n-k} \to \Omega_M^{n+k}$, for $1 \leq k \leq
n$. Consequently,
\begin{equation} \label{eq:isomorphism-n}
	\sigma^k \colon \bigoplus_{i=-n}^{n-k} G_{i,-k} \to 
	\bigoplus_{i=-n+k}^{n} G_{i,k} \decal{2k}
\end{equation}
is an isomorphism (in the derived category); the range of the summation is controlled
by \Cref{lem:amplitude}. As in the proof of Matsushita's theorem (in
\Cref{chap:Matsushita}), we can expand $\sigma^k = (\sigma_1 + \sigma_0 + \dotsb)^k$
into components; the component of highest degree is exactly the morphism
$\sigma_1^k \colon G_{i,-k} \to G_{i+k,k} \decal{2k}$, and the components of lower degree
map $G_{i,-k}$ into those summands in the second sum whose first index is $\leq i+k-1$. 

\newpar
For the sake of clarity, let me write $\mu_{i,-k} \colon
G_{i,-k} \to G_{i+k,k} \decal{2k}$ for the action by $\sigma_1^k$ on the $i$-th
summand in the first sum. Suppose by induction that we have already proved the
statement we want for all $i \in \{d+1, \dotsc, n\}$. Using \Cref{lem:sums-derived},
we can delete those terms from the isomorphism in \eqref{eq:isomorphism-n}, and still
get an isomorphism
\[
	\bigoplus_{i=-d}^{d-k} G_{i,-k} \to \bigoplus_{i=k-d}^d G_{i,k} \decal{2k};
\]
here the morphism is still basically $\sigma^k$, except that we have to remove those
components that go into summands with $i < k-d$ in the second sum. We now apply
\Cref{lem:sums-derived} again, taking $A = G_{-d,-k}$ and $B$ the rest of
the first sum, $C = G_{k-d,k} \decal{2k}$ and $D$ the rest of the second sum.
The conclusion is that we have (for every $k = 0, \dotsc, n$)
\begin{equation} \label{eq:mu-splitting-1}
	G_{k-d,k} \decal{2k} \cong G_{-d,-k} \oplus \Cone(\mu_{-d,-k}).
\end{equation}
By a similar argument, applied to the term with $i=d-k$, we also get
\begin{equation} \label{eq:mu-splitting-2}
	G_{d-k,-k} \cong G_{d,k} \decal{2k} \oplus \Cone(\mu_{d-k,-k}).
\end{equation}
Here is a schematic picture of the relevant morphisms:
\begin{center}
\begin{tikzpicture}[scale=0.7] 
	\draw[help lines] (-3,-3) grid (4,4);
	\draw[thick] (-3,-3) -- (-3,1) -- (-2,1) -- (-2,2) -- (-1,2) -- (-1,3) -- (0,3) -- (0,4) 
		-- (4,4) -- (4, 0) -- (3,0) -- (3,-1) -- (2,-1) -- (2,-2) -- (1,-2) -- (1,-3)
		-- cycle;
	\draw[thick,->] (-2.5,-2.5) -- (3.5, 0.5);
	\draw[thick,->] (-1.5,-2.5) -- (2.5, -0.5);
	\draw[thick,->] (-0.5,-2.5) -- (1.5, -1.5);
	\draw[thick,->] (-2.5,0.5) -- (3.5,3.5);
	\draw[thick,->] (-1.5,1.5) -- (2.5,3.5);
	\draw[thick,->] (-0.5,2.5) -- (1.5,3.5);
 	\draw (-3,-2.5) node[anchor=east] {$i=-d$};
 	\draw (-3,0.5) node[anchor=east] {$i=0$};
 	\draw (-3,3.5) node[anchor=east] {$i=d$};
 	\draw (3.5,-3) node[anchor=north] {$k$};
\end{tikzpicture}
\end{center}

\newpar
Using the relative Hard Lefschetz isomorphism, we can rewrite \eqref{eq:mu-splitting-1} as
\[
	G_{d-k,d} \decal{d-k} \cong G_{d,d-k} \decal{d-k} \oplus \Cone(\mu_{-d,-k})
	\decal{-d-k}.
\]
Now we take the sum over all $0 \leq k \leq n$ to produce an isomorphism
\begin{equation} \label{eq:iso-sum-1}
	\bigoplus_{k=d-n}^d G_{k,d} \decal{k} \cong \bigoplus_{k=d-n}^d G_{d,k} \decal{k}
	\oplus \bigoplus_{k=0}^n \Cone(\mu_{-d,-k}) \decal{-d-k}.
\end{equation}
Before we can apply the statement in \ref{thm:equivalence-b}, we have to find a way to add the
missing terms. When $k < d-n$, both complexes $G_{k,d}$ and $G_{d,k}$ are exact (by
\Cref{lem:amplitude}). When $k > d$, the inductive hypothesis, the relative Hard
Lefschetz theorem, and the isomorphism in \eqref{eq:mu-splitting-2} combine to give
us a chain of isomorphisms
\[
	G_{k,d} \cong G_{k-d,-d} \decal{-2d} \cong G_{d-k,-k} \decal{-2k}
	\cong G_{d,k} \oplus \Cone(\mu_{d-k,-k}) \decal{-2k}.
\]
If we add these terms to the sum in \eqref{eq:iso-sum-1}, and shift both sides by
$d$, we finally arrive at an isomorphism (in the derived category)
\[
	\bigoplus_{k=-n}^n G_{k,d} \decal{k+d} \cong \bigoplus_{k=-n}^n G_{d,k} \decal{k+d}
	\oplus \bigoplus_{k=0}^n \Cone(\mu_{-d,-k}) \decal{-k}
	\oplus \bigoplus_{k=d+1}^n \Cone(\mu_{d-k,-k}) \decal{d-k}.
\]
Since the first and second term are isomorphic by hypothesis, this is only possible
if $\Cone(\mu_{-d,-k}) = 0$ for every $0 \leq k \leq n$, which means exactly that
\[
	\sigma_1^k \colon G_{-d,-k} \to G_{k-d,k} \decal{2k}
\]
is an isomorphism for every $k = 0, \dotsc, n$. By duality, this shows that
$G_{d-k,-k} \cong G_{d,k} \decal{2k}$; but then \eqref{eq:mu-splitting-2} gives
$\Cone(\mu_{d-k,-k}) \cong 0$, and consequently, 
\[
	\sigma_1^k \colon G_{d-k,-k} \to G_{d,k} \decal{2k}
\]
is an isomorphism as well. This completes the induction, and hence the proof that
the statement in part (b) of \Cref{thm:equivalence} implies the one in (c).

\section{Lagrangian fibrations and differential forms}
\label{chap:differential-forms}

\newpar
This chapter contains some general considerations about $1$-forms and vector fields
on Lagrangian fibrations, related to Matsushita's isomorphism $R^1 \pil \shO_M
\cong \Omega_B^1$. The idea is that the symplectic form and the K\"ahler form
together allow us to transform holomorphic $1$-forms on the base $B$ of the
Lagrangian fibration into $\dbar$-closed $(0,1)$-forms on the holomorphic symplectic
manifold $M$. This construction is needed for the proof of \Cref{thm:isomorphism}.

\newpar
The symplectic form $\sigma \in H^0(M, \Omega_M^2)$ gives 
\[
	\bigoplus_{k=-n}^n \Omega_M^{n+k}
\]
the structure of a module over the Lie algebra $\sltwo$. Following one of several
competing conventions, we denote the three standard generators of $\sltwo$ by the
letters
\[
	\Xsl = \begin{pmatrix} 0 & 1 \\ 0 & 0 \end{pmatrix}, \quad
	\Ysl = \begin{pmatrix} 0 & 0 \\ 1 & 0 \end{pmatrix}, \quad
	\Hsl = \begin{pmatrix} 1 & 0 \\ 0 & -1 \end{pmatrix}.
\]
The relations are $[\Hsl, \Xsl] = 2 \Xsl$, $[\Hsl, \Ysl] = -2 \Ysl$, and
$[\Xsl, \Ysl] = \Hsl$. In our specific representation, the semisimple element $\Hsl$
acts as multiplication by the integer $k$ on the summand $\Omega_M^{n+k}$, and the
nilpotent element $\Xsl$ acts as wedge product with the holomorphic $2$-form
$\sigma$. By general theory, this can be lifted to a representation of the Lie group
$\SL_2(\CC)$, and the symmetry between $\Omega_M^{n+k}$ and $\Omega_M^{n-k}$ is best
expressed with the help of the Weil element
\[
	\wsl = e^{\Xsl} e^{-\Ysl} e^{\Xsl} 
	= \begin{pmatrix} 0 & 1 \\ -1 & 0 \end{pmatrix} \in \SL_2(\CC).
\]
The Weil element functions as a sort of linear algebra version of the Hodge
$\ast$-operator (from the Hodge theory of compact K\"ahler manifolds). It has the
property, easily checked by a small computation with $2 \times 2$-matrices, that
\[
	\wsl \Hsl \wsl^{-1} = -\Hsl, \quad \wsl \Xsl \wsl^{-1} = -\Ysl, \quad
	\wsl \Ysl \wsl^{-1} = -\Xsl.
\]
The resulting isomorphism between $\Omega_M^{n-k}$ and $\Omega_M^{n+k}$ works as follows.
Suppose that $\alpha \in \Omega_M^{n-k}$ is primitive, meaning that $\sigma^{k+1}
\wedge \alpha = 0$. Then for every $0 \leq j \leq k$, one has
\[
	\wsl \left( \frac{1}{j!} \sigma^j \wedge \alpha \right)
	= \frac{(-1)^j}{(k-j)!} \sigma^{k-j} \wedge \alpha.
\]
In particular, $\wsl$ acts on the primitive summand of $\Omega_M^{n-k}$ as
multiplication by $\frac{1}{k!} \sigma^k$; on the other summands in the Lefschetz
decomposition, it still acts basically as $\sigma^k$, but with
certain rational coefficients that make the whole operator behave better than just
$\sigma^k$ by itself. We will see right away why this is useful.

\newpar
Let $\beta \in H^0(U, \Omega_B^1)$ be a holomorphic $1$-form, defined on some open
subset $U \subseteq B$.
In order to keep the notation from getting complicated, let us agree (for the purposes of this
chapter) to replace the Lagrangian fibration $\pi \colon M \to B$ by its restriction
$\pi \colon \pi^{-1}(U) \to U$, so that we can work with $\beta \in H^0(B,
\Omega_B^1)$ and do not have to say ``restricted to $U$'' all the time. We may also
consider the
pullback $\piu \beta$ as a smooth $(1,0)$-form that satisfies $\dbar(\piu \beta) =
0$. Because the symplectic form $\sigma$ is nondegenerate, we get an associated
holomorphic vector field $\xi \in H^0(M, \shT_M)$, defined by the formula 
$\piu \beta = \xi \cont \sigma = \sigma(\xi, \argbl)$, where the symbol $\cont$ means
contraction by a vector field. 

\newpar
The Weil element induces an isomorphism
\[
	\wsigma \colon A^{n-k,q}(M) \to A^{n+k,q}(M)
\]
between the two spaces of smooth forms, and the action by $\wsigma$ exchanges wedge
product with $\piu \beta$ and contraction with the associated vector field $\xi$.

\begin{plem} \label{lem:contraction}
	For every $k,q \in \ZZ$, the following diagram commutes:
	\[
		\begin{tikzcd}
			A^{n-k,q}(M) \dar{\piu \beta \wedge} \rar{\wsigma} & A^{n+k,q}(M) \dar{\xi \cont} \\
			A^{n-k+1,q}(M) \rar{\wsigma} & A^{n+k-1,q}(M)
		\end{tikzcd}
	\]
\end{plem}

\begin{proof}
	Suppose first that $\alpha \in A^{n-k,q}(M)$ is primitive, so that $\sigma^{k+1}
	\wedge \alpha = 0$. Then
	\[
		\xi \cont \wsigma(\alpha) = \frac{1}{k!} \xi \cont (\sigma^k \wedge \alpha) 
		= \frac{1}{(k-1)!} \sigma^{k-1} \wedge \piu \beta \wedge \alpha 
			+ \frac{1}{k!} \sigma^k \wedge (\xi \cont \alpha).
	\]
	We still have $\sigma^{k+1} \wedge \piu \beta \wedge \alpha = 0$, and therefore
	\[
		\piu \beta \wedge \alpha = \gamma_0 + \sigma \wedge \gamma_1,
	\]
	where $\gamma_0 \in A^{n-k+1,q}(M)$ and $\gamma_1 \in A^{n-k-1,q}(M)$ are both
	primitive. In particular, we have $\sigma^k \wedge \piu \beta \wedge \alpha = \sigma^{k+1}
	\wedge \gamma_1$. If we contract the identity $\sigma^{k+1} \wedge \alpha = 0$
	with the holomorphic vector field $\xi$, we obtain
	\[
		(k+1) \sigma^k \wedge \piu \beta \wedge \alpha + \sigma^{k+1} \wedge (\xi \cont
		\alpha) = 0,
	\]
	and since $\sigma^{k+1} \colon A^{n-k-1,q}(M) \to A^{n+k+1,q}(M)$ is an
	isomorphism, it follows that $\xi \cont \alpha = -(k+1) \gamma_1$. We can now
	substitute into the identity from above and get
	\begin{align*}
		\xi \cont \wsigma(\alpha)
	   &= \frac{1}{(k-1)!} \sigma^{k-1} \wedge \gamma_0 + \frac{1}{(k-1)!} \sigma^k
		\wedge \gamma_1 - \frac{k+1}{k!} \sigma^k \wedge \gamma_1 \\
		& \frac{1}{(k-1)!} \sigma^{k-1} \wedge \gamma_0 - \frac{1}{k!} \sigma^k \wedge \gamma_1
		= \wsigma(\piu \beta \wedge \alpha).	
	\end{align*}
	This calculation explains why we need the Weil element (instead of just $\sigma^k$).

	The general case is almost the same. Because of the Lefschetz decomposition, it is
	enough to consider forms of the shape $\frac{1}{j!} \sigma^j \wedge \alpha$ where
	$\alpha \in A^{n-k,q}(M)$ is primitive and $0 \leq j \leq k$. As before, $\piu \beta
	\wedge \alpha = \gamma_0 + \sigma \wedge \gamma_1$ and $\xi \cont \alpha = -(k+1)
	\gamma_1$. This gives us
	\begin{align*}
		\xi \cont \wsigma \left( \frac{1}{j!} \sigma^j \wedge \alpha \right) 
		&= \frac{(-1)^j}{(k-j-1)!} \sigma^{k-j-1} \wedge \piu \beta \wedge \alpha
		+ \frac{(-1)^j}{(k-j)!} \sigma^{k-j} \wedge (\xi \cont \alpha) \\
		&= \frac{(-1)^j}{(k-j-1)!} \sigma^{k-j-1} \wedge \gamma_0
		+ \frac{(-1)^{j+1}(j+1)}{(k-j)!} \sigma^{k-j} \wedge \gamma_1.
	\end{align*}
	On the other hand, the relation $\piu \beta \wedge \alpha = \gamma_0 + \sigma \wedge
	\gamma_1$ implies that
	\begin{align*}
		\wsigma \left( \piu \beta \wedge \frac{1}{j!} \sigma^j \wedge \alpha \right) 
		&= \wsigma \left( \frac{1}{j!} \sigma^j \wedge \gamma_0 
			+ (j+1) \frac{1}{(j+1)!} \sigma^{j+1} \wedge \gamma_1 \right) \\
		&= \frac{(-1)^j}{(k-1-j)!} \sigma^{k-1-j} \wedge \gamma_0 + 
			(j+1) \frac{(-1)^j}{(k-j)!} \sigma^{k-j} \wedge \gamma_1.
	\end{align*}
	The two expressions match, and so the diagram commutes.
\end{proof}

\newpar
Since $\beta$ comes from $B$, the holomorphic vector field $\xi$ is tangent to the
fibers of $\pi$. To see why, let $x \in M$ be a point, and assume that the
fiber $M_b = \pi^{-1}(b)$ is smooth at $x$, where $b = \pi(x)$.  Since $\pi$ is a
Lagrangian fibration, the tangent space $T_x M_b$ is a maximal isotropic subspace of
$T_x M$.  For any holomorphic tangent vector $\eta \in T_x M_b$, we have
\[
	\sigma(\xi, \eta) = (\piu\beta) (\eta) = 0,
\]
because $\piu \beta$ vanishes along the fibers of $\pi$. By maximality, we must have
$\xi \in T_x M_b$, and so $\xi$ is indeed tangent to the fiber at the point $x$.

\newpar
From the holomorphic vector field $\xi$, we can further construct a $(0,1)$-form on
$M$ with the help of the K\"ahler form $\omega$. We define $\theta = - \xi \cont \omega
= -\omega(\xi, \argbl)$; the minus sign is justified by \Cref{lem:contraction-omega}
below. As one would expect, the $(0,1)$-form $\theta$ is $\dbar$-closed, and
therefore defines a class in $H^1(M, \shO_M)$.

\begin{plem}
	We have $\dbar \theta = 0$.
\end{plem}

\begin{proof}
	This is easiest if we use the explicit formula for the exterior derivative. According
	to this formula, for any two smooth $(0,1)$-vector fields $\lambda$ and $\mu$, one
	has
	\[
		\bigl( \dbar \theta \bigr)(\lambda, \mu) 
		= \lambda \cdot \theta(\mu) - \mu \cdot \theta(\lambda) -
		\theta \bigl( \lbrack \lambda, \mu \rbrack \bigr)
		= \mu \cdot \omega(\xi, \lambda) -\lambda \cdot \omega(\xi, \mu) 
		+ \omega \bigl( \xi, \lbrack \lambda, \mu \rbrack \bigr).
	\]
	Because $d \omega = 0$, we have
	\[
		\xi \cdot \omega(\lambda, \mu) + \lambda \cdot
		\omega(\mu, \xi) + \mu \cdot \omega(\xi, \lambda)
		+ \omega \bigl( \xi, [\lambda, \mu] \bigr)
		+ \omega \bigl( \lambda, [\mu, \xi] \bigr)
		+ \omega \bigl( \mu, [\xi, \lambda] \bigr) = 0,
	\]
	and so we can rewrite the right-hand side in the form
	\[
		\omega \bigl( \lambda, \lbrack \xi, \mu
		\rbrack \bigr) - \omega \bigl( \mu, \lbrack \xi, \lambda \rbrack \bigr)
		-\xi \cdot \omega(\lambda, \mu).
	\]
	The third term vanishes because $\omega$ has type $(1,1)$, and the other two terms
	vanish because $\xi$ is holomorphic. Therefore $\dbar \theta = 0$.
\end{proof}

\newpar
The K\"ahler form $\omega$ determines another representation of the Lie algebra
$\sltwo$, and the resulting Weil element induces an isomorphism
\[
	\womega \colon A^{n-p,n-q}(M) \to A^{n+q,n+p}(M)
\]
between the two spaces of smooth forms. This time, it exchanges contraction with the
vector field $\xi$ and wedge product with the $(0,1)$-form $\theta = -\xi \cont
\omega$. 

\begin{plem} \label{lem:contraction-omega}
	For every $p,q \in \ZZ$, the following diagram commutes:
	\[
		\begin{tikzcd}
			A^{n-p,n-q}(M) \dar{\xi \cont} \rar{\womega} & A^{n+q,n+p}(M) \dar{\theta
			\wedge} \\
			A^{n-p-1,n-q}(M) \rar{\womega} & A^{n+q,n+p+1}(M)
		\end{tikzcd}
	\]
\end{plem}

\begin{proof}
	Suppose first that $\alpha \in A^{n-p,n-q}(M)$ is primitive, meaning that
	$\omega^{k+1} \wedge \alpha = 0$, where $k = p+q \geq 0$. Contraction with $\xi$
	gives
	\begin{equation} \label{eq:xi-theta}
		0 = \xi \cont (\omega^{k+1} \wedge \alpha) 
		= - (k+1) \omega^k \wedge \theta \wedge \alpha 
		+ \omega^{k+1} \wedge (\xi \cont \alpha).
	\end{equation}
	It follows that $\omega^{k+2} \wedge (\xi \cont \alpha) = 0$, and therefore $\xi
	\cont \alpha \in A^{n-p-1,n-q}(M)$ is also primitive. Since we know how the Weil
	element acts on primitive forms, we then compute that
	\[
		\womega (\xi \cont \alpha) 
		= \frac{1}{(k+1)!} \omega^{k+1} \wedge (\xi \cont \alpha) 
		= \frac{1}{k!} \omega^k \wedge \theta \wedge \alpha 
		= \theta \wedge \womega(\alpha).
	\]
	The general case is only marginally harder. Because of the Lefschetz
	decomposition, it is enough to consider forms of the shape $\frac{1}{j!} \omega^j
	\wedge \alpha$ where $\alpha \in A^{n-p,n-q}(M)$ is primitive and $0 \leq j \leq k$.
	We already know that $\xi \wedge \alpha \in A^{n-p-1,n-q}(M)$ is primitive.
	At the same time, $\theta \wedge \alpha \in A^{n-p,n-q+1}(M)$ is annihilated by
	$\omega^{k+1}$, and so it has a Lefschetz decomposition
	\[
		\theta \wedge \alpha = \gamma_0 + \omega \wedge \gamma_1,
	\]
	where $\gamma_0 \in A^{n-p,n-q+1}(M)$ and $\gamma_1 \in A^{n-p-1,n-q}(M)$ are
	primitive. From \eqref{eq:xi-theta}, we get
	\[
		\omega^{k+1} \wedge (\xi \cont \alpha)
		= (k+1) \omega^k \wedge \theta \wedge \alpha
		= (k+1) \omega^{k+1} \wedge \gamma_1,
	\]
	which implies that $\xi \cont \alpha = (k+1) \gamma_1$. Consequently,
	\begin{align*}
		\xi \cont \frac{1}{j!} \omega^j \wedge \alpha
		&= -\frac{1}{(j-1)!} \omega^{j-1} \wedge \theta \wedge \alpha
		+ \frac{1}{j!} \omega^j \wedge (\xi \cont \alpha) \\
		&= -\frac{1}{(j-1)!} \omega^{j-1} \wedge \gamma_0
		+ (k+1-j) \frac{1}{j!} \omega^j \wedge \gamma_1.
	\end{align*}
	If we now apply the Weil element $\womega$, we obtain
	\begin{align*}
		\womega \left( \xi \cont \frac{1}{j!} \omega^j \wedge \alpha \right)
		&= \frac{(-1)^j}{(k-j)!} \omega^{k-j} \wedge \gamma_0
		+ \frac{(-1)^j}{(k-j)!} \omega^{k+1-j} \wedge \gamma_1 \\
		&= \frac{(-1)^j}{(k-j)!} \omega^{k-j} \wedge \theta \wedge \alpha
		= \theta \wedge \womega \left( \frac{1}{j!} \omega^j \wedge \alpha \right).
	\end{align*}
	This proves that the diagram is commutative.
\end{proof}

\newpar
We are really interested in the operation of taking wedge product with $\theta$.
Since
\begin{equation} \label{eq:theta-commutator}
	\theta \wedge \alpha = -(\xi \cont \omega) \wedge \alpha = 
	\omega \wedge (\xi \cont \alpha) -\xi \cont (\omega \wedge \alpha),
\end{equation}
we can realize this as the commutator of the Lefschetz operator $\omega \wedge$ and
the contraction operator $\xi \cont$. (This formula shows one more time why we need
the minus sign in $\theta = - \xi \cont \omega$.)

\newpar 
Let us now relate this construction to the decomposition theorem and to the
complexes $G_{i,k}$. Let $\algS^{\bullet} = \Sym^{\bullet}(\shT_B)$. Recall from
\parref{par:Laumon} that we have an isomorphism 
\[
	\gr_{\bullet}^F \Cpi \cong \bigoplus_{i=-n}^n \gr_{\bullet}^F \Pmod_i \decal{-i}
\]
in the derived category of graded $\algS$-modules, where the decomposition is induced
by the one in the decomposition theorem. Here $\gr_{\bullet}^F \Cpi$ is the complex
of graded $\algS$-modules with
\[
	\gr_{\bullet}^F \Cpi^i = \bigoplus_{p+q=i} 
	\pil \shA_M^{2n+p,q} \tensor_{\shO_M} \algS^{\bullet+n+p};
\]
the differential in the complex is given by the formula
\[
	\dpi(\alpha \tensor P) = \dbar \alpha \tensor D 
	+ \sum_{j=1}^n \piu(\dtj) \wedge \alpha \tensor \partial_j D,
\]
where $t_1, \dotsc, t_n$ are local holomorphic coordinates on $B$, and $\partial_j =
\partial/\partial t_j$. Since $\dbar \beta=0$, wedge product with the
holomorphic $1$-form $\piu \beta$ defines a morphism of complexes
\[
	\piu \beta \colon \gr_{\bullet}^F \Cpi \to \gr_{\bullet-1}^F \Cpi \decal{1}
\]
that increases the cohomological degree by $1$ and decreases the degree with respect
to the grading by $1$. This gives us a morphism (in the derived category)
\[
	\piu \beta \colon \bigoplus_{i=-n}^n \gr_{\bullet}^F \Pmod_i \decal{-i},
	\to \bigoplus_{i=-n}^n \gr_{\bullet-1}^F \Pmod_i \decal{1-i}.
\]
As usual, we break this up into components $\piu \beta = (\piu \beta)_1 + (\piu
\beta)_0 + \dotsb$, where each 
\[
	(\piu \beta)_k \colon \gr_{\bullet}^F \Pmod_i \to \gr_{\bullet-1}^F \Pmod_{i+k}
	\decal{1-k}
\]
is a morphism in the derived category of graded $\algS$-modules. Not surprisingly,
the topmost component is zero, due to the fact that $\piu \beta$ vanishes on the
fibers of $\pi$.

\begin{plem} \label{lem:piu-beta}
	We have $(\piu \beta)_1 = 0$.
\end{plem}

\begin{proof}
	Let us first treat the case when $d \beta = 0$. Wedge product with $\piu
	\beta$ defines a morphism
	\[
		\piu \beta \colon \Cpi \to \Cpi \decal{1}
	\]
	that increases the cohomological degree by $1$ and decreases the degree with
	respect to the filtration by $1$. This follows from the fact that 
	\[
		\Cpi^i = \bigoplus_{p+q=i} 
		\pil \shA_M^{n+p,n+q} \tensor_{\shO_M} (\Dmod_B, F_{\bullet+p} \Dmod_B),
	\]
	and that the differential in the complex is
	\[
		\dpi \colon \Cpi^i \to \Cpi^{i+1}, \quad
		\dpi \bigl( \alpha \tensor D \bigr) = d \alpha \tensor D + \sum_{j=1}^n \piu(\dtj) 
		\wedge \alpha \tensor \partial_j D.
	\]
	The topmost component is now a morphism $(\piu \beta)_1 \colon (\Pmod_i,
	F_{\bullet} \Pmod_i) \to (\Pmod_{i+1}, F_{\bullet-1} \Pmod_{i+1})$. Since $\piu
	\beta$ vanishes on the fibers of $\pi$, we can then show by exactly the same
	argument as in \Cref{lem:sigma_2} that $(\piu \beta)_1 = 0$. 

	To deal with the general case, we note that the problem is local on $B$, due
	to the fact that $(\piu \beta)_1$ is a morphism between two graded
	$\algS$-modules. Working locally, we can choose holomorphic coordinates $t_1,
	\dotsc, t_n$, and write $\beta = \sum_{j=1}^n f_j \dtj$. Since the entire
	construction is $\shO_B$-linear, we have $(\piu \beta)_1 = \sum_{j=1}^n f_j (\piu
	\dtj)_1$; but $(\piu \dtj)_1 = 0$ because $\dtj$ is closed.
\end{proof}

\newpar
We can use the properties of Deligne's decomposition to get a much better result.
(For a more direct proof, see \Cref{par:action-forms}.)

\begin{plem} \label{lem:beta_k}
	For any holomorphic form $\beta \in H^0(B, \Omega_B^1)$, one has $(\piu \beta)_k =
	0$ for all $k \neq 0$, and $(\piu \beta)_0$ commutes with $\omega_2$. 
\end{plem}

\begin{proof}
	We already know that $\piu \beta = (\piu \beta)_0 + (\piu \beta)_{-1} + \dotsb$.
	Recall from \Cref{par:Deligne} that Deligne's decomposition has the property that
	$\omega = \omega_2 + \omega_0 + \omega_{-1} + \dotsb$, and that the individual
	components $\omega_k$ with $k \leq 0$ are primitive (with respect to the
	$\sltwo$-representation determined by $\omega_2$.) Since $\piu \beta$ and $\omega$
	commute (as forms on $M$), the corresponding operators satisfy the relation
	$\lbrack \omega, \piu \beta \rbrack = 0$. If we expand this, we get
	\[
		\bigl[ \omega_2, (\piu \beta)_0 \bigr] = 0, \quad
		\bigl[ \omega_2, (\piu \beta)_{-1} \bigr] = 0, \quad
		\bigl[ \omega_2, (\piu \beta)_{-2} \bigr] 
		+ \bigl[ \omega_0, (\piu \beta)_0 \bigr] = 0,
	\]
	and so on. The first relation shows that $(\piu \beta)_0 \in \ker(\ad \omega_2)$;
	since $(\piu \beta)_0$ has weight $0$ (with respect to the
	$\sltwo$-representation), it must be primitive. The second relation
	gives $(\piu \beta)_{-1} \in \ker(\ad \omega_2)$, and because $(\piu \beta)_{-1}$ has
	weight $-1$, it follows that $(\piu \beta)_{-1} = 0$. 

	By induction, we can assume that we already have $(\piu \beta)_{-1} = \dotsb = (\piu
	\beta)_{-k+1} = 0$ for some $k \geq 2$. Let us prove that $(\piu \beta)_{-k} = 0$.
	From the relation $[\omega, \piu \beta] = 0$, we get
	\[
		\bigl[ \omega_2, (\piu \beta)_{-k} \bigr] 
		+ \bigl[ \omega_{2-k}, (\piu \beta)_0 \bigr] = 0.
	\]
	Since $\omega_{2-k}$ is primitive of weight $-k+2$, it satisfies $(\ad
	\omega_2)^{k-1} \omega_{2-k} = 0$. This gives
	\[
		(\ad \omega_2)^k (\piu \beta)_{-k} 
		= -(\ad \omega_2)^{k-1} \bigl[ \omega_{2-k}, (\piu \beta)_0 \bigr] 
		= - \bigl[ (\ad \omega_2)^{k-1} \omega_{2-k}, (\piu \beta)_0 \bigr] = 0,
	\]
	as $(\piu \beta)_0 \in \ker(\ad \omega_2)$. Since $(\piu
	\beta)_{-k}$ has weight $-k$, it follows that $(\piu \beta)_{-k} = 0$.
\end{proof}

\newpar \label{par:claim-beta}
This result means concretely that the action by $\piu \beta$ on 
\[
	\gr_{\bullet}^F \Cpi \cong \bigoplus_{i=-n}^n \gr_{\bullet}^F \Pmod_i \decal{-i}
\]
is diagonal (provided that we use Deligne's decomposition). The individual morphisms
\[
	(\piu \beta)_0 \colon \gr_{\bullet}^F \DR(\Pmod_i) \to \gr_{\bullet-1}^F \DR(\Pmod_i)
	\decal{1}
\]
are of course just given by the action of $\beta \in H^0(B, \Omega_B^1)$ on the
graded pieces of the de Rham complex. Indeed, for any filtered $\Dmod$-module
$(\Pmod, F_{\bullet} \Pmod)$, we have a morphism of complexes
\[
	\Omega_B^1 \tensor \gr_k^F \DR(\Pmod) \to \gr_{k-1}^F \DR(\Pmod) \decal{1},
\]
which is defined (in local coordinates) by the formula
\[
	\Omega_B^1 \tensor \gr_{k+i}^F \Pmod \tensor \wed^{-i} \shT_B
	\to \gr_{k+i}^F \Pmod \tensor \wed^{-i-1} \shT_B, \quad
	\dt_j \tensor s \tensor \partial_J \mapsto \sgn(J,j) s \tensor \partial_{J
	\setminus \{j\}}.
\]
We will prove this claim in \Cref{par:claim-beta} below. 

\newpar
We can analyze the action of the holomorphic vector field $\xi$ in much the same way.
First, we observe that contraction with $\xi$ induces a morphism 
\[
	\xi \colon \gr_{\bullet}^F \Cpi \to \gr_{\bullet+1}^F \Cpi \decal{-1}.
\]
To see why this is the case, let $t_1, \dotsc, t_n$ be local holomorphic coordinates
on $B$, and let $\eta_j$ denote the holomorphic vector field
associated to the form $\piu(\dtj)$, so that $\piu(\dtj) = \eta_j \cont \sigma$.
Then the compatibility with the differential in the complex comes down to the
identity
\[
	\xi \cont \piu(\dtj) = \sigma(\eta_j, \xi) = 0,
\]
which holds because both $\eta_j$ and $\xi$ are tangent to the fibers of $\pi$. We
therefore get another morphism (in the derived category) that we denote by the symbol
\[
	\xi \colon \bigoplus_{i=-n}^n \gr_{\bullet}^F \Pmod_i \decal{-i},
	\to \bigoplus_{i=-n}^n \gr_{\bullet+1}^F \Pmod_i \decal{-i-1}.
\]
As before, we decompose this into components $\xi = \xi_{-1} + \xi_0 + \dotsb$, where
each component is a morphism $\xi_k \colon \gr_{\bullet}^F \Pmod_i \to
\gr_{\bullet+1}^F \Pmod_{i+k} \decal{-1-k}$. This time, it is much less obvious that
all the components with $k \leq 0$ have to vanish.

\begin{plem} \label{lem:xi_k}
	We have $\xi_k = 0$ for every $k \neq -1$, and $(\ad \omega_2)^2(\xi_{-1}) = 0$.
\end{plem}

\begin{proof}
	By \eqref{eq:theta-commutator}, the commutator of the Lefschetz operator and
	contraction with $\xi$ is equal to wedge product with the $(0,1)$-form
	$\theta$. Since $\omega$ and $\theta$ commute (as forms on $M$), it follows that
	$\bigl\lbrack \omega, \lbrack \omega, \xi \rbrack \bigr\rbrack = 0$. After
	decomposing this relation by degree, we first get
	\[
		(\ad \omega_2)^2(\xi_{-1}) = 0,
	\]
	which means that $\xi_{-1}$ is primitive (with respect to the
	representation of $\sltwo$ determined by $\omega_2$). Next, we get $(\ad
	\omega_2)^2(\xi_{-2}) = 0$, and since $\xi_{-2}$ has weight $-2$, it follows that
	$\xi_{-2} = 0$. 

	By induction, we may again assume that $\xi_{-2} = \dotsb = \xi_{-k+1} = 0$ for
	some $k \geq 3$. Let us prove that $\xi_{-k} = 0$. From the relation $\bigl\lbrack
	\omega, \lbrack \omega, \xi \rbrack \bigr\rbrack = 0$, we obtain
	\[
		(\ad \omega_2)^2(\xi_{-k}) + \ad \omega_2 \ad \omega_{3-k} (\xi_{-1})  
		+ \ad \omega_{3-k} \ad \omega_2(\xi_{-1}) = 0.
	\]
	Recall that $\omega_{3-k}$ is primitive of weight $3-k$, which means that $(\ad
	\omega_2)^{k-2} \omega_{3-k} = 0$. If we apply the operator $(\ad \omega_2)^{k-2}$
	to the identity above, it becomes
	\begin{align*}
		(\ad \omega_2)^k(\xi_{-k}) &= - (\ad \omega_2)^{k-1} \ad \omega_{3-k} (\xi_{-1})  
		- (\ad \omega_2)^{k-2} \ad \omega_{3-k} \ad \omega_2(\xi_{-1}) \\
		&= - (\ad \omega_2)^{k-1} \bigl[ \omega_{3-k}, \xi_{-1} \bigr] 
		- (\ad \omega_2)^{k-2} \bigl[ \omega_{3-k}, \ad \omega_2(\xi_{-1}) \bigr] = 0.
	\end{align*}
	In the last line, we used the fact that $\ad \omega_2$ is a derivation, and that
	$\omega_{3-k} \in \ker(\ad \omega_2)^{k-2}$ and $\xi_{-1} \in \ker(\ad
	\omega_2)^2$. Because $\xi_{-k}$ has weight $-k$, this is enough to conclude that
	$\xi_{-k} = 0$.
\end{proof}

\newpar
Both results illustrate why Deligne's decomposition is ``less bad'' than the other
possible choices in the decomposition theorem. In our setting,  it has the nice
effect of making a holomorphic $1$-form on $B$ act entirely in the ``horizontal''
direction on the summands in the
decomposition theorem, whereas the associated holomorphic vector field acts entirely
in the ``vertical'' direction (meaning along the fibers), just as one would expect
from the geometry of a Lagrangian fibration. I doubt that any other choice of
decomposition has this property.

\newpar
Let us conclude this chapter by recording the effect of the various operators on the
complexes $G_{i,k}$. We are going to make a slight change in the notation and assume
from now on that $\beta \in H^0(U, \Omega_B^1)$ is a holomorphic $1$-form, defined on
an open subset $U \subseteq B$. As pointed out at the beginning of the chapter, this
causes no problems, because all the arguments can be used locally on $B$. With the
help of the symplectic form $\sigma$ and the K\"ahler form $\omega$, the holomorphic
$1$-form $\beta$ gives us on the one hand a holomorphic vector field
\[
	v(\beta) \in H^0 \bigl( \pi^{-1}(U), \shT_M \bigr), \quad
	\piu \beta = v(\beta) \cont \sigma,
\]
and on the other hand a $\dbar$-closed $(0,1)$-form (with a minus sign)
\[
	f(\beta) \in A^{0,1} \bigl( \pi^{-1}(U) \bigr), \quad
	f(\beta) = -v(\beta) \cont \omega.
\]
We use the same symbols for the morphisms
\[
	\piu \beta \colon \gr_{\bullet}^F \Cpi \to \gr_{\bullet-1}^F \Cpi \decal{1}
	\quad \text{and} \quad
	f(\beta) \colon \gr_{\bullet}^F \Cpi \to \gr_{\bullet}^F \Cpi \decal{1}
\]
induced by wedge product with the forms $\piu \beta$ and $f(\beta)$, and for the
morphism
\[
	v(\beta) \colon \gr_{\bullet}^F \Cpi \to \gr_{\bullet+1}^F \Cpi \decal{-1}
\]
induced by contraction with the vector field $v(\beta)$. All these morphisms are of
course defined only on the open set $U \subseteq B$. We saw in
\eqref{eq:theta-commutator} that $f(\beta) = \lbrack \omega, v(\beta)
\rbrack$. Because of \Cref{lem:xi_k}, the only nonzero component of $v(\beta)$ is
$v(\beta)_{-1}$, which means that the components of $f(\beta)$ can be computed by the
simple formula
\[
	f(\beta)_k = \lbrack \omega_{k+1}, v(\beta) \rbrack.
\]
In particular, the topmost component is $f(\beta)_1 = \lbrack \omega_2, v(\beta) 
\rbrack$. 

\newpar
From $\piu \beta = (\piu \beta)_0 \colon \gr_{\bullet}^F \Pmod_i \to \gr_{\bullet-1}^F
\Pmod_i \decal{1}$, we get a morphism of complexes
\[
	\piu \beta \colon G_{i,k} \to G_{i,k+1} \decal{1},
\]
which is given by contracting the holomorphic vector fields in the de Rham complex
against the holomorphic form $\beta$.  Similarly, from $v(\beta) = v(\beta)_{-1}
\colon \gr_{\bullet}^F \Pmod_i \to \gr_{\bullet+1}^F \Pmod_{i-1}$, we get a second
morphism of complexes
\[
	v(\beta) \colon G_{i,k} \to G_{i-1,k-1} \decal{-1}.
\]
Lastly, we have the morphism of complexes
\[
	f(\beta)_1 = \lbrack \omega_2, v(\beta) \rbrack \colon G_{i,k} \to G_{i+1,k} \decal{1}.
\]
All three morphisms will play an important role in the proof of \Cref{thm:isomorphism}.

\newpar
The relation with Matsushita's theorem (in \Cref{thm:Matsushita}) is the following.
The construction above produces a morphism of sheaves
\[
	\Omega_B^1 \to R^1 \pil \shO_M,
\]
by assigning to a holomorphic $1$-form $\beta \in H^0(U, \Omega_B^1)$ the image of
the cohomology class of the $\dbar$-closed $(0,1)$-form $f(\beta)$ under the morphism
$H^1 \bigl( \pi^{-1}(U), \shO_M \bigr) \to H^0(U, R^1 \pil \shO_M)$. Matsushita's
theorem is then saying concretely that this morphism is an isomorphism.

\section{The BGG correspondence}
\label{chap:BGG}

\newpar
This chapter contains a brief review of the BGG correspondence \cite{BGG,EFS}, in a way
that is convenient for our purposes. We will define everything carefully, with the
correct signs, but only sketch the proofs, which mostly amount to checking that the
signs in front of various terms match up correctly; the details can in any case be
found in \cite{EFS}.

\newpar
Let $B$ be a complex manifold of dimension $n$, and denote by $\shT_B$ the
holomorphic tangent sheaf. The BGG correspondence relates, on the level of the
derived category, graded modules over two graded $\shO_B$-algebras. The first is the
symmetric algebra
\[
	\algS = \Sym(\shT_B) = \bigoplus_{j \in \NN} \Sym^j(\shT_B),
\]
with the obvious multiplication and the natural grading in which $\Sym^j(\shT_B)$ has
degree $j$. The second is the algebra of holomorphic forms
\[
	\algE = \bigoplus_{j=0}^n \Omega_B^j,
\]
with multiplication given by wedge product. Note that, unlike \cite{EFS}, we give
the algebra $\algE$ the naive grading in which $\Omega_B^j$ lives in degree
$j$. If $M = M_{\bullet}$ is a graded module over either of these algebras, we denote
by $M(d)$ the same module with the grading shifted according to the rule $M(d)_k =
M_{d+k}$. In the case of $\algE$, we only work with \emph{left} modules. 

\newpar
The central object in the BGG correspondence is the Koszul complex
\begin{equation} \label{eq:Koszul}
\begin{tikzcd}[column sep=small]
	\algS(-n) \tensor \wed^n \shT_B \rar{\delta} & \dotsb
	\rar{\delta} & 
	\algS(-1) \tensor \shT_B \rar{\delta} & \algS(0) \tensor \shO_B,
\end{tikzcd}
\end{equation}
which lives in cohomological degrees $-n, \dotsc, 0$, and gives a free resolution of the
trivial graded $\algS$-module $\shO_B$. The differential is defined by the following
compact formula:
\[
	\delta(s \tensor \partial_J) = \sum_{j=1}^n \sgn(J,j) \cdot \partial_j s \tensor
	\partial_{J \setminus \{j\}}.
\]
Here the notation is as follows. Let $t_1, \dotsc, t_n$ be local holomorphic coordinates on
$B$, and denote by $\partial_j = \partial/\partial t_j$ the resulting holomorphic
vector fields. For any subset $J \subseteq \{1, \dotsc, n\}$, we list the elements
in increasing order as $j_1 < \dotsb < j_{\ell}$, and then define
\[
	\partial_J = \partial_{j_1} \wedge \dotsb \wedge \partial_{j_{\ell}}
	\quad \text{and} \quad
	\dtJ = \dt_{j_1} \wedge \dotsb \wedge \dt_{j_{\ell}},
\]
with the convention that both expressions equal $1$ when $J$ is empty. We also define 
\[
	\sgn(J, j) = \begin{cases}
		(-1)^{k-1} &\text{if $j = j_k$,} \\
		0 & \text{if $j \not\in J$.}
	\end{cases}
\]
Note that we are always using Deligne's Koszul sign rule, according to which
swapping two elements of degrees $p$ and $q$ leads to a sign $(-1)^{pq}$; this is the
reason for the factor $\sgn(J,j)$.

\newpar
Let us start by defining a functor $\BGGL$ from complexes of graded $\algE$-modules
to complexes of graded $\algS$-modules. Let $(M, d)$ be a complex of graded left
$\algE$-modules. We can think of this concretely as a bigraded sheaf of
$\shO_B$-modules
\[
	M = \bigoplus_{i,k} M_k^i,
\]
where $i$ is the cohomological degree and $k$ the degree with respect to the grading;
the differential $d$ maps each $M_k^i$ to $M_k^{i+1}$ and is linear over $\algE$. The
idea is to send this to the induced $\algS$-module $M \otimes_{\shO_B} \algS$, but where
we put the summand $M_k^i \tensor \algS^j$ into cohomological degree $i+k$ and in
degree $j-k$ with respect to the grading.\footnote{The graded $\algS$-modules that we are
interested in come from filtered right $\Dmod_B$-modules, and so they are
naturally \emph{right} $\algS$-modules. Since $\algS$ is commutative, we do not need
to distinguish between left and right modules, but the signs work out more nicely if
we put the factor $\algS$ on the right.} More precisely, we define $\BGGL(M, d)$ to
be the complex whose $i$-th term is the graded $\algS$-module 
\[
	\BGGL(M, d)^i = \bigoplus_{p+q=i} M_p^q \tensor \algS(p),
\]
and whose differential acts on the summand $M_p^q \tensor \algS(p)$ by the
formula
\[
	\sum_{j=1}^n \dt_j \tensor \partial_j + (-1)^p d \tensor \id.
\]
This is the simple complex associated to the double complex with terms $M_p^q
\tensor \algS(p)$, with Deligne's sign rules for the differential. Since $\algS$ is
commutative, the differential is $\algS$-linear and preserves the grading. 

\begin{pexa}
For instance, $\BGGL(\algE)$ is the complex of graded $\algS$-modules
\[
	\shO_B \tensor \algS(0) \to \Omega_B^1 \tensor \algS(1) \to \dotsb \to
	\Omega_B^n \tensor \algS(n),
\]
with differential $\alpha \tensor s \mapsto \sum_j \dtj \wedge \alpha \tensor
\partial_j s$. Up to a factor of $(-1)^n$, this is just the Koszul resolution of the
trivial $\algS$-module $\omega_B$, placed in cohomological degree $n$ and with the
grading shifted by $n$ steps; in other words, $\BGGL(\algE) \cong \omega_B(n) \decal{-n}$.
\end{pexa}

\newpar
We denote by $G(\algS)$ the category of graded $\algS$-modules, and by
$\Dbcoh G(\algS)$ the derived category of cohomologically bounded and coherent
complexes of graded $\algS$-modules; we use similar notation for $\algE$, with
the understanding that modules over the noncommutative algebra $\algE$ are always
\emph{left} modules. One checks that $\BGGL$ descends to an exact functor
\[
	\BGGL \colon \Dbcoh G(\algE) \to \Dbcoh G(\algS)
\]
between the two derived categories. Indeed, a morphism $f \colon (M,d) \to (M',d')$
between two complexes of graded $\algE$-modules clearly induces a morphism of
complexes 
\[
	\BGGL(f) \colon \BGGL(M,d) \to \BGGL(M',d'). 
\]
The point is that if $f$ is a quasi-isomorphism, then $\BGGL(f)$ is also a
quasi-isomorphism: when $(M,d)$ and $(M',d')$ are cohomologically bounded and coherent,
this can easily be checked by a spectral sequence argument.

\newpar
For later use, let us see how the complex $\gr_{\bullet}^F \Cpi$ fits into this
framework.

\begin{pexa} \label{ex:grFCpi}
Recall that $\gr_{\bullet}^F \Cpi$ is the complex of graded $\algS$-modules with terms
\[
	\gr_{\bullet}^F \Cpi^i = \bigoplus_{p+q=i} \pil \shA_M^{n+p,n+q} \tensor_{\shO_M}
	\algS(p)
\]
and with differential (written in local coordinates)
\[
	\dpi \bigl( \alpha \tensor P \bigr) = \dbar \alpha \tensor P + 
	\sum_{j=1}^n \piu(\dtj) \wedge \alpha \tensor \partial_j P.
\]
If we compare this with the definition above, we see that this is exactly
$\BGGL(M,d)$, where $(M,d)$ is the complex of graded $\algE$-modules with
\[
	M_k^i = \pil \shA_M^{n+k,n+i}
\]
and with differential $d = (-1)^k \dbar$. The $\algE$-module structure is the obvious
one: a holomorphic form $\beta \in \Omega_B^j$ acts via wedge product with the
pullback $\piu \beta$. This is compatible with the differential because of the factor
$(-1)^k$.
\end{pexa}

\newpar
This is a good place to prove the claim we made in \Cref{par:claim-beta} when we
looked at holomorphic forms and the decomposition theorem: for any local section
$\beta$ of $\Omega_B^1$, the morphism
\[
	(\piu \beta)_0 \colon \gr_{\bullet}^F \DR(\Pmod_i) \to \gr_{\bullet-1}^F \DR(\Pmod_i)
	\decal{1}
\]
comes from the action of $\Omega_B^1$ on the de Rham complex. Let us 
restate the problem using the BGG correspondence. The morphism $(\piu \beta)_0 \colon
\gr_{\bullet}^F \Pmod_i \to \gr_{\bullet-1}^F \Pmod_i \decal{1}$ induces a morphism
\[
	\BGGL \bigl( \gr_{\bullet}^F \Pmod_i \bigr) \to \BGGL \bigl( \gr_{\bullet-1}^F
	\Pmod_i \decal{1} \bigr) = \BGGL \bigl( \gr_{\bullet}^F \Pmod_i \bigr)(1),
\]
and the claim is that this is just multiplication by $\beta \in \Omega_B^1$. 
We know from the previous paragraph that $\gr_{\bullet}^F \Cpi$ corresponds, under
the BGG correspondence, to the complex $(M,d)$; consequently, the decomposition
theorem gives us an isomorphism (in the derived category)
\[
	(M, d) \cong \bigoplus_{i=-n}^n \BGGL \bigl( \gr_{\bullet}^F \Pmod_i \bigr)
	\decal{-i}.
\]
As we have just seen, the $\algE$-module structure on $(M,d)$ is such that $\beta \in
\Omega_B^1$ acts via wedge product with $\piu \beta$. But this means that wedge
product with $\piu \beta$ also gives the $\algE$-module structure on each
summand $\BGGL \bigl( \gr_{\bullet}^F \Pmod_i \bigr)$, and this is exactly what we
wanted to prove.

\newpar
Next, we define a functor $\BGGR$ from complexes of graded $\algS$-modules to
complexes of graded $\algE$-modules. The general idea is to take the tensor product
with the Koszul complex in \eqref{eq:Koszul}, but to adjust both the degree and the
grading in order to get $\algE$ to act correctly. 

\begin{pexa}
	Suppose we tensor a single graded $\algS$-module $N$ by the Koszul complex. This
	produces a collection of complexes $C_k$, indexed by $k \in \ZZ$, that look like this:
	\[
		\begin{tikzcd}[column sep=small]
			\dotsb \rar & N_{k-2} \tensor \wed^2 \shT_B \rar{\delta} & 
			N_{k-1} \tensor \shT_B \rar{\delta} & N_k \tensor \shO_B
		\end{tikzcd}
	\]
	If we define the action by $\Omega_B^1$ in the obvious way as
	\[
		\Omega_B^1 \tensor \Bigl( N_k \tensor \wed^i \shT_B \Bigr)
		\to N_k \tensor \wed^{i-1} \shT_B, \quad
		\dtj \tensor (s \tensor \partial_J) \mapsto \sgn(J,j) \cdot s \tensor
		\partial_{J \setminus \{j\}},
	\]
	then a short calculation shows that we get a morphism of complexes $\Omega_B^1
	\tensor C_k \to C_{k-1} \decal{1}$, but where the differential in the second
	complex is $-\delta$. So both the grading and the differentials are wrong. To fix
	this problem, we need to work with the complexes $C_k' = C_{-k} \decal{k}$, with
	differential $(-1)^k \delta$, because this makes $\Omega_B^1 \tensor C_k' \to C_{k+1}'$
	behave as it should.
\end{pexa}

\newpar
With this in mind, let $(N, d)$ be a complex of graded $\algS$-modules. We define
$\BGGR(N, d)$ as the complex whose $i$-th term is the graded $\shO_B$-module
\[
	\BGGR(N,d)_k^i = \bigoplus_{p+q=i} N_q^p \tensor \wed^{-q-k} \shT_B,
\]
and whose differential acts on the summand $N_q^p \tensor \wed^{-q-k} \shT_B$
by the formula
\[
	d \tensor \id + (-1)^{p+k} \delta,
\]
where $\delta$ is the standard Koszul differential. This is again the simple
complex associated to the double complex with terms $N_q^p \tensor \wed^{-q-k} \shT_B$; the
extra $(-1)^k$ is justified by the example.
Each term in the complex becomes a graded module over $\algE$ through the morphism
\begin{align*}
	\Omega_B^1 \tensor \Bigl( N_q^p \tensor \wed^{-q-k} \shT_B \Bigr)
	&\to N_q^p \tensor \wed^{-q-k-1} \shT_B, \\
	\dtj \tensor (s \tensor \partial_J) &\mapsto \sgn(J,j) \cdot s \tensor
	\partial_{J \setminus \{j\}},
\end{align*}
and one checks (by the same calculation as in the example) that the differential is
indeed linear over $\algE$. Once again, $\BGGR$ descends to an exact functor
\[
	\BGGR \colon \Dbcoh G(\algS) \to \Dbcoh G(\algE)
\]
between the two derived categories of graded modules.

\begin{pexa} \label{ex:BGG-DR}
	Let $(\Pmod, F_{\bullet} \Pmod)$ be a filtered right $\Dmod_B$-module. The
	associated graded object $\gr_{\bullet}^F \Pmod$ is a graded $\algS$-module.
	Term by term, we have
	\[
		\BGGR \bigl( \gr_{\bullet}^F \Pmod \bigr)_k^i 
			= \gr_i^F \Pmod \tensor \wed^{-i-k} \shT_B,
	\]
	and since the differential is exactly $(-1)^k \delta$, we obtain
	\[
		\BGGR \bigl( \gr_{\bullet}^F \Pmod \bigr) 
			= \bigoplus_{k \in \ZZ} \gr_{-k}^F \DR(\Pmod) \decal{k},
	\]
	where the grading is by $k$, and the $\algE$-module structure is defined
	(as above) by contraction, viewed as a morphism of complexes
	\[
		\Omega_B^j \tensor \gr_{-k}^F \DR(\Pmod) \decal{k} \to \gr_{-k-j}^F \DR(\Pmod)
		\decal{k+j}.
	\]
	So the graded pieces of the de Rham complex of a filtered $\Dmod$-module naturally
	fit into the framework of the BGG correspondence.
\end{pexa}

\newpar
The first result is that $\BGGL$ and $\BGGR$ are adjoint functors. 

\begin{pthm} \label{thm:BGG-adjoint}
	We have a natural isomorphism of bifunctors
	\[
		\Hom_{\Dbcoh G(\algS)} \bigl( \BGGL \argbl, \argbl \bigr)
		\cong \Hom_{\Dbcoh G(\algE)} \bigl( \argbl, \BGGR \argbl),
	\]
	which means that $(\BGGL, \BGGR)$ is an adjoint pair of functors.
\end{pthm}

\begin{proof}
	Let $(M,d)$ be a complex of graded $\algE$-modules and let $(N,d)$ be a complex of
	graded $\algS$-modules. A morphism of complexes of graded $\algS$-modules
	\[
		\BGGL(M,d) \to (N,d)
	\]
	is the same as a collection of morphisms of graded $\algS$-modules
	\[
		\BGGL(M,d)^i = \bigoplus_{p+q=i} M_p^q \tensor \algS(p) \to N^i
	\]
	that are compatible with the differentials in the two complexes. This is
	equivalent to giving a collection of morphisms of $\shO_B$-modules $f \colon
	M_p^q \to N_{-p}^{p+q}$, subject to the condition that
	\[
		d f(m) = (-1)^p f(dm) + \sum_{j=1}^n \partial_j f(\dtj \cdot m)
		\quad \text{for $m \in M_p^q$.}
	\]
	From this data, we can define morphisms of $\shO_B$-modules
	\[
		g \colon M_k^i \to \BGGR(N,d)_k^i = N_{-k}^{i+k} \tensor \shO_B \oplus N_{-k-1}^{i+k+1}
		\tensor \shT_B \oplus N_{-k-2}^{i+k+2} \tensor \wed^2 \shT_B \oplus \dotsb
	\]
	by the explicit formula
	\[
		g(m) = (-1)^{ik} \sum_J (-1)^{i \abs{J}} \eps \bigl( \abs{J} \bigr) 
		\cdot f(\dtJ \cdot m) \tensor \partial_J,
	\]
	where $\eps(\ell) = (-1)^{\ell(\ell-1)/2}$ and the summation runs over all subsets
	of $\{1, \dotsc, n\}$. A straightforward calculation shows that this is compatible with
	the differentials and with the action by $\Omega_B$, and therefore defines a
	morphism of complexes of graded $\algE$-modules
	\[
		(M,d) \to \BGGR(N,d).
	\]
	This construction is reversible, and passes to the derived category.
\end{proof}

\newpar
The content of the BGG correspondence is that $\BGGL$ is an equivalence of
categories. We continue to denote by $\Dbcoh G(\algS)$ the derived category of
cohomologically bounded and coherent complexes of graded $\algS$-modules, and
similarly for $\algE$.

\begin{pthm} \label{thm:BGG-equivalence}
	The two functors 
	\[
		\BGGL \colon \Dbcoh G(\algE) \to \Dbcoh G(\algS) \quad \text{and} \quad
		\BGGR \colon \Dbcoh G(\algS) \to \Dbcoh G(\algE)
	\]
	are equivalences of categories that are inverse to each other.
\end{pthm}

\begin{proof}
	Let $(M,d)$ be a complex of graded $\algE$-modules. The adjointness of the two
	functors gives a morphism $(M,d) \to \BGGR \bigl( \BGGL(M,d) \bigr)$. Concretely,
	we have
	\[
		\BGGR \bigl( \BGGL(M,d) \bigr)_k^i 
		= \bigoplus_{p+q+r=i} M_p^q \tensor \algS^{p+r} \tensor \wed^{-r-k} \shT_B,
	\]
	which looks like the simple complex associated to a triple complex with grading
	$(p,q,r)$; and the differential is indeed given by the formula
	\[
		\sum_{j=1}^n \dt_j \tensor \partial_j \tensor \id + 
		(-1)^p d \tensor \id \tensor \id + (-1)^{p+q+k} \id \tensor \delta,
	\]
	where $\delta$ is again the standard Koszul differential. According to
	\Cref{thm:BGG-adjoint}, the morphism $M_k^i \to \BGGR \bigl( \BGGL(M,d)
	\bigr)_k^i$ is described by the formula
	\[
		m \mapsto (-1)^{ik} \sum_J (-1)^{i \abs{J}} \eps \bigl( \abs{J} \bigr) \cdot
		\dtJ m \tensor 1 \tensor \partial_J.
	\]
	The boundedness assumption ensures that the three different spectral sequences of
	the triple complex converge. The third spectral sequence starts from the
	differential $\delta$, and the fact that the Koszul complex in \eqref{eq:Koszul}
	is a resolution of the trivial $\algS$-module $\shO_B$ implies that the only
	nonzero cohomology object is $M_k^i$ for $(p,q,r) = (k,i,-k)$. The convergence of
	the spectral sequence therefore shows that $(M,d) \to \BGGR \bigl( \BGGL(M,d) \bigr)$
	is a quasi-isomorphism. The proof in the other direction is similar.
\end{proof}

\newpar
We are going to need two other facts about the BGG correspondence. The first is a
simple-minded bound on the amplitude of the complex $\BGGL(M,d)$, in terms of the
amplitude of the individual complexes of $\shO_B$-modules $(M_k,d)$.

\begin{pthm} \label{thm:BGG-amplitude}
	Let $(M,d)$ be a bounded complex of graded $\algE$-modules. If $\shH^q
	\bigl( M_p, d) = 0$ whenever $p+q > 0$, then $\shH^i \BGGL(M,d) = 0$ for $i > 0$.
\end{pthm}

\begin{proof}
	We view $\BGGL(M,d)$ as a double complex with terms $M_p^q \tensor \algS(p)$ and
	with the two commuting differentials $\sum_j \dtj \tensor \partial_j$ and $d
	\tensor \id$. The spectral sequence that starts from the differential $d \tensor
	\id$ converges because $M_p^q \tensor \algS(p) = 0$ for $\abs{q} \gg 0$. On
	the $E_1$-page, we get the graded $\algS$-modules $\shH^q(M_p,
	d) \tensor \algS(p)$, which vanish for $p+q > 0$ by assumption. The result now
	follows from the fact that the spectral sequence converges to the cohomology of
	the complex $\BGGL(M,d)$.
\end{proof}

\newpar
The second fact we need is that the BGG correspondence interacts nicely with
duality. Let $\algS'$ denote $\algS$, but with the $\algS$-module structure in which
$\Sym^k(\shT_B)$ act with an additional sign of $(-1)^k$. (Geometrically, this
amounts to pulling back by the automorphism that acts as $-1$ on the fibers of the
contangent bundle of $B$.)

\begin{pthm} \label{thm:BGG-duality}
	Let $(M,d)$ be a bounded complex of finitely generated graded $\algE$-modules.
	One has a natural isomorphism 
	\[
		\derR \shHom_{\algS} \Bigl( \BGGL(M,d), \algS \tensor \omega_B \decal{n} \Bigr)
		\cong \BGGL \Bigl( \derR \shHom_{\shO_B} \bigl( (M,d), \omega_B \decal{n}
		\bigr) \Bigr) \tensor_{\algS} \algS'.
	\]
\end{pthm}

\begin{proof}
	The boundedness assumption implies that there are only finitely many value of $i$
	and $k$ for which $M_k^i \neq 0$; this is needed in order to avoid infinite direct
	sums (which do not commute with $\shHom$). Pick a resolution for $\omega_B \decal{n}$ by
	injective $\shO_B$-modules, say
	\[
		0 \to \omega_B \to I^{-n} \to \dotsb \to I^{-1} \to I^0 \to 0;
	\]
	we will abbreviate this by $(I, d)$; recall that the injective dimension of
	$\omega_B$ is $n = \dim B$ according to \cite{Golovin}. We can then represent
	$\derR \shHom_{\shO_B} \bigl( (M,d), \omega_B \decal{n} \bigr)$ by the complex of
	graded $\shO_B$-modules $(\Mh, d)$, where
	\[
		\Mh_k^i = \bigoplus_{p+q=i} \shHom_{\shO_B} \bigl( M_{-k}^{-p}, I^q \bigr),
	\]
	and where the differential is given by the formula $df = f \circ d + (-1)^p d
	\circ f$ for any local section $f \in \shHom_{\shO_B} \bigl( M_{-k}^{-p}, I^q
	\bigr)$. Since each $M^i$ is a graded left $\algE$-module, the terms in the
	complex $(\Mh, d)$ are naturally
	\emph{right} $\algE$-modules, but we can convert them back into left $\algE$-modules
	by letting $\Omega_B^j$ act with an extra factor of $(-1)^j$, meaning that
	\[
		(\dtj f)(m) = -f \bigl( \dtj m \bigr).
	\]
	With this convention, $(\Mh,d)$ is a complex of graded left $\algE$-modules, and
	so $\BGGL(\Mh,d)$ is defined. This describes the complex on the right-hand side of
	the claimed isomorphism.

	Now we turn to the complex $\derR \shHom_{\algS} \bigl( \BGGL(M,d), \algS
	\tensor \omega_B \decal{n} \bigr)$ on the left-hand side. It is realized by the
	simple complex associated to the double complex $\shHom_{\algS} \bigl( \BGGL(M,d),
	\algS \tensor I^{\bullet} \bigr)$. After some rearranging, the $i$-th term of
	the resulting simple complex comes out to be 
	\begin{align*}
		\bigoplus_{p+q=i} \shHom_{\algS} \bigl( \BGGL(M,d)^{-p}, \algS \tensor I^q \bigr) 
		= \bigoplus_{p+q=i} \Mh_p^q \tensor \algS(p),
	\end{align*}
	and the differential matches up with the differential in $\BGGL(\Mh,d)$, except for
	an extra $-1$ in front of the term $\sum_j \dt_j \tensor \partial_j$. This is
	corrected by tensoring with $\algS'$, whence the result.
\end{proof}

\newpar
Note that $\algS' \cong \algS$ are isomorphic as graded $\algS$-modules, and so we
can remove the tensor product with $\algS'$ from the statement if we like.

\section{Proof of the main theorem}

\newpar
In this chapter, we give the proof of \Cref{thm:isomorphism}; along the way, we
also establish \Cref{thm:HL-sigma} (and therefore \Cref{conj:symmetry}). Even though
we are interested in relating $\gr_{\bullet}^F \Pmod_k$ and $\Omega_M^{n+k}$
individually, it turns out that the necessary structure is only there if we look at
all of these objects together. The general idea is the following. There are
three different ways to make the direct sum of the complexes $G_{i,k}$ (with
appropriate shifts) into a complex of graded modules over the algebra $\Omega_B =
\bigoplus_j \Omega_B^j$. In the first, a local section $\beta$ of $\Omega_B^1$ acts
via $\piu \beta$; in the second, via the associated vector field $v(\beta)$; and in
the third, via the associated $(0,1)$-form $f(\beta)_1$ (in the notation of
\Cref{chap:differential-forms}). These three different structures are related by the
action of the symplectic form $\sigma$ and the K\"ahler form $\omega$, and together
with the BGG correspondence and basic facts about Hodge modules, this gives us enough
information to prove \Cref{thm:isomorphism}.

\newpar
From the Hodge modules $P_{-n}, \dotsc, P_n$, we get a collection of graded
$\algS$-modules $\gr_{\bullet}^F \Pmod_i$, for $i = -n, \dotsc, n$. The BGG
correspondence associates to each of these graded $\algS$-modules a complex of graded
$\algE$-modules $\BGGR(\gr_{\bullet}^F \Pmod_i)$; recall from \Cref{ex:BGG-DR} that
\[
	\BGGR(\gr_{\bullet}^F \Pmod_i) = \bigoplus_{k=-n}^n G_{i,k} \decal{i+k},
\]
and that the degree of the summand $G_{i,k} \decal{i+k}$ with respect to the grading
is $k$. Adding all of these objects together, we obtain the first object
\[
	G = \bigoplus_{i,k=-n}^n G_{i,k} \decal{k} =
	\BGGR \left( \bigoplus_{i=-n}^n \gr_{\bullet}^F \Pmod_i \decal{-i} \right).
\]
In other words, $G$ is a complex of graded $\algE$-modules, with the summand
$G_{i,k} \decal{k}$ in graded degree $k$; it is the object that corresponds, under
the BGG correspondence in \Cref{thm:BGG-equivalence}, to the complex of graded
$\algS$-modules
\[
	\bigoplus_{i=-n}^n \gr_{\bullet}^F \Pmod_i \decal{-i}.
\]
In the notation that we introduced in \Cref{chap:differential-forms}, a local section 
$\beta$ of $\Omega_B^1$ acts on the complex $G$ via the collection of graded
morphisms
\[
	\piu \beta \colon G_{i,k} \decal{k} \to G_{i,k+1} \decal{k+1};
\]
of course, this uniquely determines the module structure on $G$, because $\Omega_B$ is
generated by $\Omega_B^1$ as an $\shO_B$-algebra. We will use this idea several times
below.

\newpar
We have a concrete model for $G$ in terms of smooth differential forms. Recall from
\Cref{par:Laumon} the definition of the complex $\gr_{\bullet}^F \Cpi$. This is the
complex of graded $\algS$-modules with terms
\[
	\gr_{\bullet}^F \Cpi^i = \bigoplus_{p+q=i} \pil \shA_M^{n+p,n+q} \tensor_{\shO_M} 
	\Sym^{\bullet+p}(\shT_B),
\]
and with differential (in local coordinates)
\[
	\dpi \bigl( \alpha \tensor P \bigr) = \dbar \alpha \tensor P + 
	\sum_{j=1}^n \piu(\dtj) \wedge \alpha \tensor \partial_j P.
\]
From Saito's theory, we get an isomorphism
\[
	\gr_{\bullet}^F \Cpi \cong \bigoplus_{i=-n}^n \gr_{\bullet}^F \Pmod_i \decal{-i}
\]
in the derived category of graded $\algS$-modules; here the decomposition is induced
by the one in the decomposition theorem. The calculation in \Cref{ex:grFCpi} shows
that $\gr_{\bullet}^F \Cpi \cong \BGGL(M,d)$, where $(M,d)$ is the complex of graded
$\algE$-modules with 
\[
	M_k^i = \pil \shA_M^{n+k,n+i} \quad \text{and} \quad d = (-1)^k \dbar.
\]
The $\algE$-module structure is the obvious one: a holomorphic form $\beta \in
\Omega_B^j$ acts via wedge product with the pullback $\piu \beta$. Since the BGG
correspondence is an equivalence of categories, this means that we have an isomorphism
\[
	G \cong (M,d)
\]
in the derived category of graded $\algE$-modules. This says in particular that the
complex $(M,d)$ splits in the derived category (because of the decomposition
theorem), which is an extremely deep fact about proper morphisms between K\"ahler manifolds.

\newpar
Note that $G$ is also related to the direct sum of all the sheaves of holomorphic
forms on $M$, because Saito's formula in \eqref{eq:pil-Omega} gives us an isomorphism
\[
	G \cong \bigoplus_{k=-n}^n \derR \pil \Omega_M^{n+k} \decal{n}
\]
in the derived category of graded $\shO_B$-modules.

\newpar
Now we introduce the second object. This is the complex
\[
	G_v = \bigoplus_{i,k} G_{i,k} \decal{k},
\]
in which the summand $G_{i,k} \decal{k}$ sits in graded degree $-k$. As a complex of
$\shO_B$-modules, this is of course isomorphic to $G$, but the grading and the
$\algE$-module structure are different. We turn $G_v$ into a complex of graded
$\algE$-modules in the following way. Recall that we have the Weil element $\wsigma$,
associated to the $\sltwo$-representation on $\bigoplus_k \Omega_M^{n+k}$. It
determines an automorphism of the complex $(M,d)$ that maps $M_k^i$ isomorphically to
$M_{-k}^i$. According to \Cref{lem:contraction}, for any local section $\beta$ of
$\Omega_B^1$, it also exchanges the action by the holomorphic form $\piu \beta$
agains the action by the holomorphic vector field $v(\beta)$. In other words,
\[
	\wsigma \colon G \to G_v
\]
is an isomorphism between $G$ and $G_v$, as complexes of graded $\shO_B$-modules.
Note that $\wsigma$ respects the grading, but there is no reason why it should
preserve the individual summands $G_{i,k} \decal{k}$ in the decompositions of $G$
and $G_v$. 

\newpar
We use this isomorphism to give $G_v$ the structure of a complex of graded
$\algE$-modules. \Cref{lem:contraction} tells us that we have, for every local
section $\beta \in \Omega_B^1$, a commutative diagram
\[
	\begin{tikzcd}
		G \rar{\wsl_{\sigma}} \dar{\beta} & G_v \dar{v(\beta)} \\
		G \rar{\wsl_{\sigma}} & G_v.
	\end{tikzcd}
\]
The $\algE$-module structure on $G_v$ is therefore the unique one for which a
locally defined holomorphic form $\beta \in \Omega_B^1$ acts via the collection of
graded morphisms
\[
	v(\beta) \colon G_{i,k} \decal{k} \to G_{i-1,k-1} \decal{k-1}.
\]
Let me stress that this step of the construction works on the level of smooth forms,
meaning on the complex $(M,d)$, because the symplectic form is a holomorphic form of
type $(2,0)$. Neither Hodge theory nor the decomposition theorem are needed here.

\newpar
Let us now describe the third object. This is the complex of graded $\shO_B$-modules
\[
	G_f = \bigoplus_{i,k} G_{i,k} \decal{i+k}
\]
in which the term $G_{i,k} \decal{k}$ sits in graded degree $i-k$. The reason for
this choice of grading is the following. The topmost component of the K\"ahler form
gives us a morphism of complexes $\omega_2 \colon G_{i,k} \to G_{i+2,k+1} \decal{2}$;
from the resulting representation of the Lie algebra $\sltwo$, we get a second Weil
element $\womegat$. For every $i,k \in \ZZ$, it induces an isomorphism of complexes
\[
	\womegat \colon G_{i,k} \decal{k} \to G_{-i,k-i} \decal{k-2i}.
\]
The resulting isomorphism of complexes of $\shO_B$-modules
\[
	\womegat \colon G_v \to G_f
\]
therefore respects the grading exactly when we put the summand $G_{i,k} \decal{i+k}$
in the complex $G_f$ in graded degree $i-k$. Unlike the other isomorphism, this one
cannot be defined on the level of smooth forms, because $\womegat$ does not commute
with the operator $\dbar$; instead, we have to rely on difficult results from
Hodge theory (such as the relative Hard Lefschetz theorem) to do the work for us.

\newpar
As in the previous step, we use the isomorphism $\womegat \colon G_v \to G_f$ to turn
the object $G_f$ into a complex of graded $\algE$-modules. We can again describe the
$\algE$-module structure on $G_f$ very concretely with the
help of the results from \Cref{chap:differential-forms}. Recall from \Cref{lem:xi_k}
that the operator $v(\beta)$ is primitive of weight $-1$ with respect to the
representation of $\sltwo$ determined by $\omega_2$. Consequently, we have 
\[
	\Ad \womegat \bigl( v(\beta) \bigr) = \womegat v(\beta) \, \womegat^{-1} 
	= [\omega_2, v(\beta)] = f(\beta)_1. 
\]
For every local section $\beta$ of $\Omega_B^1$, we therefore get another commutative
diagram
\[
	\begin{tikzcd}
		G_v \rar{\womegat} \dar{v(\beta)} & G_f \dar{f(\beta)_1} \\
		G_v \rar{\womegat} & G_f.
	\end{tikzcd}
\]
It follows that the $\algE$-module structure on the complex $G_f$ is the unique one
for which a local section $\beta$ of $\Omega_B^1$ acts via the collection of graded
morphisms
\[
	f(\beta)_1 \colon G_{i,k} \decal{i+k} \to G_{i+1,k} \decal{i+k+1}.
\]
This is compatible with the grading (because the left-hand side sits in
graded degree $i-k$ and the right-hand side in graded degree $i-k+1$.) 

\newpar
The rest of the proof consists mostly in applying the BGG correspondence for the two
graded algebras $\algS = \Sym(\shT_B)$ and $\algE = \bigoplus_j \Omega_B^j$. Recall
that we have
\[
	G = \BGGR \left( \bigoplus_{i=-n}^n \gr_{\bullet}^F \Pmod_i \decal{-i} \right).
\]
By construction, $\womegat \wsl_{\sigma} \colon G \to G_f$ is an
isomorphism in the derived category of graded $\algE$-modules; the $\algE$-module
structure on $G_f$ has the property that a local section $\beta$ of $\Omega_B^1$
acts via the operator $f(\beta)_1$. Since $f(\beta)_1 \colon G_{i,k} \decal{i} \to
G_{i+1,k} \decal{i+1}$ only changes the index $i$, it is
obvious that $G_f$ decomposes, as a complex of graded $\algE$-modules, into a direct sum
\[
	G_f = \bigoplus_{k=-n}^n \left( \bigoplus_{i=-n}^n G_{i,k} \decal{i} \right) \decal{k}.
\]
Let $F_{-k}$ denote the complex of graded $\algS$-modules that the BGG correspondence
associates to the $k$-th summand in this decomposition; in symbols,
\[
	F_{-k} = \BGGL \left( \bigoplus_{i=-n}^n G_{i,k} \decal{i} \right). 
\]
Since the BGG correspondence is an equivalence of categories (by
\Cref{thm:BGG-equivalence}), we conclude from the isomorphism $G \cong G_f$ in the
derived category of graded $\algE$-modules that
\begin{equation} \label{eq:iso-Fk}
	\bigoplus_{i=-n}^n \gr_{\bullet}^F \Pmod_i \decal{-i} 
	\cong \bigoplus_{k=-n}^n F_k \decal{-k},
\end{equation}
in the derived category of graded $\algS$-modules. 

\newpar
What we need to do now is to prove that $F_k \cong \gr_{\bullet}^F \Pmod_k$ for all
$k = -n, \dotsc, n$. Since we know next to nothing about the complexes $F_k$, this may
seem impossible -- but in fact, we have just enough information to make it work.

\newpar
The left-hand side of \eqref{eq:iso-Fk} is a direct sum of graded
$\algS$-modules, each of which is an $n$-dimensional Cohen-Macaulay module. 
It follows that each complex $F_k$ also splits into a direct sum of
$n$-dimensional Cohen-Macaulay modules; consequently, we get a decomposition
\[
	F_k \cong \bigoplus_{\ell \in \ZZ} F_{k, \ell} \decal{-\ell}
\]
in the derived category of graded $\algS$-modules. Each $F_{k,\ell}$ is a graded
$\algS$-module that is $n$-dimensional and Cohen-Macaulay. Our task now becomes showing that
$F_{k,\ell} = 0$ for $\ell \neq 0$. It turns out that this is a purely formal
consequence of what we know about the complexes $G_{i,k}$. On the one hand, the bound
on the amplitude of the complex $G_{i,k}$ (in \Cref{lem:amplitude}) implies that
$F_k$ is concentrated in nonpositive degrees.

\begin{plem} \label{lem:Fkell}
	We have $F_{k, \ell} = 0$ for $\ell > 0$.
\end{plem}

\begin{proof}
	Recall that we defined $F_k$ with the help of the BGG correspondence as
	\[
		F_k = \BGGL \left( \bigoplus_{i=-n}^n G_{i,-k} \decal{i} \right).
	\]
	According to \Cref{lem:amplitude}, we have $\shH^j G_{i,k} = 0$ for $j > k$, and
	therefore $\shH^j \bigl( G_{i,-k} \decal{i} \bigr) = \shH^{i+j} G_{i,-k} = 0$ for
	$i+j+k > 0$. Since the summand $G_{i,-k} \decal{i}$ has degree $i+k$ with respect
	to the grading on $G_f$ (and therefore on the object in parentheses), this is
	exactly what we need in order to apply \Cref{thm:BGG-amplitude}.  The conclusion
	is that $\shH^j F_k = 0$ for $j > 0$.  
\end{proof}

\newpar
We can use duality to prove the vanishing of the graded $\algS$-modules $F_{k,\ell}$ for
$\ell < 0$. For the sake of clarity, let us temporarily write
\[
	G_k = \bigoplus_{i=-n}^n G_{i,-k} \decal{i}
\]
for the complex of graded $\algE$-modules on the right-hand side; then $F_k =
\BGGL(G_k)$. Recall from \eqref{eq:duality} that $G_{-i,-k} \cong \derR
\shHom_{\shO_B} \bigl( G_{i,k}, \omega_B \decal{n} \bigr)$. Therefore
\[
	\Gh_k = \derR \shHom_{\shO_B} \Bigl( G_k, \omega_B \decal{n} \Bigr)
	\cong \bigoplus_{i=-n}^n G_{i,k} \decal{i},
\]
where the notation $\Gh_k$ comes from the proof of \Cref{thm:BGG-duality}. Note that
the summand $G_{i,k} \decal{i}$ again ends up having degree $i-k$ with respect to the
induced grading on $\Gh_k$.\footnote{In fact, it should be the case that $\Gh_k \cong
G_{-k}$, but this would be tedious to check, and fortunately it turns out to be
irrelevant for the proof.} By the same argument as in \Cref{lem:Fkell}, the complex
$\BGGL \bigl( \Gh_k \bigr)$ is concentrated in degrees $\leq 0$.
Now \Cref{thm:BGG-duality} gives 
\begin{align*}
	\BGGL(\Gh_k) &\cong
	\derR \shHom_{\algS} \Bigl( \BGGL(G_k), \algS \tensor \omega_B \decal{n} \Bigr)
	\cong \omega_B \tensor \derR \shHom_{\algS} \Bigl( F_k, \algS \decal{n} \Bigr)  \\
	&\cong \bigoplus_{\ell \leq 0} 
	\omega_B \tensor \derR \shHom_{\algS} \Bigl( F_{k, \ell}, \algS \decal{n} \Bigr)
	\decal{\ell}.
\end{align*}
Since each $F_{k,\ell}$ is an $n$-dimensional Cohen-Macaulay module, the complex on
the right-hand side is concentrated in degrees $\geq 0$. But the complex on the
left-hand side is concentrated in degrees $\leq 0$, and so it must be that
$F_{k,\ell} = 0$ for $\ell < 0$. 

\newpar
The conclusion is that each complex $F_k$ is actually a single graded $\algS$-module
(in cohomological degree $0$). Because of the isomorphism in \eqref{eq:iso-Fk}, we
then get
\[
	F_k \cong \gr_{\bullet}^F \Pmod_k.
\]
If we now use the BGG correspondence in the other direction, we find that
\begin{equation} \label{eq:iso-final}
	\bigoplus_{i=-n}^n G_{i,-k} \decal{i} \cong 
	\bigoplus_{i=-n}^n G_{k,i} \decal{i+k}
\end{equation}
are isomorphic in the derived category of graded $\algE$-modules. We will now take a
short break from proving \Cref{thm:isomorphism} and turn to the
symplectic relative Hard Lefschetz theorem (in \Cref{thm:HL-sigma}) and the symmetry
conjecture of Shen and Yin (in \Cref{conj:symmetry}).

\newpar
If we forget the grading and the $\algE$-module structure, \eqref{eq:iso-final}
becomes an isomorphism in the derived category of $\shO_B$-modules. After replacing
$k$ by $-k$ and using the relative Hard Lefschetz isomorphism in \eqref{eq:HL-omega},
we can put it into the form
\[
	\bigoplus_{i=-n}^n G_{i,k} \decal{i+k} \cong 
	\bigoplus_{i=-n}^n G_{-k,i} \decal{i} \cong
	\bigoplus_{i=-n}^n G_{k,i+k} \decal{i+2k} \cong
	\bigoplus_{i=-n}^n G_{k,i} \decal{i+k}.
\]
This is true for every $k = -n, \dotsc, n$, and so we are in a position where we can
apply \Cref{thm:equivalence}. The conclusion is that the morphism
\[
	\sigma_1^k \colon G_{i,-k} \to G_{i,k} \decal{2k}
\]
is an isomorphism for every $k \geq 1$; this establishes the ``symplectic relative
Hard Lefschetz theorem''. It also follows that the complexes $G_{i,k}$ and $G_{k,i}$
are isomorphic in the derived category, as predicted by the symmetry conjecture of
Shen and Yin. 

\newpar \label{par:XYH}
We can now go back and finish the proof of \Cref{thm:isomorphism}. The relative Hard
Lefschetz theorem gives us a representation of the Lie algebra $\sltwo$ on the direct
sum of all the $G_{i,k}$ (with appropriate shifts). To simplify the notation, let us
denote the three operators by the symbols $\Xsl_1, \Ysl_1, \Hsl_1$, and the Weil
element by the symbol $\wsl_1$. Concretely,
\[
	\Xsl_1 = \omega_2 \colon G_{i,k} \to G_{i+2,k+1} \decal{2} \quad \text{and} \quad
	\Ysl_1 = \Ysl_{\omega_2} \colon G_{i,k} \to G_{i-2,k-1} \decal{-2},
\]
$\Hsl_1$ acts on $G_{i,k}$ as multiplication by the integer $i$, and the
Weil element is the isomorphism
\[
	\wsl_1 = \wsl_{\omega_2} \colon G_{i,k} \to G_{-i,k-i} \decal{-2i}.
\]
Likewise, the symplectic relative Hard Lefschetz theorem gives us a second
representation of $\sltwo$, for which we use the symbols $\Xsl_2, \Ysl_2,
\Hsl_2$ and $\wsl_2$. Once again,
\[
	\Xsl_2 = \sigma_1 \colon G_{i,k} \to G_{i+1,k+2} \decal{2} \quad \text{and} \quad
	\Ysl_2 = \Ysl_{\sigma_1} \colon G_{i,k} \to G_{i-1,k-2} \decal{-2},
\]
$\Hsl_2$ acts on $G_{i,k}$ as multiplication by the integer $k$, and the
Weil element is the isomorphism
\[
	\wsl_2 = \wsl_{\sigma_1} \colon G_{i,k} \to G_{i-k,-k} \decal{-2k}.
\]
We know from \Cref{lem:sigma_k} that $[\omega_2, \sigma_1] = 0$, which translates
into the relation $[X_1, X_2] = 0$.

\newpar
We can now establish a direct relationship between $\gr_{\bullet}^F
\Pmod_i$ and $\derR \pil \Omega_M^{n+i} \decal{n}$ with the help of the two
$\sltwo$-representations. Recall from \Cref{ex:BGG-DR} that
\[
	\BGGR(\gr_{\bullet}^F \Pmod_i) = \bigoplus_{k=-n}^n G_{i,k} \decal{i+k}.
\]
Consider the following chain of isomorphisms
\[
	\begin{tikzcd}
		G_{i,k} \rar{\wsl_1} &
		G_{-i,k-i} \decal{-2i} \rar{\wsl_2} &
		G_{-k,i-k} \decal{-2k} \rar{\wsl_1} &
		G_{k,i}
	\end{tikzcd}
\]
that, very concretely, realizes the symmetry between $G_{i,k}$ and $G_{k,i}$. (In the
drawing in \Cref{par:hexagon}, the composition $\wsl_1 \wsl_2 \wsl_1$ is exactly
reflection in the third diagonal of the hexagon.) 

\newpar
Let $\beta \in \Omega_B^1$ be a locally defined holomorphic $1$-form. From the
construction in \Cref{chap:differential-forms}, we get three operators (that are
defined on the same open set as $\beta$, of course):
\begin{align*}
	\piu \beta = (\piu \beta)_0 &\colon G_{i,k} \to G_{i,k+1} \decal{1} \\
	v(\beta) = v(\beta)_{-1} &\colon G_{i,k} \to G_{i-1,k-1} \decal{-1} \\
	f(\beta)_1 = [X_1, v(\beta)] &\colon G_{i,k} \to G_{i,k+1} \decal{1}
\end{align*}
Let us see how these three operators interact with the two Weil elements $\wsl_1$ and
$\wsl_2$. We know that $\ad X_1(\piu \beta) = 0$ (from \Cref{lem:beta_k}), which implies that
\[
	\Ad \wsl_1 \bigl( \piu \beta \bigr) 
	= \wsl_1 \bigl( \piu \beta \bigr) \wsl_1^{-1} = \piu \beta.
\]
We also showed (in \Cref{lem:xi_k}) that $(\ad X_1)^2 v(\beta) = 0$; because $v(\beta)$
has weight $-1$ with respect to $\Hsl_1$, it follows that
\[
	\Ad \wsl_1 \bigl( v(\beta) \bigr) = [X_1, v(\beta)] = f(\beta)_1.
\]
Lastly, we know from \Cref{lem:contraction} that $[\sigma, v(\beta)] = \piu \beta$,
and because $v(\beta)$ and $\piu \beta$ have weight $-1$ and $0$ with respect
to $\Hsl_1$, we get $[X_2, v(\beta)] = [\sigma_1, v(\beta)] = \piu \beta$. At the
same time, $\piu \beta$ has weight $1$ with respect to $\Hsl_2$ and commutes with
$X_2$, and therefore
\[
	\Ad \wsl_2 \bigl( \piu \beta \bigr) = -[\Ysl_2, \piu \beta] = v(\beta).
\]

\newpar
The conclusion of all these computations is that we get a commutative diagram
\[
	\begin{tikzcd}
		G_{i,k} \rar{\wsl_1} \dar{\piu \beta} &
		G_{-i,k-i} \decal{-2i} \rar{\wsl_2} \dar{\piu \beta} &
		G_{-k,i-k} \decal{-2k} \rar{\wsl_1} \dar{v(\beta)} &
		G_{k,i} \dar{f(\beta)_1} \\
		G_{i,k+1} \decal{1} \rar{\wsl_1} &
		G_{-i,k-i+1} \decal{-2i+1} \rar{\wsl_2} &
		G_{-k-1,i-k-1} \decal{-2k-1} \rar{\wsl_1} &
		G_{k+1,i} \decal{1}
	\end{tikzcd}
\]
in which all the horizontal morphisms are isomorphisms. After shifting everything by
$i+k$ and taking the direct sum over all $k \in \ZZ$, this gives us an isomorphism
between
\[
	\BGGR(\gr_{\bullet}^F \Pmod_i) = \bigoplus_{k=-n}^n G_{i,k} \decal{i+k}
\]
and the complex of graded $\algE$-modules
\begin{equation} \label{eq:new-complex}
	\bigoplus_{k=-n}^n G_{k,i} \decal{i+k},
\end{equation}
in the derived category of graded $\algE$-modules. Here the $\algE$-module structure
on the second complex is the unique one where $\beta \in \Omega_B^1$ acts through the
collection of graded morphisms
\[
	f(\beta)_1 \colon G_{k,i} \decal{i+k} \to G_{k+1,i} \decal{i+k+1},
\]
and the grading has the summand $G_{k,i} \decal{i+k}$ in graded degree $k$. The
argument above shows that this object is isomorphic, via $\wsl_1 \wsl_2 \wsl_1$, to
the object
\[
	\BGGR(\gr_{\bullet}^F \Pmod_i) = \bigoplus_{k=-n}^n G_{i,k} \decal{i+k},
\]
which is the image of $\gr_{\bullet}^F \Pmod_i$ under the BGG correspondence.

\newpar
To conclude the proof, we only need to describe how the object in
\eqref{eq:new-complex} is related to $\derR \pil \Omega_M^{n+i} \decal{n}$. In the
derived category of $\shO_B$-modules, we do have an isomorphism
\begin{equation} \label{eq:Omega-grading}
	\derR \pil \Omega_M^{n+i} \decal{n} \cong \bigoplus_{k=-n}^n G_{k,i} \decal{i}.
\end{equation} 
The object on the left-hand side is naturally a module over the algebra
\[
	\derR \pil \shO_M \cong \bigoplus_{j=0}^n \Omega_B^j \decal{-j},
\]
where the isomorphism comes from Matsushita's theorem (in \Cref{thm:Matsushita}). If
we take the associated
graded of this action with respect to the perverse filtration, which is just the
filtration by increasing values of $k$ in the above decomposition, 
\eqref{eq:Omega-grading} becomes an isomorphism of graded modules
\[
	\gr_{\bullet}^P \Bigl( \derR \pil \Omega_M^{n+i} \decal{n} \Bigr) 
		\cong \bigoplus_{k=-n}^n G_{k,i} \decal{i}.
\]
Moreover, $\Omega_B^1 \decal{-1}$ now acts on the object on the right-hand side 
exactly through the collection of morphisms $f(\beta)_1 \colon G_{i,k} \to
G_{i+1,k} \decal{1}$. We can turn this into an honest action by $\algE$ by adding a
shift by $k$ to the $k$-th term in the sum; in this way, we arrive at the object
\[
	\bigoplus_{k=-n}^n G_{k,i} \decal{i+k}
\]
with the $\algE$-module structure and the grading that appeared in the proof above. 

\newpar
In other words, we need to take the associated graded of $\derR \pil \Omega_M^{n+i}
\decal{n}$ with respect to the perverse filtration, in order to extract from the
action by $f(\beta)$ its topmost component $f(\beta)_1$. After adding appropriate
shifts (determined by the degree with respect to the grading), we then obtain a
complex of graded $\algE$-modules, and under the BGG correspondence, this goes to the
graded $\algS$-module $\gr_{\bullet}^F \Pmod_i$.  With this, we have proved all the
claims that were made in the introduction.

\newpar
The interesting point is that the process we have described does relate $\derR \pil
\Omega_M^{n+i} \decal{n}$ to the filtered $\Dmod$-module $(\Pmod_i, F_{\bullet}
\Pmod_i)$, but only after we take the associated graded on both sides: with respect
to the perverse filtration on one side, and with respect to the Hodge filtration on
the other.  I do not know whether one can expect to relate the two sides without
going to the associated graded objects.

\section{Symmetries and the Lie algebra $\slthree$}

\newpar
In this chapter, we explain why the Lie algebra $\slthree$ acts on the direct
sum of the complexes $G_{i,k}$ (with suitable shifts). As in
\Cref{chap:differential-forms}, this relies on computations with 
differential forms on $M$; the main point is to understand how the Weil element
$\wsigma$ for the holomorphic symplectic form $\sigma$ interacts with the K\"ahler
form $\omega$.

\newpar
Recall that the holomorphic symplectic form $\sigma$ determines an isomorphism
$\shT_M \cong \Omega_M^1$. If we think of the K\"ahler form $\omega$ as a
$\dbar$-closed $(0,1)$-form with coefficients in $\Omega_M^1$, hence as an element
$\omega \in A^{0,1}(M, \Omega_M^1)$, the isomorphism provides us with another element
\[
	i(\omega) \in A^{0,1}(M, \shT_M),
\]
which is again $\dbar$-closed. We can also express the relation between $\omega$ and
$i(\omega)$ as
\[
	i(\omega) \cont \sigma = \omega.
\]
More generally, contraction with $i(\omega)$ is an operator
\[
	i(\omega) \cont \colon A^{p,q}(M) \to A^{p-1,q+1}(M),
\]
that acts as follows. Let $z_1, \dotsc, z_{2n}$ be local holomorphic
coordinates on $M$. Then
\[
	i(\omega) = \sum_{j,k} f_{j,k} \frac{\partial}{\partial z_j} \tensor \dzb_k,
\]
and we define the contraction against $(p,q)$-forms as
\[
	i(\omega) \cont \dz_J \wedge \dzb_K = (-1)^p \sum_{j,k} f_{j,k} \sgn(J,j) 
	\dz_{J \setminus \{j\}} \tensor \dzb_k \wedge \dzb_K,
\]
where $\abs{J} = p$ and $\abs{K} = q$. The sign $(-1)^p$ is caused by swapping the
order of $\dzb_k$ and $\dz_J$.

\newpar
The following lemma is proved in exactly the same way as \Cref{lem:contraction}.

\begin{plem}
	For every $k,q \in \ZZ$, the following diagram commutes:
	\[
		\begin{tikzcd}
			A^{n-k,q}(M) \dar{\omega \wedge} \rar{\wsigma} & A^{n+k,q}(M) \dar{i(\omega) \cont} \\
			A^{n-k+1,q+1}(M) \rar{\wsigma} & A^{n+k-1,q+1}(M)
		\end{tikzcd}
	\]
\end{plem}

\newpar
Now let us try to understand the commutator of the two operators $\alpha \mapsto
\omega \wedge \alpha$ and $\alpha \mapsto i(\omega) \cont \alpha$. We notice that the
contraction
\[
	\iomega = - i(\omega) \cont \omega \in A^{0,2}(M)
\]
is a $\dbar$-closed $(0,2)$-form. The following lemma shows that the commutator in
question is nothing but wedge product with $\iomega$.

\begin{plem}
	We have $\bigl[ \omega \wedge, i(\omega) \cont \bigr] = \iomega \wedge$.
\end{plem}

\begin{proof}
	This is due to the identity 
	\[
		i(\omega) \cont \bigl( \omega \wedge \alpha \bigr) 
		= \bigl( i(\omega) \cont \omega \bigr)
		\wedge \alpha + \omega \wedge \bigl( i(\omega) \cont \alpha \bigr)
		= -\iomega \wedge \alpha + \omega \wedge \bigl( i(\omega) \cont \alpha \bigr),
	\]
	which is easily proved by a computation in local coordinates.
\end{proof}

\newpar
Because $\iomega = -i(\omega) \cont \omega$ is a $\dbar$-closed $(0,2)$-form, it 
acts on the complex $\gr_{\bullet}^F \Cpi$. Using the decomposition theorem, it
therefore determines a morphism (in the derived category)
\[
	\iomega \colon \bigoplus_{i=-n}^n \gr_{\bullet}^F \Pmod_i \decal{-i}
	\to \bigoplus_{i=-n}^n \gr_{\bullet}^F \Pmod_i \decal{2-i}.
\]
As in earlier chapters, this morphism breaks up into a finite sum $\iomega = \iomega_2
+ \iomega_1 + \iomega_0 + \dotsb$; each component $\iomega_j$ is a morphism
\[
	\iomega_j \colon \gr_{\bullet}^F \Pmod_i \to \gr_{\bullet}^F \Pmod_{i+j}
	\decal{2-j}
\]
in the derived category of graded $\Sym \shT_B$-modules. We get induced morphisms
\[
	\iomega_j \colon G_{i,k} \to G_{i+j, k} \decal{2}.
\]
What matters for our computation is that $\iomega_j = 0$ for $j \geq 3$. This holds
because $\iomega$ is a $\dbar$-closed $(0,2)$-form and
therefore commutes with the differential in the complex $\gr_{\bullet}^F \Cpi$. 

\newpar
We now consider how $i(\omega)$ and $\iomega$ act on the complex $(M, d)$, where $M_k^i =
\pil \shA_M^{n+k,n+i}$ and $d = (-1)^k \dbar$. Recall that this complex is isomorphic
(in the derived category) to
\[
	G = \bigoplus_{i,k=-n}^n G_{i,k} \decal{k} \cong
	\bigoplus_{k = -n}^n \derR \pil \Omega_M^{n+k} \decal{n},
\]
but now we purposely forget the action by $\algE$, because contraction with
$i(\omega)$ does not preserve it. Contraction against $i(\omega)$ and wedge
product with $\iomega$ define two morphisms
\[
	i(\omega) \colon G \to G(-1) \decal{1} \quad \text{and} \quad
	\iomega \colon G \to G \decal{2}
\]
in the derived category; the first morphism changes the grading, but the second
one preserves it. The analysis in the previous two paragraphs shows that 
\[
	\iomega = -\bigl[ i(\omega), \omega \bigr] 
	= - \bigl[ \wsigma \cdot \omega \cdot \wsigma^{-1}, \omega \bigr],
\]
where $\wsigma$ is the Weil element for the representation of $\sltwo$ on
\[
	\bigoplus_{k=-n}^n \Omega_M^{n+k}
\]
coming from the action of the symplectic form $\sigma$. For the sake of clarity, let
us denote the third generator of this representation by the symbol $\Ysigma$.

\newpar
As $[\sigma, \omega] = 0$, the operator $\omega$ is primitive with respect to this
representation, and so
\[
	\wsigma \cdot \omega \cdot \wsigma^{-1} = -[\Ysigma, \omega].
\]
After combining this with the formula for $\iomega$, we get
\[
	\iomega = - \bigl[ \wsigma \cdot \omega \cdot \wsigma^{-1}, \omega \bigr]
	= \bigl[ [\Ysigma, \omega], \omega \bigr] 
		= \bigl[ \omega, [\omega, \Ysigma] \bigr].
\]
Now we look at the topmost component on each side, with respect to the filtration by
the first index $i$. On the right-hand side, by \Cref{lem:Ysigma}, this is the
iterated commutator
\[
	\Bigl[ \omega_2, \bigl[ \omega_2, \Ysl_{\sigma_1} \bigr] \Bigr],
\]
which is a morphism $G_{i,k} \to G_{i+3,k} \decal{2}$. The corresponding term on the
left-hand side is $\iomega_3$, and we already know that $\iomega_3 = 0$. This observation
proves the following lemma.

\begin{plem} \label{lem:relation}
	We have $\bigl[ \omega_2, [\omega_2, \Ysl_{\sigma_1}] \bigr] = 0$, and therefore
	$[\Ysl_{\omega_2}, \Ysl_{\sigma_1}] = 0$.
\end{plem}

\begin{proof}
	The first assertion is clear. Since $\Ysl_{\sigma_1} \colon G_{i,k} \to
	G_{i-1,k-2} \decal{-2}$, this is saying that $\Ysl_{\sigma_1}$ is primitive of
	weight $-1$ (with respect to the $\sltwo$-representation determined by
	$\omega_2$), which gives the second assertion.
\end{proof}

\newpar
We can now show that the two $\sltwo$-representations determined by $\omega_2$ and
$\sigma_1$ can be combined into a single representation of the larger Lie algebra
$\slthree$. Let us briefly review its structure. The Lie algebra $\slthree$ is
associated to the Dynkin diagram of type $\mathsf{A}_2$, and so it has two simple
roots, and the resulting Cartan matrix is
\[
	\begin{pmatrix} a_{1,1} & a_{1,2} \\ a_{2,1} & a_{2,2} \end{pmatrix}
	= \begin{pmatrix} 2 & -1 \\ -1 & 2 \end{pmatrix}.
\]
According to a theorem by Serre, $\slthree$ is therefore generated as a Lie algebra
by six elements $\esl_1, \fsl_1, \hsl_1$, $\esl_2, \fsl_2, \hsl_2$, subject to the
following relations:
\begin{enumerate}
	\item $[\hsl_i,\hsl_j] = 0$, $[\hsl_i, \esl_j] = a_{i,j} \esl_j$, and $[\hsl_i, \fsl_j] = -a_{i,j} \fsl_j$
	\item $[\esl_i, \fsl_j] = \delta_{i,j} \hsl_j$, where $\delta_{i,j} = 1$ if $i=j$, and $0$
		otherwise
	\item $(\ad \esl_i)^{1-a_{i,j}} \esl_j = 0$ and $(\ad \fsl_i)^{1-a_{i,j}} \fsl_j = 0$.
\end{enumerate}

\newpar
The point is that the operators $\omega_2$ and $\sigma_1$ satisfy the Serre relations
for $\slthree$. 

\begin{pprop}
	If we define $\esl_1 = \omega_2$, $\fsl_1 = \Ysl_{\omega_2}$, $\esl_2 = \Ysl_{\sigma_1}$,
	and $\fsl_2 = \sigma_1$, and let $\hsl_1$ and $\hsl_2$ act on
	the complex $G_{i,k}$ as multiplication by $i$ respectively $-k$, then these six
	operators satisfy the Serre relations for the Lie algebra $\slthree$.
\end{pprop}

\begin{proof}
	The relations on the first line hold because $\omega_2$ and $\sigma_1$ each
	determine a representation of the Lie algebra $\sltwo$, and because $\omega_2
	\colon G_{i,k} \to G_{i+2,k+1} \decal{2}$ and $\sigma_1 \colon G_{i,k} \to
	G_{i+1,k+2} \decal{2}$. The relations on the second line hold because $[\esl_1, \fsl_2]
	= [\omega_2, \sigma_1] = 0$, and because $[\fsl_1, \esl_2] = [\Ysl_{\omega_2},
	\Ysl_{\sigma_1}] = 0$ by \Cref{lem:relation}. The relations on the third line
	now follow from the finite-dimensional representation theory of $\sltwo$: for
	example, $[\hsl_1, \esl_2] = -\esl_2$ and $[\fsl_1, \esl_2] = 0$ are saying that
	$\esl_2$ is primitive of weight $-1$ with respect to the $\sltwo$-representation
	$\esl_1, \fsl_1, \hsl_1$, and consequently $(\ad \esl_1)^2 \esl_2 = 0$.
\end{proof}

\newpar
Because the operators $\omega_2$ and $\sigma_1$ each act with a shift, we need to be a
little bit careful if we want to say exactly which object in the derived category the
Lie algebra $\slthree$ acts on. There are two possible answers. One possibility is to
take the direct sum
\[
	\bigoplus_{i,k,\ell \in \ZZ} G_{i,k} \decal{\ell}.
\]
All six operators $\esl_1, \fsl_1, \hsl_1$ and $\esl_2, \fsl_2, \hsl_2$ clearly act on this object; the
disadvantage is that there are infinitely many nonzero terms. The other
possibility is to look at the direct sum
\[
	\bigoplus_{i,k = -n}^n G_{i,k} \Bigl[ \lfloor \tfrac{2}{3}(i+k) \rfloor \Bigr].
\]
This only has finitely many nonzero terms, and all six operators act on it; the only
disadvantage is that the extra shift -- by the integer part of $\tfrac{2}{3}(i+k)$ --
seems somewhat artificial.

\newpar
The proof above depends on knowing the topmost component of the operator $\Ysigma$.
The following lemma, whose proof uses the symplectic relative Hard Lefschetz theorem
(in \Cref{thm:HL-sigma}), shows that this is exactly $\Ysl_{\sigma_1}$.

\begin{plem} \label{lem:Ysigma}
	The difference $\Ysigma - \Ysl_{\sigma_1}$ maps $G_{i,k}$ into the sum of the
	$G_{i+j,k-2} \decal{-2}$ with $j \leq -2$.
\end{plem}

\begin{proof}
	To simplify the notation, let us again denote the operators in the
	$\sltwo$-representation determined by $\sigma_1$ by the letters $\Xsl_2 =
	\sigma_1$, $\Ysl_2$, and $\Hsl_2$. Let us also set $\tau = \Ysigma$, and expand
	this according to
	its degree in $i$ as $\tau = \sum_j \tau_j$; the individual components are operators
	\[
		\tau_j \colon G_{i,k} \to G_{i+j,k-2} \decal{-2}.
	\]
	In these terms, the lemma is asserting that $\tau_j = 0$ for $j \geq 0$, and that
	$\tau_{-1} = \Ysl_2$. Let $j$ be the largest integer such that $\tau_j \neq 0$. If
	$j \geq 0$, then after expanding the relation
	\[
		\Hsl_2 = [\sigma, \tau] = [\Xsl_2 + \sigma_0 + \dotsb, \tau_j + \tau_{j-1} + 
		\dotsb]
	\]
	by degree, we would get $[\Xsl_2, \tau_j] = 0$. But this is impossible because
	$\tau_j$ has weight $-2$ with respect to $\ad \Hsl_2$. The same reasoning shows that
	$\Hsl_2 = [\Xsl_2, \tau_{-1}]$, and as $\Ysl_2$ is uniquely determined by $\Hsl_2$ and
	$\Xsl_2$, it follows that $\tau_{-1} = \Ysl_2$. Consequently, $\tau = \Ysl_2 +
	\tau_{-2} + \dotsb$.
\end{proof}

\newpar
The relations for the Lie algebra $\slthree$ can also be interpreted nicely in terms
of the reflections given by the two Weil elements $\wsl_1$ and $\wsl_2$.

\begin{plem} \label{lem:reflection}
	We have $(\Ad \wsl_1 \circ \Ad \wsl_2 \circ \Ad \wsl_1)(\omega_2) = \sigma_1$.
\end{plem}

\begin{proof}
	Recall from \Cref{par:XYH} that $\Xsl_1 = \omega_2$ and that $\Xsl_2 = \sigma_1$.
	The definition of
	the Weil element shows that $\Ad \wsl_1(\Xsl_1) = \wsl_1 \Xsl_1 \wsl_1^{-1} =
	-\Ysl_1$. The relations $[\Ysl_2, \Ysl_1] = 0$ and $[\Hsl_2, \Ysl_1] = -\Ysl_1$ are
	saying that $\Ysl_1$ is primitive of
	weight $-1$ with respect to the $\sltwo$-representation $\Xsl_2, \Ysl_2,
	\Hsl_2$, and therefore $\Ad \wsl_2(-\Ysl_1) = [\Xsl_2, -\Ysl_1] = [\Ysl_1, \Xsl_2]$.
	Because $(\ad \Ysl_1)^2 \Xsl_2 = 0$, this is now primitive of weight $-1$ with
	respect to the $\sltwo$-representation $\Xsl_1, \Ysl_1, \Hsl_1$, and so finally 
	\[
		(\Ad \wsl_1 \circ \Ad \wsl_2 \circ \Ad \wsl_1)(\Xsl_1) 
		= \Ad \wsl_1 \bigl( [\Ysl_1, \Xsl_2] \bigr) = \bigl[ \Xsl_1, [\Ysl_1, \Xsl_2]
		\bigr] = \Xsl_2.
	\]
	This proves the assertion.
\end{proof}

\section{The compact case}
\label{chap:slfour}

\newpar
We end this paper with a short analysis of the compact case. Most notably, we
prove that the action by the Lie algebra $\slthree$ can be upgraded, in the case
where $M$ and $B$ are compact, to an action by the Lie algebra $\slfour$. Since
$\slfour \cong \sosix$, this gives an alternative proof (not relying on the existence
of a hyperk\"ahler metric) for a result by Looijenga-Lunts \cite[\S4]{LL} and
Verbitsky \cite{Verbitsky}: the cohomology of an irreducible compact hyperk\"ahler
manifold with a Lagrangian fibration carries an action by $\sosix$. 

\newpar
We assume from now on that $M$ and $B$ are both compact; in other words, $M$ is a
holomorphic symplectic compact K\"ahler manifold of dimension $2n$, and the base $B$
of our Lagrangian fibration $\pi \colon M \to B$ is a compact
K\"ahler manifold of dimension $n$. In this case, the holomorphic symplectic form
$\sigma \in H^0(M, \Omega_M^2)$ is automatically closed. If $M$ is simply connected,
then it has a hyperk\"ahler metric (by Yau's theorem) and $B$ is a product of
projective spaces (by Hwang's theorem), but we are not assuming that this is the case. 

\newpar
Let $\lambda \in A^{1,1}(B)$ be a K\"ahler class on $B$. Recall that we have
\[
	\derR \pil \QQ_M(n) \decal{2n} \cong \bigoplus_{i=-n}^n P_i \decal{-i},
\]
and that each $P_i$ is a polarizable Hodge module of weight $i$ on $B$. Since $B$ is
a compact K\"ahler manifold, the cohomology groups $H^j(B, P_i)$ therefore have
Hodge structures of weight $i+j$, and we have an isomorphism of Hodge
structures
\[
	H^{2n+j}(M, \QQ)(n) \cong \bigoplus_{i=-n}^n H^{j-i}(B, P_i).
\]
By the Hard Lefschetz theorem (for polarizable Hodge modules on compact K\"ahler
manifolds), cup product with $\lambda$ determines, for each $j \geq 1$, an isomorphism
\[
	\lambda^j \colon H^{-j}(B, P_i) \to H^j(B, P_i)(j)
\]
of Hodge structures of weight $i-j$.

\newpar \label{par:action-forms}
Since $\lambda$ is closed, its pullback $\piu \lambda \in A^{1,1}(M)$ acts on the
complex $\Cpi$ from \Cref{par:Laumon}. Let us check that, with respect to the isomorphism 
\[
	\Cpi \cong \bigoplus_{i=-n}^{n} (\Pmod_i, F_{\bullet} \Pmod_i) \decal{-i}
\]
from \eqref{eq:Cpi-decomposition}, the action by $\piu \lambda$ is diagonal. With our
usual notation, this amounts to saying that $\piu \lambda = (\piu \lambda)_0$, and
that the morphism
\[
	(\piu \lambda)_0 \colon (\Pmod_i, F_{\bullet} \Pmod_i) \to (\Pmod_i, F_{\bullet-1}
	\Pmod_i) \decal{2}
\]
is the one induced by the K\"ahler form $\lambda \in A^{1,1}(B)$. This is a
consequence of the projection formula: if we view $\lambda$ as a morphism $\lambda
\colon \CC_B \to \CC_B \decal{2}$ in the derived category of constructible sheaves,
then the claim is that 
\[
	\derR \pil \left( \CC_M \xrightarrow{\piu \lambda} \CC_M \decal{2} \right)
		\cong \left( \CC_B \xrightarrow{\lambda} \CC_B \decal{2} \right) 
		\tensor \derR \pil \CC_M.
\]
Using the de Rham resolutions of $\CC_M$ and $\CC_B$ by differential forms, the
projection formula in this case amounts to the following statement.

\begin{plem} \label{lem:projection}
	The morphism of complexes
	\[
		(\shA_B^{\bullet}, d) \tensor (\pil \shA_M^{\bullet}, d) \to 
		(\pil \shA_M^{\bullet}, d), \quad
		\alpha \tensor \beta \mapsto \piu \alpha \wedge \beta,
	\]
	is a quasi-isomorphism.
\end{plem}

\begin{proof}
	Since we are using Deligne's sign conventions, the tensor product has terms
	\[
		\bigoplus_{i+j=k} \shA_B^i \tensor \pil \shA_M^j
	\]
	and differential $d(\alpha \tensor \beta) = d \alpha \tensor \beta + (-1)^i \alpha
	\tensor d \beta$. The morphism of complexes is
	\[
		\bigoplus_{i+j=k} \shA_B^i \tensor \pil \shA_M^j \to \pil \shA_M^k,
		\quad \sum_{i+j=k} \alpha_i \tensor \beta_j \mapsto \sum_{i+j=k} \piu \alpha_i
		\wedge \beta_j.
	\]
	This is clearly surjective: if $U \subseteq B$ is open,  then the element $1
	\tensor \beta \in A^0(U) \tensor A^k \bigl( \pi^{-1}(U) \bigr)$ is a preimage
	for $\beta \in A^k \bigl( \pi^{-1}(U) \bigr)$. To show that the morphism is a
	quasi-isomorphism, one then argues locally with the help of the Poincar\'e lemma,
	which says that $\shA_B^{\bullet}$ is a resolution of the constant sheaf $\CC_B$.
\end{proof}

\newpar
The Hard Lefschetz theorem for the K\"ahler form $\lambda$ gives us a third 
action by the Lie algebra $\sltwo$. In the rest of this chapter, we shall
argue that all three actions can be combined into an action of the Lie algebra
$\slfour$. Consider the vector spaces
\begin{equation} \label{eq:Hijk}
	H^{i,j,k} = H^{i+j}(B, G_{i,k}) = H^j \bigl( B, \gr_{-k}^F \DR(\Pmod_i) \bigr) 
\end{equation}
for $i,j,k \in \ZZ$. These vector spaces refine the Hodge decomposition on $M$,
taking into account the decomposition theorem for $\pi \colon M \to B$. Indeed,
according to the direct image theorem for the cohomology of the polarizable Hodge
module $P_i$ on the compact K\"ahler manifold $B$, the $\QQ$-vector space $H^j(B,
P_i)$ carries a Hodge structure of weight $i+j$, and
\[
	\gr_{-k}^F \bigl( H^j(B, P_i) \tensor_{\QQ} \CC \bigr)
	\cong \gr_{-k}^F H^j \bigl( B, \DR(\Pmod_i) \bigr)
	\cong H^j \bigl( B, \gr_{-k}^F \DR(\Pmod_i) \bigr).
\]
This leads to the following isomorphism with the Hodge decomposition on $H^j(B,
P_i)$:
\begin{equation} \label{eq:Hijk-Hodge}
	H^{i,j,k} \cong \bigl( H^j(B, P_i) \tensor_{\QQ} \CC \bigr)^{k, i+j-k}
\end{equation}
In order to relate this to the Hodge decomposition on $M$, we use the isomorphism 
\[
	\derR \pil \Omega_M^{n+k} \decal{n-k} \cong \bigoplus_{i=-n}^n G_{i,k}
\]
from \eqref{eq:pil-Omega}. Substituting in the definition of $H^{i,j,k}$, we see 
immediately that
\begin{equation} \label{eq:Hodge-numbers}
	H^{n+k, n+j}(M) \cong H^{n+j}(M, \Omega_M^{n+k}) 
	\cong \bigoplus_{i=-n}^n H^{j+k}(B, G_{i,k})
	= \bigoplus_{i=-n}^n H^{i, j+k-i, k}.
\end{equation}
The vector spaces $H^{i,j,k}$ also interact with duality in the expected way. Indeed,
because of \eqref{eq:duality}, the Grothendieck dual of $G_{i,k}$ is isomorphic to
$G_{-i,-k}$, and so 
\begin{equation} \label{eq:Hijk-duality}
	\Hom_{\CC} \bigl( H^{i,j,k}, \CC \bigr) \cong H^{-i,-j,-k}.
\end{equation}

\newpar
The discussion in the preceding paragraph leads to the following concrete
interpretation for the three indices $i,j,k$, similar to what happens for the
complexes $G_{i,k}$:
\begin{enumerate}
	\item The first index $i$ records the cohomological degree along the fibers of
		$\pi$, in the sense that $H^{i,j,k}$ is associated with the
		$(n+i)$-th cohomology groups of the fibers.
	\item The second index $j$ records the cohomological degree along the base $B$, in
		the sense that $H^{i,j,k}$ is associated with the $(n+j)$-th cohomology groups
		on $B$.
	\item The third index $k$ records the holomorphic degree, in the sense that
		$H^{i,j,k}$ is associated with the sheaf $\Omega_M^{n+k}$ of holomorphic forms
		of degree $(n+k)$. 
\end{enumerate}

\newpar
The K\"ahler form $\omega \in A^{1,1}(M)$ induces a morphism
\[
	\omega_2 \colon H^{i,j,k} = H^{i+j}(B, G_{i,k}) \to H^{i+j+2}(B, G_{i+2,k+1}) =
	H^{i+2, j, k+1},
\]
and the relative Hard Lefschetz theorem for $\omega$ becomes an isomorphism
\begin{equation} \label{eq:HL-omega-ijk}
	\omega_2^i \colon H^{-i,j,k} \to H^{i,j,k+i} \quad \text{for $i \geq 1$.}
\end{equation}
Similarly, the holomorphic symplectic form $\sigma \in H^0(M, \Omega_M^2)$ induces a
morphism
\[
	\sigma_1 \colon H^{i,j,k} = H^{i+j}(B, G_{i,k}) \to H^{i+j+2}(B, G_{i+1,k+2})
	= H^{i+1, j+1, k+2},
\]
and the symplectic relative Hard Lefschetz theorem (\Cref{thm:HL-sigma}) becomes
an isomorphism
\begin{equation} \label{eq:HL-sigma-ijk}
	\sigma_1^k \colon H^{i,j,-k} \to H^{i+k,j+k,k} \quad \text{for $k \geq 1$.}
\end{equation}
Finally, by \eqref{eq:Hijk-Hodge}, the K\"ahler form $\lambda \in A^{1,1}(B)$
induces a morphism
\[
	\lambda \colon H^{i,j,k} \cong \bigl( H^j(B, P_i) \tensor_{\QQ} \CC
		\bigr)^{k,i+j-k} \to 
		\bigl( H^{j+2}(B, P_i) \tensor_{\QQ} \CC \bigr)^{k+1, i+j-k+1}
		\cong H^{i,j+2,k+1},
\]
and the Hard Lefschetz theorem for $\lambda$ becomes an isomorphism
\begin{equation} \label{eq:HL-lambda-ijk}
	\lambda^j \colon H^{i,-j,k} \to H^{i,j,k+j} \quad \text{for $j \geq 1$.}
\end{equation}

\newpar
From these isomorphisms and \Cref{lem:amplitude}, one can deduce the following result.

\begin{plem}
	We have $H^{i,j,k} = 0$ unless $\max \bigl( \abs{i}, \abs{j}, \abs{k}, \abs{i-k},
	\abs{j-k}, \abs{i+j-k} \bigr) \leq n$.
\end{plem}

We may arrange the $H^{i,j,k}$ on a three-dimensional grid, by putting the
vector space for the multi-index $(i,j,k)$ at the point with coordinates $\half \bigl(
2k-i-j, i \sqrt{3}, j \sqrt{3} \bigr)$. In geometric terms, the conditions in the lemma 
are describing a \emph{rhombic dodecahedron}:
\begin{center}
	\includegraphics[width=5cm]{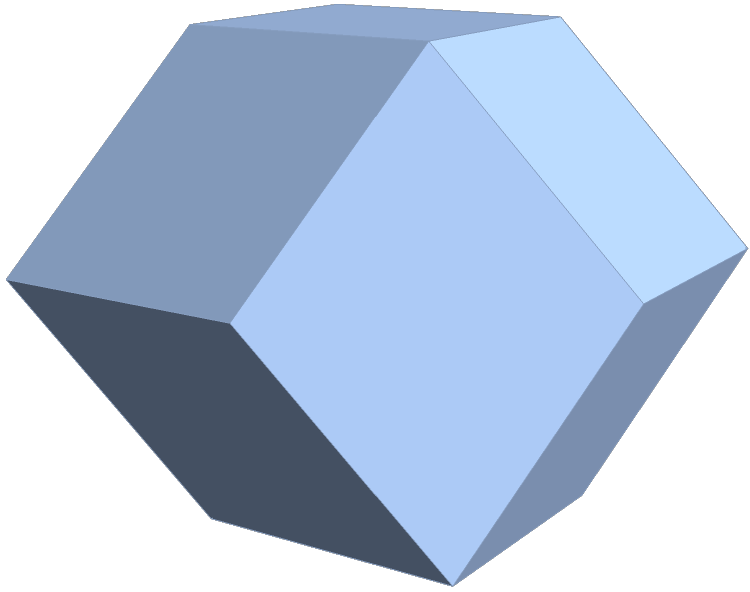}
\end{center}
The lemma is asserting that all the points corresponding to nonzero
$H^{i,j,k}$ lie inside this rhombic dodecahedron. The example of a projection $\pi
\colon A \times B \to B$ from the product of two $n$-dimensional abelian varieties
shows that this is the best one can expect in general. For irreducible compact
hyperk\"ahler manifolds, Nagai's conjecture would imply that the convex hull of the
points corresponding to nonzero $H^{i,j,k}$ is actually an octahedron
\cite[\S3.8]{LF}.

\newpar
The three isomorphisms in \eqref{eq:HL-omega-ijk}, \eqref{eq:HL-sigma-ijk}, and
\eqref{eq:HL-lambda-ijk} give us three reflections 
\[
	(i,j,k) \to (-i,j,k-i), \quad 
	(i,j,k) \to (i,-j,k-j), \quad
	(i,j,k) \to (i-k,j-k,-k),
\]
and a little bit of calculation shows that these three reflections together generate
a total of $24$ symmetries, with a group structure isomorphic to the symmetric group
$S_4$. This is not surprising, because $S_4$ is the Weyl group of $\slfour$.
One of the symmetries is
\[
	(i,j,k) \to (j,i,i+j-k),
\]
and because of \eqref{eq:Hodge-numbers}, it induces an isomorphism $H^{p,q}(M) \cong
H^{q,p}(M)$. It would be interesting if one could explain this particular
symmetry in a geometric way.

\newpar
As in the previous chapter, we need to establish one additional identity in order to
have an action by the Lie algebra $\slfour$. Once again, the isomorphism $\shT_M
\cong \Omega_M^1$ associates to the cohomology class of $\piu \lambda \in H^1(M,
\Omega_M^1)$ a new element
\[
	i(\piu \lambda) \in A^{0,1}(M, \shT_M),
\]
that is $\dbar$-closed and satisfies $i(\piu \lambda) \cont \sigma = \piu \lambda$.
The following lemma is also proved in the same way as \Cref{lem:contraction}.

\begin{plem}
	For every $k,q \in \ZZ$, the following diagram commutes:
	\[
		\begin{tikzcd}
			A^{n-k,q}(M) \dar{(\piu \lambda) \wedge} \rar{\wsigma} & A^{n+k,q}(M)
			\dar{i(\piu \lambda) \cont} \\
			A^{n-k+1,q+1}(M) \rar{\wsigma} & A^{n+k-1,q+1}(M)
		\end{tikzcd}
	\]
\end{plem}

\newpar
The crucial observation is that the commutator of the two operators
$\alpha \mapsto (\piu \lambda) \wedge \alpha$ and $\alpha \mapsto i(\omega) \cont
\alpha$ vanishes.

\begin{plem}
	We have $\bigl[ (\piu \lambda) \wedge, i(\piu \lambda) \cont \bigr] = 0$.
\end{plem}

\begin{proof}
	As before, the commutator is the wedge product with the $\dbar$-closed $(0,2)$-form
	\[
		-i(\piu \lambda) \cont (\piu \lambda) \in A^{0,2}(M).
	\]
	In a nutshell, this vanishes because the vector fields in $i(\piu \lambda)$ are
	tangent to the fibers of $\pi$. It is enough to prove this locally, and so we may
	assume without loss of generality that
	\[
		\lambda = \sum_{j=1}^n \dt_j \wedge \theta_j,
	\]
	where $t_1, \dotsc, t_n$ are local coordinates on $B$ and $\theta_1, \dotsc,
	\theta_n$ are $\dbar$-closed $(0,1)$-forms. Let $\eta_j$ be the unique holomorphic
	vector field such that $\piu(\dt_j) = \eta_j \cont \sigma$. Then
	\[
		i(\piu \lambda) = \sum_{j=1}^n \eta_j \tensor \piu \theta_j,
	\]
	and consequently
	\[
		-i(\piu \lambda) \cont (\piu \lambda)
		= \sum_{j,k=1}^n \bigl( \eta_j \cont (\piu \dt_k) \bigr) \cdot \piu(\theta_j \wedge
		\theta_k) = 0,
	\]
	due to the fact that the vector fields $\eta_j$ are tangent to the fibers of $\pi$.
\end{proof}

\newpar
We now consider how $i(\piu \omega)$ acts on the complex $(M, d)$, where $M_k^i =
\pil \shA_M^{n+k,n+i}$ and $d = (-1)^k \dbar$. Contraction against $i(\piu \lambda)$
defines a morphism.
\[
	i(\piu \lambda) \colon G \to G(-1) \decal{1} 
\]
in the derived category; note that this morphism changes the grading. Since $\piu
\lambda$ and $\sigma$ commute (as $2$-forms on $M$), we have $i(\piu \lambda) =
\wsigma \cdot (\piu \lambda) \cdot \wsigma^{-1} = -[\Ysigma, \piu \lambda]$,
and therefore 
\[
	\bigl[ \piu \lambda, [\piu \lambda, \Ysigma] \bigr] 
	= -[i(\piu \lambda), \piu \lambda] = 0.
\]
If we now take the topmost component, and remember that the action by $(\piu
\lambda)_0$ is diagonal and equal to $\lambda$, we arrive at the following result.

\begin{plem} \label{lem:relation-slfour}
	We have $\bigl[ \lambda, [\lambda, \Ysl_{\sigma_1}] \bigr] = 0$, and therefore
	$[\Ysl_{\lambda}, \Ysl_{\sigma_1}] = 0$.
\end{plem}

\newpar
We can finally prove that the three $\sltwo$-representations determined by $\omega_2$,
$\sigma_1$, and $\lambda$ can be combined into a single representation of $\slfour$.
Recall that the Lie algebra $\slfour$ is associated to the Dynkin diagram of type
$\mathsf{A}_3$, and so it has three simple roots, and the resulting Cartan matrix is
\[
	\begin{pmatrix} a_{1,1} & a_{1,2} & a_{1,3} \\ a_{2,1} & a_{2,2} & a_{2,3} \\
	a_{3,1} & a_{3,2} & a_{3,3} \end{pmatrix}
	= \begin{pmatrix} 2 & -1 & 0 \\ -1 & 2 & -1 \\ 0 & -1 & 2 \end{pmatrix}.
\]
By Serre's theorem, $\slfour$ is generated as a Lie algebra
by nine elements $\esl_1, \fsl_1, \hsl_1$, $\esl_2, \fsl_2, \hsl_2$, $\esl_3,
\fsl_3, \hsl_3$, subject to the following relations (which have the same shape as
before):
\begin{enumerate}
	\item $[\hsl_i,\hsl_j] = 0$, $[\hsl_i, \esl_j] = a_{i,j} \esl_j$, and $[\hsl_i, \fsl_j] = -a_{i,j} \fsl_j$
	\item $[\esl_i, \fsl_j] = \delta_{i,j} \hsl_j$, where $\delta_{i,j} = 1$ if $i=j$, and $0$
		otherwise
	\item $(\ad \esl_i)^{1-a_{i,j}} \esl_j = 0$ and $(\ad \fsl_i)^{1-a_{i,j}} \fsl_j = 0$.
\end{enumerate}

\newpar
The proof of the Serre relations is very similar to what we did for $\slthree$, and
so we will only sketch the argument.

\begin{pprop}
	If we define $\esl_1 = \omega_2$, $\fsl_1 = \Ysl_{\omega_2}$, $\esl_2 = \Ysl_{\sigma_1}$,
	$\fsl_2 = \sigma_1$, $\esl_3 = \lambda$, and $\fsl_3 = \Ysl_{\lambda}$, and let
	$\hsl_1$, $\hsl_2$ and $\hsl_3$ act on the vector space $H^{i,j,k}$ respectively
	as multiplication by $i$, $-k$ and $j$, then these nine
	operators satisfy the Serre relations for the Lie algebra $\slfour$.
\end{pprop}

\begin{proof}
	Since $\omega$, $\sigma$, and $\piu \lambda$ commute (as $2$-forms on $M$), it is
	easy to see that $[\esl_1, \esl_3] = [\omega_2, \lambda] = 0$ and $[\fsl_2,
	\esl_3] = [\sigma_1, \lambda] = 0$. Because $[\hsl_1, \esl_3] = 0$, it follows
	that $[\fsl_1, \esl_3] = 0$, and so the two Lie algebras generated by $\esl_1,
	\fsl_1, \hsl_1$ and $\esl_3, \fsl_3, \hsl_3$ commute; this gives about half of the
	necessary relations. Among the remaining new relations, the only nontrivial ones
	are $(\ad \esl_3)^2 \esl_2 = 0$ and $[\fsl_3, \esl_2] = 0$, and these are of
	course contained in \Cref{lem:relation-slfour}.
\end{proof}


\section{Acknowledgements}

\newpar
I got interested in Lagrangian fibrations during a workshop about hyperk\"ahler
manifolds at the Simons Center in January of 2023, and I am grateful to Ljudmila
Kamenova, Giovanni Mongardi, and Alexei Oblomkov for organizing it. During the
workshop, Junliang Shen gave a series of lectures, and I thank him for explaining the
conjectures, and also for some very useful discussions afterwards. I thank Davesh
Maulik for inviting me to MIT to talk about some details of the proof. I thank Chris
Brav for a pleasant afternoon chatting about the categorical aspects of the BGG
correspondence. Finally, I am grateful to Daniel Huybrechts and Mirko Mauri for their
detailed comments about a first draft of the paper, as well as for their survey paper
\cite{LF} about Lagrangian fibrations, from which I learned a lot of the general
theory.


\phantomsection
\addcontentsline{toc}{chapter}{References}

\providecommand{\bysame}{\leavevmode\hbox to3em{\hrulefill}\thinspace}
\providecommand{\MR}{\relax\ifhmode\unskip\space\fi MR }
\providecommand{\MRhref}[2]{%
  \href{http://www.ams.org/mathscinet-getitem?mr=#1}{#2}
}
\providecommand{\href}[2]{#2}

\end{document}